\documentclass[12pt,xcolor=dvipsnames]{amsart}
\usepackage[colorlinks=true]{hyperref}
\usepackage[utf8]{inputenc}
\usepackage{amsmath,amsthm,amssymb,url,mathdots,verbatim,psfrag,color,mathtools}
\usepackage[dvipsnames]{xcolor}
 \usepackage[margin = 2.5cm]{geometry}

\usepackage{tikz}
\usetikzlibrary{calc,decorations}
\usepackage[outline]{contour}
\contourlength{2pt}
\usepackage{ytableau}
\usepackage{comment}
\usepackage{MnSymbol}
\usepackage{eurosym}
\usepackage{graphicx}
\usepackage{scalerel}
\usepackage[inline]{enumitem}
\usepackage[normalem]{ulem}
\pgfdeclaredecoration{arrows}{draw}{
	\state{draw}[width=\pgfdecoratedinputsegmentlength]{%
		\path [every arrow subpath/.try] \pgfextra{%
			\pgfpathmoveto{\pgfpointdecoratedinputsegmentfirst}%
			\pgfpathlineto{\pgfpointdecoratedinputsegmentlast}%
		};
}}
\tikzset{every arrow subpath/.style={->, draw, ultra thick}}

\newtheorem{theorem}{Theorem}

\newtheorem{prop}[theorem]{Proposition}
\newtheorem{cor}[theorem]{Corollary}
\newtheorem{lemma}[theorem]{Lemma}
\newtheorem{definition}[theorem]{Definition}

\theoremstyle{definition}

\newtheorem{example}[theorem]{Example}
\newtheorem{remark}[theorem]{Remark}

\numberwithin{equation}{section}
\numberwithin{theorem}{section}

\renewcommand{\u}[1]{\underline{#1}}

\DeclareMathOperator{\sgn}{sgn}

\newcommand{\nenwarrow}{\nwarrow \!\!\!\!\!\;\!\! \nearrow}
\usepackage{enumitem}
\newcommand{\si}[2]{\u{[#1,#2]}}

\newcommand{\e}{{\operatorname{E}}}
\newcommand{\W}{{\operatorname{W}}}

\newcommand{\dec}{\operatorname{decor}}
\newcommand{\asym}{\mathbf{ASym}}
\newcommand{\sym}{\mathbf{Sym}}

\definecolor{dg}{rgb}{0.0, 0.2, 0.13}

\definecolor{verylight-gray}{gray}{0.95}
\definecolor{light-gray}{gray}{0.9}
  
\colorlet{mygreen}{Green}
\colorlet{myyellow}{yellow}
\colorlet{myblue}{NavyBlue}
\colorlet{myred}{BrickRed}

\title[Littlewood identity]{Bounded Littlewood identity related to alternating sign matrices}

\author{Ilse Fischer}
\thanks{The author acknowledge the financial support from the Austrian Science Foundation FWF, grant P34931}

\begin{document}

\maketitle 

\begin{abstract}
An identity that is reminiscent of the Littlewood identity plays a fundamental role in recent proofs of the facts that alternating sign triangles are equinumerous with totally symmetric self-complementary plane partitions and that alternating sign trapezoids are equinumerous with holey cyclically symmetric lozenge tilings of a hexagon. We establish a bounded version of a generalization of this identity. Further, we provide combinatorial interpretations of both sides of the identity. The ultimate goal would be to construct a combinatorial proof of this identity (possibly via an appropriate variant of the Robinson-Schensted-Knuth correspondence) and its unbounded version as this would improve the understanding of the relation between alternating sign trapezoids and plane partition objects.
\end{abstract}

\section{Introduction}

Littlewood's identity reads as 
\begin{equation}
\label{littlewood}  
\sum_{\lambda} s_{\lambda}(X_1,\ldots,X_n) = \prod_{i=1}^{n} \frac{1}{1-X_i} \prod_{1 \le i < j \le n} \frac{1}{1-X_i X_j}, 
\end{equation} 
where $s_{\lambda}(X_1,\ldots,X_n)$ denotes the Schur polynomial associated with the partition $\lambda$ and the sum is over all partitions $\lambda$. In fact, the identity was already known to Schur, see \cite[p.\ 163]{schur1} or \cite[p.\ 456]{schur2},  and written down by Littlewood in \cite[p.\ 238]{Littlewood}.
This identity has a beautiful combinatorial proof that is based on the Robinson-Schensted-Knuth correspondence and exploits its symmetry, see Appendix~\ref{combclassical}, and e.g., \cite{Sta99} for details.

In recent papers \cite{Fis19b,Fis19a,fourfold}, where ``alternating sign matrix objects'' (namely, \emph{alternating sign triangles} and \emph{alternating sign trapezoids}) have been connected to certain ``plane partition objects'' (namely, \emph{totally symmetric self-complementary plane partitions} and
\emph{column strict shifted plane partitions of fixed class}, which generalize the better known \emph{descending plane partitions}), a very similar identity played the crucial role to establish this still mysterious \cite{cube} connection. All these proofs are not of a combinatorial nature and involve rather complicated calculations, and so the study of the combinatorics of our Littlewood-type identity is very likely lead to a better understanding of the combinatorics of this relation.

In order to formulate the identity, we rewrite \eqref{littlewood} using the bialternant formula for the Schur polynomial \cite[7.15.1]{Sta99}
$$
s_{(\lambda_1,\ldots,\lambda_n)}(X_1,\ldots,X_n) = \frac{\det_{1 \le i,j \le n} \left( X_i^{\lambda_j+n-j} \right)}
{\prod_{1 \le i < j \le n} (X_i-X_j)} = \frac{ \asym_{X_1,\ldots,X_n} \left[ \prod_{i=1}^n X_i^{\lambda_i+n-i} \right] }
{\prod_{1 \le i < j \le n} (X_i-X_j)},
$$
allowing zeros at the end section of $(\lambda_1,\ldots,\lambda_n)$, with
$$\asym_{X_1,\ldots,X_n} f(X_1,\ldots,X_n) = \sum_{\sigma \in {\mathcal S}_n} \sgn \sigma \cdot f(X_{\sigma(1)},\ldots,X_{\sigma(n)})$$ 
as follows.    
$$
\frac{\asym_{X_1,\ldots,X_n} \left[ \sum_{0 \le k_1 < k_2 < \ldots < k_n} X_1^{k_1} X_2^{k_2} \cdots X_n^{k_n} \right]}{\prod_{1 \le i < j \le n} (X_j-X_i)} = 
\prod_{i=1}^{n} \frac{1}{1-X_i} \prod_{1 \le i < j \le n} \frac{1}{1-X_i X_j}
$$
We have used the following identity in \cite{Fis19b,Fis19a}. There it is proved by induction with respect to $n$.
\begin{multline} 
\label{littlewoodASM1} 
\frac{\asym_{X_1,\ldots,X_n} \left[ \prod_{1 \le i < j \le n} (1+X_j + X_i X_j) \sum_{0 \le k_1 < k_2 < \ldots < k_n} X_1^{k_1} X_2^{k_2} \cdots X_n^{k_n} \right]}{\prod_{1 \le i <  j \le n} (X_j-X_i)} \\ = 
\prod_{i=1}^{n} \frac{1}{1-X_i} \prod_{1 \le i < j \le n}  \frac{1+X_i + X_j}{1-X_i X_j}
\end{multline} 
In \cite{fourfold}, an additional parameter has been introduced, which has to be set to $1$ to obtain \eqref{littlewoodASM1}.
\begin{multline} 
\label{Hans}
\frac{\asym_{X_1,\ldots,X_n} 
\left[\prod_{1 \le i < j \le n} (Q+(Q-1) X_i + X_j + X_i X_j) 
\sum_{0 \le k_1 < k_2 < \ldots < k_n} \prod_{i=1}^n \left( \frac{X_i(1+X_i)}{Q+X_i} \right)^{k_i}  
\right]}{\prod_{1 \le i < j \le n} (X_j-X_i)}  \\ = \prod_{i=1}^n \frac{Q+X_i}{Q-X_i^2} \frac{ \prod_{1 \le i < j \le n} (Q(1+X_i)(1+X_j)- X_i X_j )}{\prod\limits_{1 \le i < j \le n} (Q-X_i X_j)}.
\end{multline}
Among other things, we will see in this paper that we can also introduce another parameter in \eqref{littlewoodASM1} as follows: 
\begin{multline} 
\label{littlewoodASM2} 
\frac{\asym_{X_1,\ldots,X_n} \left[ \prod_{1 \le i < j \le n} (1+ w X_i + X_j + X_i X_j) \sum_{0 \le k_1 < k_2 < \ldots < k_n} X_1^{k_1} X_2^{k_2} \cdots X_n^{k_n} \right]}{\prod_{1 \le i <  j \le n} (X_j-X_i)} \\ = 
\prod_{i=1}^{n} \frac{1}{1-X_i} \prod_{1 \le i < j \le n}  \frac{1+X_i + X_j + w X_i X_j}{1-X_i X_j},
\end{multline} 
in fact, there is even the following common generalization of \eqref{Hans} and \eqref{littlewoodASM2}. 
\begin{multline} 
\label{mostgeneralunbounded}
\frac{\asym_{X_1,\ldots,X_n}  \left[ \prod_{1 \le i < j \le n} (Q+(Q+r) X_i + X_j + X_i X_j) 
\sum_{0 \le k_1 < k_2 < \ldots < k_n} \prod_{i=1}^n \left( \frac{X_i(1+X_i)}{Q+X_i} \right)^{k_i} \right]}
{\prod_{1 \le i < j \le n} (X_j-X_i)}  
\\ = \prod_{i=1}^n \frac{Q+X_i}{Q-X_i^2} \frac{ \prod_{1 \le i < j \le n} Q(1+X_i)(1+X_j)+ r X_i X_j}{\prod\limits_{1 \le i < j \le n} (Q-X_i X_j)} 
\end{multline} 

The main purpose of this paper is to derive bounded versions of these identities and to provide combinatorial interpretations of the identities that would allow us to approach them with a combinatorial proof, possibly by a variant of the Robinson-Schensted-Knuth correspondence that mimics the proof for the classical Littlewood identity.
By bounded version we mean that the sums 
$\sum_{0 \le k_1 < k_2 < \ldots < k_n}$ are restricted to, say, 
$\sum_{0 \le k_1 < k_2 < \ldots < k_n \le m}$.  Macdonald \cite{macdonald} has provided such a bounded version of the classical identity \eqref{littlewood}, namely 
\begin{multline} 
\label{littlewoodbounded} 
\sum_{\lambda \subseteq (m^n)} s_{\lambda}(X_1,\ldots,X_n) = 
\sum_{0 \le k_1 \le k_2 \le \ldots \le k_n \le m} 
s_{(k_n,k_{n-1},\ldots,k_1)}(X_1,\ldots,X_n) \\ = 
\frac{ \det_{1 \le i, j \le n} \left( X_i^{j-1} - X_i^{m+2n-j}  \right)
}{\prod_{i=1}^n (1-X_i) \prod_{1 \le i < j \le n} (X_j-X_i)(1-X_i X_j)}, 
\end{multline} 
which he used to prove MacMahon's conjecture. Very recent work on bounded Littlewood identities can be found in \cite{RW21}.

More specifically, we will prove the following.
\begin{theorem} 
\label{boundedrQ} 
For $n \ge 1$, we have 
\begin{multline}
\label{boundedrQid} 
\frac{1}{\prod\limits_{1 \le i < j \le n} (X_j-X_i)} \asym_{X_1,\ldots,X_n} \left[ \prod_{1 \le i < j \le n} (Q+(Q+r) X_i+X_j + X_i X_j) \right. \\ \times 
\left. \sum_{0 \le k_1 < k_2 < \ldots < k_n \le m} \left( \frac{X_1(1+X_1)}{Q+X_1} \right)^{k_1} \left( \frac{X_2(1+X_2)}{Q+X_2} \right)^{k_2} \cdots \left( \frac{X_n(1+X_n)}{Q+X_n} \right)^{k_n}  \right]  \\ = 
{
\frac{\det_{1 \le i, j \le n} \left( a_{j,m,n}(Q,r;X_i) \right)}{\prod\limits_{1 \le i \le j \le n} (Q-X_i X_j) \prod\limits_{1 \le i < j \le n} (X_j-X_i)}}
\end{multline} 
with 
\begin{multline*} 
a_{j,m,n}(Q,r;X) = (1+Q X^{-1}) X^{j}  (1+X)^{j-1} (Q+ r X + Q X)^{n-j}  \\ 
- X^{2 n} Q^{-n} \left( \frac{(1+X) X}{Q+X} \right)^{m}  (1+X) \left( Q X^{-1} \right)^{j} (1+Q X^{-1})^{j-1}  (Q+r Q X^{-1} + Q^2 X^{-1})^{n-j}.
\end{multline*} 
\end{theorem} 

Setting $Q=1$ and $r=w-1$, we obtain, after simplifying the right-hand side, the following corollary. 

\begin{cor} For $n \ge 1$, we have 
\begin{multline} 
\label{rincluded} 
\frac{1}{\prod\limits_{1 \le i < j \le n} (X_j-X_i)} \asym_{X_1,\ldots,X_n} \left[  \prod_{1 \le i < j \le n} (1+w X_i+X_j + X_i X_j) 
\sum_{0 \le k_1 < k_2 < \ldots < k_n \le m} X_1^{k_1} X_2^{k_2} \cdots X_n^{k_n} \right]  \\ = \frac{\det_{1 \le i, j \le n} \left( X_i^{j-1}  (1+X_i)^{j-1} (1+ w X_i)^{n-j}  - X_i^{m+2n-j} (1+X_i^{-1})^{j-1}  (1+w X_i^{-1})^{n-j} \right)}{\prod\limits_{i=1}^n (1-X_i) \prod\limits_{1 \le i < j \le n} (1-X_i X_j)(X_j-X_i)}.
\end{multline} 
\end{cor} 

In the second part of the paper, we will then provide combinatorial interpretations for both sides of 
the identity in the corollary.

\subsection*{Outline}
In Section~\ref{proof}, we give a proof of \eqref{boundedrQ}. In Appendix~\ref{combclassical}, we discuss a point of view on the  combinatorics of the classical Littlewood identity \eqref{littlewood} and its bounded version \eqref{littlewoodbounded} that is beneficial for possible combinatorial proofs of the Littlewood-type identities that we establish in this paper. Recall that this is of interest because such identities have been used several times 
\cite{Fis19a,Fis19b,fourfold} to establish connections between alternating sign matrix objects and plane partition objects. To approach this, we offer combinatorial interpretations of the left-hand sides of \eqref{littlewoodASM2} and \eqref{rincluded} in Section~\ref{LHS} and in Appendix~\ref{furtherLHS}. Then, in Section~\ref{RHS}, 
we offer a combinatorial interpretation of the right-hand sides of \eqref{littlewoodASM2} and \eqref{rincluded}. These interpretations are nicest in the cases $w=0,1$. 
In Section~\ref{outlook}, we offer an outlook on related work on the cases $w=0,-1$, which will appear in a forthcoming paper with Florian Schreier-Aigner.

\section{Proof of Theorem~\ref{boundedrQ}} 
\label{proof}

Bressoud's elementary proof \cite{BressoudElement} of \eqref{littlewoodbounded} turned out to be useful to obtain the following (still elementary, but admittedly very complicated) proof of Theorem~\ref{boundedrQ} provided here. Conceptually the proof is not difficult: We use induction with respect to $n$ and show that both sides satisfy the same recursion. 

\subsection{The case $m \to \infty$.} We start by proving the $m \to \infty$ case of Theorem~\ref{boundedrQ}.
This is equivalent to proving that  
\begin{multline} 
\label{infinite} 
\asym_{X_1,\ldots,X_n} 
\left[ \prod_{1 \le i < j \le n} (Q+(Q+r) X_i + X_j + X_i X_j) 
\prod_{i=1}^n \left( \frac{X_i(1+X_i)}{Q+X_i} \right)^{i-1} 
 \prod_{i=1}^n \frac{1}{1- \prod_{j=i}^n \frac{X_j(1+X_j)}{Q+X_j}} \right]
\\ = \prod_{i=1}^n \frac{Q+X_i}{Q-X_i^2} \frac{ \prod_{1 \le i < j \le n} (X_j - X_i)(Q(1+X_i)(1+X_j)+ r X_i X_j )}{\prod\limits_{1 \le i < j \le n} (Q-X_i X_j)},
\end{multline} 
which is just \eqref{mostgeneralunbounded} multiplied on both sides with $\prod_{1 \le i < j \le n} (X_j-X_j)$.
To see this, we rewrite the left-hand side of \eqref{boundedrQid} by using the summation formula for the geometric series $n$ times. As 
$m \to \infty$, $a_{j,m,n}(Q,r;X_i)$ simplifies to 
$$
(1+Q X^{-1}) X^{j}  (1+X)^{j-1} (Q+ r X + Q X)^{n-j}
$$
in a formal power series sense, and
$$
\det_{1 \le i,j \le n} \left( (1+Q X^{-1}) X^{j}  (1+X)^{j-1} (Q+ r X + Q X)^{n-j} \right) 
$$
can be computed using the Vandermonde determinant evaluation, we are led to the right-hand side of \eqref{infinite} eventually.

We denote by $L_n(X_1,\ldots,X_n)$ the left-hand side of \eqref{infinite} and observe that the following recursion is satisfied.
\begin{multline} 
\label{rec} 
L_n(X_1,\ldots,X_n) = \sum_{k=1}^n (-1)^{k-1} 
\frac{1}{1-\prod_{i=1}^n \frac{X_i(1+X_i)}{Q+X_i} } L_{n-1}(X_1,\ldots,\widehat{X_k},\ldots,X_n) \\
\times \prod_{1 \le j \le n, \atop  j \not= k} \frac{X_j(1+X_j)(Q+(Q+r) X_k + X_j + X_k X_j)}{Q+X_j}, 
\end{multline} 
where $\widehat{X_k}$ means that we omit $X_k$.
Indeed, suppose more generally that 
$$
P(X_1,\ldots,X_n) = \asym_{X_1,\ldots,X_n} \left[ \prod_{1 \le i < j \le n} s(X_i,X_j) 
\prod_{i=1}^n t(X_i)^{i-1} \prod_{i=1}^n \frac{1}{1- \prod_{j=i}^n u(X_j)} \right], 
$$
then 
\begin{equation} 
\label{rec1}
\begin{aligned} 
P(X_1,\ldots,X_n) &= \sum_{k=1}^n \sum_{\sigma \in {\mathcal S}_n: \atop \sigma(1)=k} \sgn \sigma 
\frac{1}{1- \prod_{j=1}^n u(X_j)} \prod_{1 \le j \le n, \atop j \not= k} s(X_k,X_j) t(X_j) \\
 &\qquad \times \sigma \left[ \prod_{2 \le i < j \le n} s(X_i,X_j) 
\prod_{i=2}^n t(X_i)^{i-2} \prod_{i=2}^n \frac{1}{1- \prod_{j=i}^n u(X_j)} \right] \\
&= \sum_{k=1}^n (-1)^{k-1}  
\frac{1}{1- \prod_{j=1}^n u(X_j)} P(X_1,\ldots,\widehat{X_k},\ldots,X_n) \\ 
&\qquad \times 
\prod_{1 \le j \le n, \atop j \not= k} s(X_k,X_j) t(X_j).
\end{aligned} 
\end{equation} 
The last equality follows from the fact that the sign of $\sigma$ is the product of $
(-1)^{k-1}$ and the sign of the restriction of $\sigma$ to $\{2,3,\ldots,n\}$, assuming $\sigma(1)=k$ and ``identifying'' the preimage 
$\{2,3,\ldots,n\}$ as well as the image $\{1,\ldots,n\} \setminus \{k\}$ with $\{1,2,\ldots,n-1\}$ in the natural way.

We show \eqref{infinite} by induction with respect to $n$. The case $n=1$ is easy to check. It suffices to show that the right-hand side of \eqref{infinite} satisfies the recursion 
\eqref{rec}, i.e., 
\begin{multline*}
\prod_{i=1}^n \frac{Q+X_i}{Q-X_i^2} \frac{ \prod_{1 \le i < j \le n} (X_j - X_i)(Q(1+X_i)(1+X_j)+ r X_i X_j )}{\prod\limits_{1 \le i < j \le n} (Q-X_i X_j)}  \\
=  \sum_{k=1}^n (-1)^{k-1}  \frac{1}{1-\prod_{i=1}^n \frac{X_i(1+X_i)}{Q+X_i} }  \prod_{1 \le j \le n, \atop  j \not= k} \frac{X_j(1+X_j)(Q+(Q+r) X_k + X_j + X_k X_j)}{Q-X_j^2} \\ \times
 \frac{ \prod_{1 \le i < j \le n,  i,j \not=k} (X_j - X_i)(Q(1+X_i)(1+X_j)+ r X_i X_j )}{\prod\limits_{1 \le i < j \le n,  i,j \not=k} (Q-X_i X_j)}.
\end{multline*} 
We multiply by $\left(1 - \prod_{i=1}^n \frac{X_i(1+X_i)}{Q+X_i} \right) \prod_{1 \le i \le j \le n} (Q-X_i X_j)$ and obtain
\begin{equation}
\label{leftright} 
\begin{aligned} 
 & \left(\prod_{i=1}^{n} (Q+X_i) -\prod_{i=1}^n X_i(1+X_i) \right) \prod_{1 \le i < j \le n} (X_j - X_i)(Q(1+X_i)(1+X_j)+ r X_i X_j )  \\
&\quad =  \prod_{1 \le i < j \le n} (X_j - X_i)(Q(1+X_i)(1+X_j)+ r X_i X_j ) \\ 
& \qquad \times \sum_{k=1}^n  (Q-X_k^{2}) 
\prod_{1 \le j \le n \atop j \not= k} \frac{X_j (X_j+1)(Q+(Q+r) X_k + X_j + X_k X_j)(Q-X_j X_k)}{(X_j-X_k)(Q(1+X_j)(1+X_k)+ r X_j X_k)}. 
\end{aligned}
\end{equation} 

For each $s \in \{1,2,\ldots,n\}$, both sides are polynomials in $X_s$ of degree no greater than $2n$. It is not hard to see that both sides vanish for $X_s = X_t$ and  $X_s =-\frac{Q(1+X_t)}{Q+Q X_t + r X_t}$ for any $t \in \{1,2,\ldots,n\} \setminus \{s\}$. Moreover, it is also not hard to see that the evaluations also agree for $X_s=0,-1$, which gives a total of $2n$ evaluations for each $X_s$. 

It follows that the difference of the left-hand side and the right-hand side is up to a constant 
in $\mathbb{Q}(Q,r)$ equal to 
\begin{equation} 
\label{difference} 
\prod_{i=1}^n X_i(1+X_i)  \prod_{1 \le i < j \le n} (X_j - X_i)(Q(1+X_i)(1+X_j)+ r X_i X_j).
\end{equation} 
To show that this constant is indeed zero, we consider the following specialization 
$$(X_1,X_2,X_3,X_4,\ldots) = \left(X_1,\frac{Q}{X_1},X_3,\frac{Q}{X_3},\ldots \right).$$ 
Note first that \eqref{difference} does not vanish at this specialization, and, therefore, it suffices to show that the left-hand side and the right-hand side of \eqref{leftright} agree on this specialization. If $n$ is even, this is particularly easy to see, because both sides vanish (on the right-hand side all summands vanish, which is due to the factor $Q-X_j X_k$). If $n$ is odd, then only the last summand on the 
right-hand side remains and it is not hard to see that it is equal to the left-hand side.

\subsection{The general case} 
We rewrite the identity from Theorem~\ref{boundedrQ} that we need to prove as follows.
\begin{multline} 
\label{rewrite}
\det_{1 \le i, j \le n} \left( a_{j,m,n}(Q,r;X_i) \right) \\ \
= 
\asym_{X_1,\ldots,X_n} \left[ \prod_{i=1}^n (Q-X_i^2) \prod_{1 \le i < j \le n} (Q+(Q+r) X_i+X_j + X_i X_j)(Q-X_i X_j)  \right. \\ \times 
\left. \sum_{0 \le k_1 < k_2 < \ldots < k_n \le m} \left( \frac{X_1(1+X_1)}{Q+X_1} \right)^{k_1} \left( \frac{X_2(1+X_2)}{Q+X_2} \right)^{k_2} \cdots \left( \frac{X_n(1+X_n)}{Q+X_n} \right)^{k_n}  \right] 
\end{multline} 
We set 
\begin{multline*} 
F(m;X_1,\ldots,X_n)= \prod_{i=1}^n (Q-X_i^2) \prod_{1 \le i < j \le n} (Q+(Q+r) X_i+X_j + X_i X_j)(Q-X_i X_j)   \\ \times 
\sum_{0 \le k_1 < k_2 < \ldots < k_n \le m} \left( \frac{X_1(1+X_1)}{Q+X_1} \right)^{k_1} \left( \frac{X_2(1+X_2)}{Q+X_2} \right)^{k_2} \cdots \left( \frac{X_n(1+X_n)}{Q+X_n} \right)^{k_n}
\end{multline*} 
and observe that 
\begin{multline*}
F(m;X_1,\ldots,X_n) = (Q-X_1^2)  \prod_{j=2}^n (Q+(Q+r) X_1+X_j + X_1 X_j)(Q-X_1 X_j) \\
\times \sum_{l=0}^m \frac{Q+X_1}{X_1(1+X_1)} \left( \prod_{i=1}^n \frac{X_i(1+X_i)}{Q+X_i} \right)^{l+1} F(m-1-l;X_2,\ldots,X_n).
\end{multline*} 
We set 
$$
A(m;X_1,\ldots,X_n) = \asym_{X_1,\ldots,X_n} F(m;X_1,\ldots,X_n), 
$$
and observe that 
\begin{align*} 
A(m;X_1,\ldots,X_n) &= \sum_{k=1}^n \sum_{l=0}^m (-1)^{k+1} (Q-X_k^2) \frac{Q+X_k}{X_k(1+X_k)} \left( \prod_{i=1}^n \frac{X_i(1+X_i)}{Q+X_i} \right)^{l+1} \\ 
&\quad \times A(m-l-1;X_1,\ldots,\widehat{X_k},\ldots,X_n)  \\
&\quad \times \prod_{1 \le i \le n, i \not= k}  (Q+(Q+r) X_k+X_i + X_i X_k)(Q-X_i X_k),
\end{align*} 
by the same argument that has led to \eqref{rec1}. By the induction hypothesis, we have 
$$
A(m-l-1;X_1,\ldots,\widehat{X_k},\ldots,X_n) = \det_{1 \le i \le n, i \not= k \atop 1 \le j \le n-1} \left( a_{j,m-l-1,n-1}(Q,r;X_i) \right).
$$
Therefore, the right-hand side of \eqref{rewrite} is 
\begin{multline} 
\label{expand} 
\sum_{k=1}^n  (-1)^{k+1} (Q-X_k^2)  \prod_{1 \le i \le n, i \not= k}  (Q+(Q+r) X_k+X_i + X_i X_k)(Q-X_i X_k) \\ 
\times 
\sum_{l=0}^m \left( \frac{X_k(1+X_k)}{Q+X_k} \right)^{l} \det_{1 \le i \le n, i \not= k \atop 1 \le j \le n-1} \left(  \left( \frac{X_i(1+X_i)}{Q+X_i} \right)^{l+1}  a_{j,m-l-1,n-1}(Q,r;X_i) \right)
\end{multline} 
and we need to show that it is equal to $\det_{1 \le i, j \le n} \left( a_{j,m,n}(Q,r;X_i) \right)$.

Noting that 
\begin{equation} 
\begin{aligned} 
& \left( \frac{X(1+X)}{Q+X} \right)^{l+1} a_{j,m-l-1,n-1}(Q,r;X) \\
&\quad = (1+Q X^{-1}) X^{j+l+1}  (1+X)^{j+l} (Q+ r X + Q X)^{n-1-j} (Q+X)^{-l-1}  \\ 
&\qquad - X^{2 n-2} Q^{-n+1} \left( \frac{(1+X) X}{Q+X} \right)^{m}  (1+X) \left( Q X^{-1} \right)^{j} (1+Q X^{-1})^{j-1}  (Q+r Q X^{-1} + Q^2 X^{-1})^{n-1-j} \\
&\quad =  X^{j+l}  (1+X)^{j+l} (Q+X)^{-l} (Q+ r X + Q X)^{n-1-j} \\ 
&\qquad -  X^{-j+m+n} (1+X)^{m+1}  (Q+X)^{j-m-1}  (X+r+Q)^{n-1-j}, 
\end{aligned} 
\end{equation} 
we can write the determinant in \eqref{expand} as 
\begin{multline*} 
\sum_{\sigma, S} (-1)^{I(\sigma)+|S|} \prod_{i \in S} X_i^{-\sigma(i)+m+n} (1+X_i)^{m+1}  (Q+X_i)^{\sigma(i)-m-1}  (X_i+r+Q)^{n-1-\sigma(i)} \\ \times \prod_{i \in \overline{S}}  X_i^{\sigma(i)+l}  (1+X_i)^{\sigma(i)+l} (Q+X_i)^{-l} (Q+ r X_i + Q X_i)^{n-1-\sigma(i)}, 
\end{multline*}
where the sum is over all bijections $\sigma: \{1,2,\ldots,n \} \setminus \{k\} \to \{1,2,\ldots,n-1\}$, all subsets $S$ of $ \{1,2,\ldots,n \} \setminus \{k\}$ and $I(\sigma)$ is the number of all inversions, i.e., pairs $i,j \in \{1,2,\ldots,n \} \setminus \{k\} $ with $i<j$ and $\sigma(i)>\sigma(j)$. Moreover, $\overline{S}$ denotes the complement of $S$ in $\{1,2,\ldots,n \} \setminus \{k\}$. Comparing with \eqref{expand}, we multiply by $\left( \frac{X_k(1+X_k)}{Q+X_k} \right)^{l}$ and  take the sum over $l$.
\begin{multline*} 
\sum_{\sigma, S} (-1)^{I(\sigma)+|S|} \prod_{i \in S} X_i^{-\sigma(i)+m+n} (1+X_i)^{m+1}  (Q+X_i)^{\sigma(i)-m-1}  (X_i+r+Q)^{n-1-\sigma(i)} \\ \times \prod_{i \in \overline{S}}  X_i^{\sigma(i)}  (1+X_i)^{\sigma(i)} (Q+ r X_i + Q X_i)^{n-1-\sigma(i)} \sum_{l=0}^{m} \left( \frac{X_k (1+X_k)}{Q+X_k} \right)^l \prod_{i \in \overline{S}} \left( \frac{X_i (1+X_i)}{Q+X_i} \right)^l. \\
\end{multline*}
We evaluate the sum and rearrange some terms.
\begin{multline*}
\sum_{S} (-1)^{|S|} \frac{1-\left( \frac{X_k (1+X_k)}{Q+X_k} \right)^{m+1} \prod_{i \in \overline{S}} \left( \frac{X_i (1+X_i)}{Q+X_i} \right)^{m+1}}
{1-\frac{X_k (1+X_k)}{Q+X_k} \prod_{i \in \overline{S}}  \frac{X_i (1+X_i)}{Q+X_i}} \\ \times  \prod_{i \in S} X_i^{m+n-1} (1+X_i)^{m+1} (Q+X_i)^{-m}  (X_i+r+Q)^{n-2}  \prod_{i \in \overline{S}} X_i (1+X_i) (Q+r X_i + Q X_i)^{n-2} 
\\ \times \sum_{\sigma} (-1)^{I(\sigma)}\prod_{i \in S} (Q X_i^{-1})^{\sigma(i)-1}  (1+Q X_i^{-1})^{\sigma(i)-1} (Q+ r Q X_i^{-1} + Q^2 X_i^{-1})^{-\sigma(i)+1} \\ \times 
\prod_{i \in \overline{S}}  X_i^{\sigma(i)-1}  (1+X_i)^{\sigma(i)-1} (Q+ r X_i + Q X_i)^{-\sigma(i)+1}
\end{multline*} 
The inner sum is a Vandermonde determinant, which we evaluate. We obtain
\begin{multline*}
\sum_{S} (-1)^{|S|} \frac{1-\left( \frac{X_k (1+X_k)}{Q+X_k} \right)^{m+1} \prod_{i \in \overline{S}} \left( \frac{X_i (1+X_i)}{Q+X_i} \right)^{m+1}}
{1-\frac{X_k (1+X_k)}{Q+X_k} \prod_{i \in \overline{S}}  \frac{X_i (1+X_i)}{Q+X_i}} \\ \times  \prod_{i \in S} X_i^{m+n-1} (1+X_i)^{m+1} (Q+X_i)^{-m}  (X_i+r+Q)^{n-2}  \prod_{i \in \overline{S}} X_i (1+X_i) (Q+r X_i + Q X_i)^{n-2} 
\\ \times 
\prod_{1 \le i < j \le n, i,j \not=k} \left(\frac{Y_j (1+Y_j)}{Q+r Y_j + Q Y_j}-\frac{Y_i (1+Y_i)}{Q+r Y_i + Q Y_i} \right), 
\end{multline*} 
with $Y_i=X_i$ if $i \in \overline{S}$ and $Y_i=Q X_i^{-1}$ if $i \in S$. 

From \eqref{expand}, we add the sum over all $k$ and finally have the full right-hand side of \eqref{rewrite}. 
We exchange the sum over $k$ and $S$: now we sum 
over all proper subsets $S \subseteq [n]$ and all $k$ not in $S$. If we write $i \notin S$, then 
we mean $i \in \{1,2,\ldots,n\} \setminus S$.
\begin{equation}
\begin{aligned} 
\label{zwischen} 
&\sum_{S} (-1)^{|S|} \frac{1- \prod_{i \notin S} \left( \frac{X_i (1+X_i)}{Q+X_i} \right)^{m+1}}
{1- \prod_{i \notin S}  \frac{X_i (1+X_i)}{Q+X_i}} \\ 
&\quad \times  \prod_{i \in S} X_i^{m+n-1} (1+X_i)^{m+1} (Q+X_i)^{-m}  (X_i+r+Q)^{n-2} 
\prod_{i \notin S} X_i (1+X_i) (Q+r X_i + Q X_i)^{n-2} \\
&\quad \times \sum_{k \notin S} (-1)^{k+1} (Q-X_k^2) X_k^{-1} (1+X_k)^{-1} (Q+r X_k + Q X_k)^{-n+2}  \\
&\qquad \times
\prod_{1 \le i \le n, i \not= k}  (Q+(Q+r) X_k+X_i + X_i X_k)(Q-X_i X_k) \\ 
&\qquad \times 
\prod_{1 \le i < j \le n, i,j \not=k} \left( \frac{Y_j (1+Y_j)}{Q+r Y_j + Q Y_j}- \frac{Y_i (1+Y_i)}{Q+r Y_i + Q Y_i} \right) 
\end{aligned} 
\end{equation}
We rewrite 
$$
\begin{aligned} 
&(-1)^{k-1} \prod_{1 \le i \le n, i \not= k}  (Q+(Q+r) X_k+X_i + X_i X_k)(Q-X_i X_k) \\ 
&\quad = 
\prod_{1 \le i \le k-1 \atop i \notin S} (Q+(Q+r) X_k+X_i + X_i X_k)(X_i X_k - Q) \\
&\qquad \times
\prod_{k+1 \le i \le n \atop  i \notin S} (Q+(Q+r) X_k+X_i + X_i X_k)(Q-X_i X_k) \\
&\qquad \times \prod_{i \in S} X_i  (X_i + r + Q)(Q+r X_k + Q X_k) \\ 
&\qquad \times 
\prod_{1 \le i \le k-1 \atop  i \in S} \left( \frac{X_k (1+X_k)}{Q+r X_k + Q X_k} - 
\frac{Q X_i^{-1} (1+Q X_i^{-1})}{Q+r Q X_i^{-1} Q^2 X_i^{-1}} \right) \\
&\qquad \times
\prod_{k+1 \le i \le n \atop  i \in S} \left( \frac{Q X_i^{-1} (1+Q X_i^{-1})}{Q+r Q X_i^{-1} Q^2 X_i^{-1}} -
\frac{X_k (1+X_k)}{Q+r X_k + Q X_k} \right). 
\end{aligned} 
$$
We use this to rewrite \eqref{zwischen} as follows.
\begin{align*} 
&\sum_{S} (-1)^{|S|} \frac{1- \prod_{i \notin S} \left( \frac{X_i (1+X_i)}{Q+X_i} \right)^{m+1}}
{1- \prod_{i \notin S}  \frac{X_i (1+X_i)}{Q+X_i}} \prod_{i} X_i \\ 
&\quad \times
\prod_{i \in S} X_i^{m+n-1} (1+X_i)^{m+1} (Q+X_i)^{-m}  (X_i+r+Q)^{n-1}  \prod_{i \notin S}  (1+X_i) (Q+r X_i + Q X_i)^{n-1} \\
&\quad \times \sum_{k \notin S} (Q-X_k^2) X_k^{-1} (1+X_k)^{-1}  \\
&\qquad \times \prod_{1 \le i \le k-1 \atop i \notin S} \frac{(Q+(Q+r) X_k+X_i + X_i X_k)(X_i X_k - Q)}{(Q+r X_k + Q X_k)(Q+r X_i + Q X_i)} \\
&\qquad \times
\prod_{k+1 \le i \le n \atop  i \notin S} \frac{(Q+(Q+r) X_k+X_i + X_i X_k)(Q-X_i X_k)}{(Q+r X_k + Q X_k)(Q+r X_i + Q X_i)} \\
&\qquad \times 
\prod_{1 \le i \le k-1 \atop  i \in S} \left( \frac{X_k (1+X_k)}{Q+r X_k + Q X_k} - 
\frac{Q X_i^{-1} (1+Q X_i^{-1})}{Q+r Q X_i^{-1} Q^2 X_i^{-1}} \right)  \\
& \qquad \times
\prod_{k+1 \le i \le n \atop  i \in S} \left( \frac{Q X_i^{-1} (1+Q X_i^{-1})}{Q+r Q X_i^{-1} Q^2 X_i^{-1}} -
\frac{X_k (1+X_k)}{Q+r X_k + Q X_k} \right) \\
&\qquad \times 
\prod_{1 \le i < j \le n, i,j \not=k} \left( \frac{Y_j (1+Y_j)}{Q+r Y_j + Q Y_j}-\frac{Y_i (1+Y_i)}{Q+r Y_i + Q Y_i} \right)
\end{align*}  
This is further equal to 
\begin{equation} 
\label{next} 
\begin{aligned}
&\sum_{S} (-1)^{|S|} \frac{1- \prod_{i \notin S} \left( \frac{X_i (1+X_i)}{Q+X_i} \right)^{m+1}}
{1- \prod_{i \notin S}  \frac{X_i (1+X_i)}{Q+X_i}} \prod_{i} X_i \\ 
&\quad \times
\prod_{i \in S} X_i^{m+n-1} (1+X_i)^{m+1} (Q+X_i)^{-m}  (X_i+r+Q)^{n-1} \\
&\quad \times  \prod_{i \notin S}  (1+X_i) (Q+r X_i + Q X_i)^{n-1} \\
&\quad \times \prod_{1 \le i < j \le n, \{i,j\} \cap S \not= \emptyset} \left( \frac{Y_j (1+Y_j)}{Q+r Y_j + Q Y_j}-\frac{Y_i (1+Y_i)}{Q+r Y_i + Q Y_i} \right) \\
&\quad \times \sum_{k \notin S} (Q-X_k^2) X_k^{-1} (1+X_k)^{-1} \\
&\qquad \times \prod_{1 \le i \le k-1 \atop i \notin S} \frac{(Q+(Q+r) X_k+X_i + X_i X_k)(X_i X_k - Q)}{(Q+r X_k + Q X_k)(Q+r X_i + Q X_i)} \\
&\qquad \times
\prod_{k+1 \le i \le n \atop  i \notin S} \frac{(Q+(Q+r) X_k+X_i + X_i X_k)(Q-X_i X_k)}{(Q+r X_k + Q X_k)(Q+r X_i + Q X_i)} \\
&\qquad \times 
\prod_{1 \le i < j \le n, i,j \notin S \cup \{k\}} \left( \frac{X_j (1+X_j)}{Q+r X_j + Q X_j}-\frac{X_i (1+X_i)}{Q+r X_i + Q X_i} \right).
\end{aligned} 
\end{equation} 
We divide \eqref{leftright} by $\prod_{i=1}^n (Q+r X_i + Q X_i)^{n-1} X_i (1+X_i)$, and, after some further modifications, we obtain  
$$
\begin{aligned} 
& \left(\prod_{i=1}^{n} \frac{Q+X_i}{X_i (1+X_i)}  - 1 \right) \prod_{1 \le i < j \le n} 
\left( \frac{X_j (1+X_j)}{Q+r X_j + Q X_j}-\frac{X_i (1+X_i)}{Q+r X_i + Q X_i} \right)  \\
&=  \sum_{k=1}^n  (Q-X_k^{2})  X_k^{-1} (1+X_k)^{-1} \\ 
&\quad \times \prod_{j=1}^{k-1} \frac{(Q+(Q+r) X_k + X_j + X_k X_j)(X_j X_k-Q)}{(Q+r X_j + Q X_j)(Q+r X_k + Q X_k)} \\
&\quad \times \prod_{j=k+1}^{n} \frac{(Q+(Q+r) X_k + X_j + X_k X_j)(Q-X_j X_k)}{(Q+r X_j + Q X_j)(Q+r X_k + Q X_k)}
\\ 
&\quad \times \prod_{1 \le i < j \le n, i,j \not=k} 
 \left( \frac{X_j (1+X_j)}{Q+r X_j + Q X_j}-\frac{X_i (1+X_i)}{Q+r X_i + Q X_i} \right).
\end{aligned} 
$$ 
We can use this to replace the sum over all $k \in S$ in \eqref{next} by something simpler.
\begin{align*} 
 &\sum_{S} (-1)^{|S|} \frac{1- \prod_{i \notin S} \left( \frac{X_i (1+X_i)}{Q+X_i} \right)^{m+1}}
{1- \prod_{i \notin S}  \frac{X_i (1+X_i)}{Q+X_i}} \prod_{i} X_i \\ 
&\quad \times
\prod_{i \in S} X_i^{m+n-1} (1+X_i)^{m+1} (Q+X_i)^{-m}  (X_i+r+Q)^{n-1} 
 \prod_{i \notin S}  (1+X_i) (Q+r X_i + Q X_i)^{n-1} \\
&\quad \times \prod_{1 \le i < j \le n, \{i,j\} \cap S \not= \emptyset} \left( \frac{Y_j (1+Y_j)}{Q+r Y_j + Q Y_j}-\frac{Y_i (1+Y_i)}{Q+r Y_i + Q Y_i} \right) \left(\prod_{i \notin S} \frac{Q+X_i}{X_i (1+X_i)}  - 1 \right) \\
&\quad \times \prod_{1 \le i < j \le n, i,j \notin S} 
\left( \frac{X_j (1+X_j)}{Q+r X_j + Q X_j}-\frac{X_i (1+X_i)}{Q+r X_i + Q X_i} \right)
\end{align*}
This can be further simplified as follows, 
\begin{multline*} 
 \sum_{S} (-1)^{|S|} \left( 1- \prod_{i \notin S} \left( \frac{X_i (1+X_i)}{Q+X_i} \right)^{m+1} \right) 
\\ \times
\prod_{i \in S} X_i^{m+n} (1+X_i)^{m+1} (Q+X_i)^{-m}  (X_i+r+Q)^{n-1}  \prod_{i \notin S}  (Q+X_i)  (Q+r X_i + Q X_i)^{n-1} \\
\times \prod_{1 \le i < j \le n} \left( \frac{Y_j (1+Y_j)}{Q+r Y_j + Q Y_j}-\frac{Y_i (1+Y_i)}{Q+r Y_i + Q Y_i} \right), 
\end{multline*}
recalling that $Y_i=X_i$ if $i \notin S$ and $Y_i = Q X_i^{-1}$ if $i \in S$. We write this as 
\begin{multline} 
\label{twosum} 
 \sum_{S,\sigma} (-1)^{|S|+I(\sigma)} 
 \prod_{i \in S} X_i^{m+n-\sigma(i)+1} (1+X_i)^{m+1} (Q+X_i)^{-m+\sigma(i)-1}  (X_i+r+Q)^{n-\sigma(i)}
 \\ 
 \times \prod_{i \notin S}  X_i^{\sigma(i)-1} (1+X_i)^{\sigma(i)-1}  (Q+X_i) (Q+r X_i + Q X_i)^{n-\sigma(i)} \\ 
 - \sum_{S,\sigma} (-1)^{|S|+I(\sigma)}  \prod_{i \in S} X_i^{m+n-\sigma(i)+1} (1+X_i)^{m+1} (Q+X_i)^{-m+\sigma(i)-1}  (X_i+r+Q)^{n-\sigma(i)}
 \\ 
 \times \prod_{i \notin S}  X_i^{m+\sigma(i)} (1+X_i)^{m+\sigma(i)}  (Q+X_i)^{-m} (Q+r X_i + Q X_i)^{n-\sigma(i)}.
\end{multline} 
Recall that the sums are over all proper subsets $S$, but since the sums are equal for $S=\{1,2,\ldots,n\}$ we can also sum over all subsets $S$. Now the second sum is equal to 
\begin{multline*} 
\prod_{i=1}^{n} X_i^{m+1} (1+X_i)^{m+1} (Q+X_i)^{-m}  
\\ \times 
 \det_{1 \le i, j \le n} \left( X_i^{j-1} (1+X_i)^{j-1}  (Q+ r X_i + Q X_i)^{n-j} -
 X_i^{n-j} (Q+X_i)^{j-1} (X_i + r + Q)^{n-j} \right).
\end{multline*} 
The determinant can be seen to vanish as follows: First observe that it is a polynomial in $X_1,\ldots,X_n$ of degree no greater than $2n-2$ in each $X_i$. For $1 \le i < j \le n$, the $i$-th row and the $j$-th row of the underlying matrix are collinear when setting $X_i=X_j$ or $X_i=Q X_j^{-1}$. Moreover, the $i$-th row vanishes when setting $X_i^2=Q$. It follows that $\prod_{i=1}^{n} (X_i^2-Q) \prod_{1 \le i < j \le n} (X_j-X_j)(1-Q X_i X_j)$ is a divisor of the determinant, but since it is of degree $2n$ in each $X_i$, the determinant vanishes.  The second sum in \eqref{twosum} remains and it can easily be seen to be equal to 
$\det_{1 \le i,j \le n} \left( a_{j,m,n}(Q,r;X_i) \right)$. This concludes the proof of Theorem~\ref{boundedrQ}.

\section{Combinatorial interpretations of the left-hand sides}
\label{LHS} 

\subsection{Arrowed Gelfand-Tsetlin patterns} 
\label{agtp} 

To continue the analogy with the ordinary Little\-wood identity \eqref{littlewood} and Macdonald's bounded version \eqref {littlewoodbounded} of it, both sides of the identities  \eqref{littlewoodASM2} and \eqref{rincluded} will be interpreted combinatorially.
For the left-hand side, this was accomplished in another recent paper \cite{nASMDPP}, and we will describe the result and adjust to our context next. 

In order to motivate the definition for the combinatorial objects, recall the combinatorial interpretation of the left-hand sides of \eqref{littlewood} and \eqref{littlewoodbounded} in terms of Gelfand-Tsetlin patterns, which is described in Appendix~\ref{gtpattern}. We need to extend the discussion from there in so far
that there is also a sensible extension of the definition of Gelfand-Tsetlin patterns to arbitrary integers sequences 
$(\lambda_1,\ldots,\lambda_n)$. The notion of signed intervals is crucial for this:
$$
\si{a}{b} = \begin{cases} [a,b], & a \le b \\ \emptyset, & b=a-1 \\ [b+1,a-1], & b < a-1 \end{cases}
$$
If we are in the last case, then the interval is said to be \emph{negative}. The condition that defines Gelfand-Tsetlin pattern can 
also be written as $a_{i,j} \in [a_{i+1,j},a_{i+1,j+1}]$. If the bottom row is weakly increasing, we can replace this condition also by 
$a_{i,j} \in \si{a_{i+1,j}}{a_{i+1,j+1}}$ (since we then have $a_{i+1,j} \le a_{i+1,j+1}$ as can be seen inductively with respect to $n$). 

We use this now as the definition for arbitrary bottom rows: A (generalized) Gelfand-Tsetlin pattern is a triangular array $A=(a_{i,j})_{1 \le j \le i \le n}$ 
of integers with $a_{i,j} \in \si{a_{i+1,j}}{a_{i+1,j+1}}$ for all $i,j$. Then the sign of a Gelfand-Tsetlin pattern $A$ is 
$$
(-1)^{\# \text{ of negative intervals $\si{a_{i+1,j}}{a_{i+1,j+1}}$}}=: \sgn A.
$$
Then 
\begin{equation}
\label{schurextension}  
s_{(\lambda_1,\ldots,\lambda_n)}(X_1,\ldots,X_n) = \sum_{A=\left( a_{i,j} \right)_{1 \le j \le i \le n}} \sgn A  \prod_{i=1}^n X_i^{\sum_{j=1}^i a_{i,j} - \sum_{j=1}^{i-1} a_{i-1,j}}, 
\end{equation} 
where the sum is over all Gelfand-Tsetlin patterns $A=(a_{i,j})_{1 \le j \le i \le n}$ with bottom row 
$(\lambda_n,\lambda_{n-1},\ldots,\lambda_1)$ and 
$$s_{(\lambda_1,\ldots,\lambda_n)}(X_1,\ldots,X_n) = \frac{\det_{1 \le i,j \le n} \left( X_i^{\lambda_j+n-j} \right)}{\prod_{1 \le i < j \le n} (X_i-X_j)}.$$ This result is a special case of Theorem~\ref{robbins} below that will also 
cover the combinatorial interpretation of the left-hand side of \eqref{littlewoodASM2} and \eqref{rincluded}. However, this special case appeared essentially also earlier in \cite{Fis05} (with some details missing).

\begin{definition}
\label{def:AGTP} 
An \emph{arrowed Gelfand-Tsetlin pattern} (AGTP)\footnote{They appeared first in \cite{nASMDPP} as extended arrowed monotone triangles.} is a triangular array of the following form
$$
\begin{array}{ccccccccccccccccc}
  &   &   &   &   &   &   &   & a_{1,1} &   &   &   &   &   &   &   & \\
  &   &   &   &   &   &   & a_{2,1} &   & a_{2,2} &   &   &   &   &   &   & \\
  &   &   &   &   &   & \dots &   & \dots &   & \dots &   &   &   &   &   & \\
  &   &   &   &   & a_{n-2,1} &   & \dots &   & \dots &   & a_{n-2,n-2} &   &   &   &   & \\
  &   &   &   & a_{n-1,1} &   & a_{n-1,2} &  &   \dots &   & \dots   &  & a_{n-1,n-1}  &   &   &   & \\
  &   &   & a_{n,1} &   & a_{n,2} &   & a_{n,3} &   & \dots &   & \dots &   & a_{n,n} &   &   &
\end{array},
$$
where each entry $a_{i,j}$ is an integer decorated with an element from $\{\nwarrow, \nearrow, \nenwarrow,\emptyset \}$ and the following is satisfied for each entry $a$ not in the bottom row: Suppose $b$ is the $\swarrow$-neighbor of $a$ and $c$ is the $\searrow$-neighbor of $a$, respectively, i.e.,
$$
\begin{array}{ccc}
&a& \\
b&&c
\end{array}.
$$
Depending on the decoration of $b, c$, denoted by $\dec(b)$ and $\dec ( c )$, respectively, we need to consider four cases:
\begin{itemize}
\item $(\dec(b),\dec( c )) \in \{\nwarrow,\emptyset\} \times  \{\nearrow, \emptyset\}$: $a \in \si{b}{c}$
\item $(\dec(b),\dec( c )) \in \{\nwarrow,\emptyset\} \times \{\nwarrow, \nenwarrow\}$: $a \in \si{b}{c-1}$
\item $(\dec(b),\dec( c )) \in \{\nearrow, \nenwarrow \} \times \{\nearrow,\emptyset\}$: $a \in \si{b+1}{c}$
\item $(\dec(b),\dec( c)) \in \{\nearrow, \nenwarrow \} \times \{\nwarrow, \nenwarrow\}$: $a \in \si{b+1}{c-1}$
\end{itemize}
\end{definition} 

An example is provided next. We write $^\nwarrow e, e^\nearrow, ^\nwarrow e^\nearrow, e$ if the entry $e$ is decorated with 
$\nwarrow, \nearrow, \nenwarrow,\emptyset$, respectively.
$$
\begin{array}{ccccccccccccccccc}
  &   &   &   &   &   &   &   & ^\nwarrow2 &   &   &   &   &   &   &   & \\
  &   &   &   &   &   &   & 2 &   & ^\nwarrow3^\nearrow &   &   &   &   &   &   & \\
  &   &   &   &   &   & ^\nwarrow2 &   & 2^\nearrow &   & 3^\nearrow &   &   &   &   &   & \\
  &   &   &   &   & 3 &   & ^\nwarrow2 &   & ^\nwarrow3^\nearrow &   & ^\nwarrow3^\nearrow &   &   &   &   & \\
  &   &   &   & 2^\nearrow &   & 4 &  &   ^\nwarrow2^\nearrow &   & 3^\nearrow   &  & 2  &   &   &   & \\
  &   &   & ^\nwarrow6 &   & ^\nwarrow2^\nearrow &   & 5 &   & 1^\nearrow &   & ^\nwarrow4 &   & ^\nwarrow2^\nearrow &   &   &
\end{array}
$$

We define the sign of an AGTP $A=(a_{i,j})_{1 \le j \le i \le n}$ as follows:
Each negative interval 
$\si{a_{i+1,j}(+1)}{a_{i+1,j+1}(-1)}$ with $i \ge 1$ and $j \le i$ contributes a multiplicative 
$-1$, choosing $a_{i+1,j}+1$ iff $\dec(a_{i+1,j}) \in 
\{\nearrow, \nenwarrow \}$  and $a_{i+1,j}$ otherwise, and choosing $a_{i+1,j+1} - 1$ iff 
$\dec(a_{i+1,j+1}) \in \{ \nwarrow, \nenwarrow \}$ and $a_{i+1,j+1}$ otherwise. There are 
no negative intervals in rows $1,2,3$, two in rows $4,5$ and three in row $6$, so that the sign of the pattern is 
$-1$.  
 
We associate the following weight to a given arrowed Gelfand-Tsetlin pattern $A=(a_{i,j})_{1 \le j \le i \le n}$:
$$
\W(A) = \sgn(A) t^{\# \emptyset}  u^{\# \nearrow} v^{\# \nwarrow} w^{\# \nenwarrow}   \prod_{i=1}^{n} X_i^{\sum_{j=1}^i a_{i,j}  - \sum_{j=1}^{i-1} a_{i-1,j} + \# \nearrow \text{in row $i$ } - \# \nwarrow \text{in row $i$ }} 
$$
The weight of our example is 
$$
- t^5 u^5 v^5 w^6 X_1 X_2^3 X_3^3 X_4^3 X_5^4 X_6^6.
$$

For this paper only arrowed Gelfand-Tsetlin patterns with weakly increasing bottom row are relevant and in this case the description of the 
objects can be simplified considerably as follows.

\begin{prop} An arrowed Gelfand-Tsetlin pattern with weakly increasing bottom row is an ordinary Gelfand-Tsetlin pattern (i.e., with weakly increasing rows), where each entry is decorated with an element from $\{\nwarrow, \nearrow, \nenwarrow,\emptyset \}$ such that the following is satisfied.
\begin{itemize}
\item Suppose an entry $a$ is equal to its $\nearrow$-neighbour and $a$ is decorated with either $\nearrow$ or $\nenwarrow$ (i.e., an arrow is pointing from $a$ to its $\nearrow$-neighbour), then the entry right of $a$ in the same row is also equal to $a$ and decorated with $\nwarrow$ or 
$\nenwarrow$.
\item Suppose an entry $a$ is equal to its $\nwarrow$-neighbour and $a$ is decorated with either $\nwarrow$ or $\nenwarrow$ (i.e., an arrow is 
pointing from $a$ to its $\nwarrow$-neighbour), then the entry left of $a$ in the same row is also equal to $a$ and decorated with 
$\nearrow$ or $\nenwarrow$. 
\end{itemize} 
The sign is $-1$ to the number of entries $a$ that are equal to their $\swarrow$-neighbor $b$ as well as to to their $\searrow$-neighbor $c$, and $b$ is decorated with $\nearrow$ or $\nenwarrow$ and $c$ is decorated with $\nwarrow$ and $\nenwarrow$. 
\end{prop} 

\begin{proof} 
Suppose $(a_{i,j})_{1 \le j \le i \le n}$ is an AGTP. If $a_{i+1,j} < a_{i+1,j+1}$ for particular $i,j$, then $a_{i+1,j} \le a_{i,j} \le a_{i+1,j+1}$. The first inequality has to be strict if the decoration of $a_{i+1,j}$ contains an arrow pointing towards $a_{i,j}$ (i.e., $\dec(a_{i+1,j}) \in \{\nearrow,\nenwarrow\}$), while the second inequality has to be strict if $a_{i+1,j+1}$ contains an arrow pointing towards $a_{i,j}$ (i.e., $\dec(a_{i+1,j+1}) \in \{\nwarrow,\nenwarrow\}$).  

On the other hand, if $a_{i+1,j} = a_{i+1,j+1}$ for particular $i,j$, then $a_{i+1,j}=a_{i,j}=a_{i+1,j+1}$. In this case 
$$
(\dec(a_{i+1,j}),\dec(a_{i+1,j+1})) \in \{\emptyset,\nwarrow\} \times \{\emptyset,\nearrow\}
$$
or 
\begin{equation}
\label{sign} 
(\dec(a_{i+1,j}),\dec(a_{i+1,j+1})) \in \{\nearrow,\nenwarrow\} \times \{\nwarrow,\nenwarrow\},
\end{equation}
where in the second case there is a contribution of $-1$ to the sign of the object.

These observations imply that, if the bottom row is weakly increasing, then the underlying undecorated triangular array is an ordinary Gelfand-Tsetlin pattern and that the properties on the decoration stated in the proposition are satisfied. The only instance when we have a contribution to the sign is in the case of \eqref{sign}. 

Conversely, a decoration of a given Gelfand-Tsetlin pattern that follows the rule as given in the statement of the proposition is eligible for an arrowed Gelfand-Tsetlin pattern according to 
Definition~\ref{def:AGTP}.
\end{proof} 

\begin{remark} 
In the case that the bottom row of an arrowed Gelfand-Tsetlin pattern is strictly increasing and we forbid the decoration $\emptyset$, we have that all rows are strictly increasing and we obtain a monotone triangle. Recall that monotone triangles are defined as Gelfand-Tsetlin patterns with strictly increasing rows; their significance comes from the fact that monotone triangles with bottom row $1,2,\ldots,n$ are in easy bijective correspondence with $n \times n$ alternating sign matrices, see, e.g., \cite{Bre99}. In such a case, there is no instance where we gain a $-1$ that contributes to the sign. These objects were used in \cite{nASMDPP} to study alternating sign matrices. Among other things, the generating function of these decorated monotone triangles can be interpreted as a generating function of (undecorated) monotone triangles, thus of alternating sign matrices.
\end{remark} 

The following explicit formula for the generating function of arrowed Gelfand-Tsetlin patterns with fixed bottom row $k_1,k_2,\ldots,k_n$ is proved in
\cite{nASMDPP}. 

\begin{theorem} 
\label{robbins}
The generating function of arrowed Gelfand-Tsetlin patterns with bottom row $k_1,\ldots,k_n$ is 
$$
\prod_{i=1}^{n} (t + u X_i + v X_i^{-1} + w)  
\prod_{1 \le i < j \le n} \left( t  + u  \e_{k_i} + v  \e_{k_j}^{-1} + w \e_{k_i} \e_{k_j}^{-1} \right) 
s_{(k_n,k_{n-1},\ldots,k_1)}(X_1,\ldots,X_n), 
$$ 
where $\e_x$ denotes the shift operator, defined as $\e_x p(x) = p(x+1)$. 
\end{theorem} 

The formula has to be applied as follows: First interpret  $k_1,\ldots,k_n$ as variables and apply the operator $\prod_{1 \le i < j \le n} \left( t  + u  \e_{k_i} + v  \e_{k_j}^{-1} + w \e_{k_i} \e_{k_j}^{-1} \right)$ to $s_{(k_n,k_{n-1},\ldots,k_1)}(X_1,\ldots,X_n)$. This will result in a linear combination of expressions of the form 
$s_{(k_n+i_n,k_{n-1}+i_{n-1},\ldots,k_1+i_1)}(X_1,\ldots,X_n)$ for some (varying) integers $i_j$. The $k_j$ are only specialized to the actual integers after that. Note that we do not necessarily have $k_n+i_n \ge k_{n-1}+i_{n-1} \ge \ldots \ge k_1+i_1$ even if $k_n \ge k_{n-1} \ge \ldots \ge k_1$, so that the extension of the Schur polynomial in \eqref{schurextension} is necessary. 

\medskip

\begin{example}
We illustrate the theorem on the example $(k_1,k_2,k_3)=(1,2,3)$. We list the $8$ Gelfand-Tsetlin pattern with bottom row $1,2,3$ and indicate 
the possible decorations (one will be listed twice with a disjoint set of decorations), where $L=\{\emptyset, \nwarrow~\}$, $R=\{\emptyset, \nearrow\}$ and 
$LR= \{\emptyset,\nwarrow,\nearrow,\nenwarrow\}$, and on the right we indicate the 
generating function restricted to the particular underlying Gelfand-Tsetlin patterns with the indicated decorations, where we use 
$$
L(X)=t+ v X^{-1}, R(X)=t+u X \quad \text{and} \quad  LR(X) = t + u X + v X^{-1} + w.
$$

\begin{tabular}{cl} 
$
\begin{array}{ccccc}
&& {LR \atop 1} && \\ 
& {L \atop 1} && {LR \atop 2}  & \\
{L \atop 1} && {L \atop 2} && {LR \atop 3}  
\end{array} 
$
&  
$X_1 X_2^2 X_3^3 LR(X_1) L(X_2) LR(X_2) L(X_3)^2 LR(X_3)$ \vspace{3mm} \\
$
\begin{array}{ccccc}
&& {LR \atop 2} && \\ 
& {LR \atop 1} && {R \atop 2}  & \\
{L \atop 1} && {L \atop 2} && {LR \atop 3}  
\end{array} 
$
& 
$X_1^2 X_2 X_3^3 LR(X_1) LR(X_2) R(X_2) L(X_3)^2 LR(X_3)$ \vspace{3mm} \\
$
\begin{array}{ccccc}
&& {LR \atop 1} && \\ 
& {L \atop 1} && {LR \atop 3}  & \\
{L \atop 1} && {LR \atop 2} && {R \atop 3}  
\end{array} 
$
& 
$X_1 X_2^3 X_3^2 LR(X_1) L(X_2) LR(X_2) L(X_3) LR(X_3) R(X_3)$ \vspace{3mm} \\
$
\begin{array}{ccccc}
&& {LR \atop 2} && \\ 
& {LR \atop 1} && {LR \atop 3}  & \\
{L \atop 1} && {LR \atop 2} && {R \atop 3}  
\end{array} 
$
& 
$X_1^2 X_2^2 X_3^2 LR(X_1) LR(X_2)^2 L(X_3) LR(X_3) R(X_3)$ \vspace{3mm} \\
$
\begin{array}{ccccc}
&& {LR \atop 3} && \\ 
& {LR \atop 1} && {R \atop 3}  & \\
{L \atop 1} && {LR \atop 2} && {R \atop 3}  
\end{array} 
$
& 
$X_1^3 X_2 X_3^2 LR(X_1) LR(X_2) R(X_2) L(X_3) LR(X_3) R(X_3)$ \vspace{3mm} \\
$
\begin{array}{ccccc}
&& {LR \atop 2} && \\ 
& {L \atop 2} && {LR \atop 3}  & \\
{LR \atop 1} && {R \atop 2} && {R \atop 3}  
\end{array} 
$
& 
$X_1^2 X_2^3 X_3 LR(X_1) L(X_2) LR(X_2) LR(X_3) R(X_3)^2$ \vspace{3mm} \\
$
\begin{array}{ccccc}
&& {LR \atop 3} && \\ 
& {LR \atop 2} && {R \atop 3}  & \\
{LR \atop 1} && {R \atop 2} && {R \atop 3}  
\end{array} 
$
& 
$X_1^3 X_2^2 X_3 LR(X_1) LR(X_2) R(X_2) LR(X_3) R(X_3)^2$ \vspace{3mm} \\
$
\begin{array}{ccccc}
&& {LR \atop 2} && \\ 
& {L \atop 2} && {R \atop 2}  & \\
{LR \atop 1} && {\emptyset  \atop 2} && {LR \atop 3}  
\end{array} 
$
& 
$X_1^2 X_2^2 X_3^2 LR(X_1) L(X_2) R(X_2) t LR(X_3)^2 $ \vspace{3mm} \\
$
\begin{array}{ccccc}
&& {LR \atop 2} && \\ 
& {\{ \nearrow, \nenwarrow \} \atop 2} && { \{ \nwarrow, \nenwarrow \} \atop 2}  & \\
{LR \atop 1} && {\emptyset  \atop 2} && {LR \atop 3}  
\end{array} 
$
& 
$-X_1^2 X_2^2 X_3^2 LR(X_1) (w+uX_2)(w+v X_2^{-1}) t LR(X_3)^2 $ \vspace{3mm} \\
\end{tabular} 
\end{example}

It is convenient for us to rewrite the formula from Theorem~\ref{robbins} as follows.

\begin{cor}
\label{bialternant}
The generating function of arrowed Gelfand-Tsetlin patterns with bottom row $k_1,\ldots,k_n$ is 
\begin{equation}
\label{asym}
\frac{\asym_{X_1,\ldots,X_n} \left[ \prod_{1 \le i \le j \le n} \left(v  + w X_i + t X_j + u X_i X_j  \right) \prod_{i=1}^{n} X_i^{k_i-1} \right] }{\prod_{1 \le i < j \le n} (X_j - X_i)}.
\end{equation}
\end{cor}

\begin{proof} 
Observe that 
\begin{align*}
& \prod_{i=1}^{n} (t + u  X_i + v  X_i^{-1}+w) \prod_{1 \le i < j \le n} \left(t+ u  \e_{k_i} + v  \e_{k_j}^{-1} + w  \e_{k_i} \e_{k_j}^{-1} \right) 
s_{(k_n,k_{n-1},\ldots,k_1)}(X_1,\ldots,X_n) \notag \\
\quad &= \prod_{i=1}^{n} (u  X_i + v  X_i^{-1}+w) \prod_{1 \le i < j \le n} \left(t +  u  \e_{k_i} + v  \e_{k_j}^{-1} + w  \e_{k_i} \e_{k_j}^{-1} \right)
\frac{\asym_{X_1,\ldots,X_n} \left[ \prod_{i=1}^{n} X_i^{k_i+i-1} \right] }{\prod_{1 \le i < j \le n} (X_j - X_i)} \notag \\ 
\quad &= \prod_{i=1}^{n} (t + u  X_i + v  X_i^{-1}+w)  \frac{\asym_{X_1,\ldots,X_n} \left[ \prod_{1 \le i < j \le n} \left(t +  u  \e_{k_i} + v  \e_{k_j}^{-1} + w  \e_{k_i} \e_{k_j}^{-1} \right) \prod_{i=1}^{n} X_i^{k_i+i-1} \right] }{\prod_{1 \le i < j \le n} (X_j - X_i)} \notag \\ 
\quad &= \prod_{i=1}^{n} (t + u  X_i + v  X_i^{-1}+w) \frac{\asym_{X_1,\ldots,X_n} \left[ \prod_{1 \le i <  j \le n} \left(t +  u X_i  + v  X_j^{-1} + w  X_i X_j^{-1} \right) \prod_{i=1}^{n} X_i^{k_i+i-1} \right] }{\prod_{1 \le i < j \le n} (X_j - X_i)} \notag \\
\quad &=  \frac{\asym_{X_1,\ldots,X_n} \left[ \prod_{1 \le i \le j \le n} \left(v  + w X_i + t X_j + u X_i X_j  \right) \prod_{i=1}^{n} X_i^{k_i-1} \right] }{\prod_{1 \le i < j \le n} (X_j - X_i)}
\end{align*} 
and the assertion follows.
\end{proof} 

\begin{remark} 
Suppose $(k_1-1,k_2-1,\ldots,k_n-1)$ is a partition (allowing zero parts) then, when setting $u=v=0$, $w=1$ and replacing $t$ by $-t$ \eqref{asym}, we obtain the Hall-Littlewood polynomials \cite{macdonald} up to a factor that is a rational function in $t$. 
\end{remark}

\subsection{Generating function with respect to a Schur polynomial weight}

We are now ready to obtain our first interpretation. Multiplying \eqref{littlewoodASM2} and \eqref{rincluded} with $\prod_{i=1}^{n} (X_i^{-1} + 1+w + X_i)$ gives 
\begin{multline} 
\label{littlewoodASM3} 
\frac{\asym_{X_1,\ldots,X_n} \left[ \prod_{1 \le i \le  j \le n} (1+ w X_i + X_j + X_i X_j) \sum_{0 \le k_1 < k_2 < \ldots < k_n} X_1^{k_1-1} X_2^{k_2-1} \cdots X_n^{k_n-1} \right]}{\prod_{1 \le i <  j \le n} (X_j-X_i)} \\ = 
\prod_{i=1}^{n} (X_i^{-1} + 1+w + X_i) \prod_{i=1}^{n} \frac{1}{1-X_i} \prod_{1 \le i < j \le n}  \frac{1+X_i + X_j + w X_i X_j}{1-X_i X_j},
\end{multline} 
and 
\begin{multline} 
\label{rincluded1} 
\frac{\asym_{X_1,\ldots,X_n} \left[  \prod_{1 \le i \le j \le n} (1+w X_i+X_j + X_i X_j) 
\sum_{0 \le k_1 < k_2 < \ldots < k_n \le m} X_1^{k_1-1} X_2^{k_2-1} \cdots X_n^{k_n-1} \right]}{\prod_{1 \le i <  j \le n} (X_j-X_i)}  \\ = 
\prod_{i=1}^{n} (X_i^{-1} + 1+w + X_i) \\ \times
\frac{\det_{1 \le i, j \le n} \left( X_i^{j-1}  (1+X_i)^{j-1} (1+ w X_i)^{n-j}  - X_i^{m+2n-j} (1+X_i^{-1})^{j-1}  (1+w X_i^{-1})^{n-j} \right)}{\prod\limits_{i=1}^n (1-X_i) \prod\limits_{1 \le i < j \le n} (1-X_i X_j)(X_j-X_i)}, 
\end{multline} 
respectively, and we can now interpret the left-hand sides as the generating function of arrowed Gelfand-Tsetlin patterns with non-negative strictly increasing bottom row, where we need to specialize $t=u=v=1$ in the weight and in the second case the entries in 
the bottom row are less than or equal to $m$.

\begin{remark}
\begin{enumerate} 
\item 
For $\mathbf{X}=(X_1,\ldots,X_n)$, let $\mathcal{AGTP}(t,u,v,w;\mathbf{k};\mathbf{X})$ denote the generating function of arrowed Gelfand-Tsetlin patterns with bottom row $\mathbf{k}=(k_1,\ldots,k_n)$. Then, using \eqref{asym}, it follows by changing 
$(X_1,\ldots,X_n)$ to $(X_n,X_{n-1},\ldots,X_1)$ that 
$$
\mathcal{AGTP}(t,u,v,w;\mathbf{k};\mathbf{X}) = (-1)^{\binom{n}{2}} 
\mathcal{AGTP}(w,u,v,t;\overline{\mathbf{k}};\mathbf{X}),
$$
where $\overline{\mathbf{k}}=(k_n,\ldots,k_1)$. Therefore, the left-hand sides are up to the sign $(-1)^{\binom{n}{2}}$ also the generating function of AGTPs with strictly \emph{decreasing} bottom row of non-negative integers, where we need to set $u=v=w=1$ and replace $t$ by $w$ in the weight, and, in the case of \eqref{rincluded}, the entries in the bottom row are less than or equal to $m$. 
\item For the case $t=0$, there is worked out a possibility in \cite{nASMDPP} 
to get around the multiplication with the extra factor $\prod_{i=1}^{n} (X_i^{-1} + 1+w + X_i)$ by working with ``down arrows'' as decorations. In our application, this can be used in combination with our second combinatorial interpretation concerning AGTPs with strictly decreasing bottom row to give combinatorial interpretations of the left-hand sides of \eqref{littlewoodASM2} and \eqref{rincluded} in the special case $w=0$. 
It is an open problem to explore whether the down-arrowed array can be extended to general $t$.
\end{enumerate}
\end{remark} 

In Appendix~\ref{furtherLHS}, we develop some other (maybe less interesting) combinatorial interpretations of the left-hand sides, which we include for the sake of completeness.

\section{Combinatorial interpretations of the right-hand sides of \eqref{littlewoodASM3} and  
\eqref{rincluded1}} 
\label{RHS} 

\subsection{Right-hand side of \eqref{littlewoodASM3}} 
 For the right-hand side of \eqref{littlewoodASM3}, which is
\begin{equation} 
\label{RHSsimple} 
 \prod_{i=1}^{n} \frac{X_i^{-1}+1+w+X_i}{1-X_i} \prod_{1 \le i < j \le n}  \frac{1+X_i + X_j + w X_i X_j}{1-X_i X_j},
\end{equation} 
it is straightforward to give a combinatorial interpretation as a generating function.
Recall that, in the ordinary case \eqref{littlewood}, the right-hand side $\prod_{i=1}^{n} \frac{1}{1-X_i} \prod_{1 \le i < j \le n}  \frac{1}{1-X_i X_j}$ is interpreted as two-line arrays with entries in $\{1,2,\ldots,n\}$, ordered lexicographically, with the top element of each  
column being greater than or equal to its bottom element. The exponent of $X_i$ in the weight is computed by subtracting from the total number of $i$'s in the two-line array the number of columns with $i$ as top and bottom element. 

To extend this to an interpretation of 
\eqref{RHSsimple}, we have one additional column $\binom{j}{i}$ for all pairs $i \le j$, which are either overlined, underlined, both or neither. 
An overlined column $\binom{j}{i}$ with $i<j$, contributes an additional multiplicative $X_j$ to the weight, while an underlined column with $i$ as bottom element contributes an additional $X_i$, and if a column is overlined and underlined then such a column contributes, in addition to $X_i X_j$, $w$. Moreover, an overlined column $\binom{i}{i}$ contributes an additional $X_i$ to the weight and if it is underlined then it contributes $X_i^{-1}$ to the weight, and, again, if the column is overlined and underlined, then it contributes also $w$. In both cases, if the column is neither underlined nor overlined, it contributes nothing in addition.
 
\subsection{Right-hand side of \eqref{rincluded1}}

The following theorem provides an interpretation of the right-hand side of  \eqref{rincluded1} as a weighted count of (partly non-intersecting) lattice paths. This right-hand side differs from the right-hand 
side of \eqref{rincluded} by a simple multiplicative factor. We work as long as possible with general $w$, however, it will turn out that we need to specialize to $w=0,1$ at some point to obtain a nicer interpretation. We present two different proofs to obtain the result, where the second one is only sketched. 

Figure~\ref{examplepaths} seeks to illustrate the theorem in the case that $m$ is odd.

\begin{theorem}
\label{RHScomplicated}
(1) Assume that $m=2l+1$. Then the right-hand side of \eqref{rincluded1} has the following interpretation as weighted count of families of $n$ lattice paths. 
\begin{itemize} 
\item The $i$-th lattice path starts in one point in the set $A_i=\{(-3i+1,-i+1),(-i+1,-3i+1)\}$, $i=1,2,\ldots,n$, and the end points of the paths are $E_j=(n-j+l+1,j-l-2)$, $j=1,2,\ldots,n$. 
\item Below and on the line $x+y=0$, the step set is $\{(1,1),(-1,1)\}$ for steps that start in $(-3i+1,-i+1)$ 
and it is $\{(1,1),(1,-1)\}$ for steps that start in $(-i+1,-3i+1)$. Steps of type $(-1,1)$ and $(1,-1)$ with distance 
$0,2,4,\ldots$ from $x+y=0$ are equipped with the weights $X_1,X_2,X_3,\ldots$, respectively, while such steps
with 
distance $1,3,5,\ldots$ are equipped with the weights $X_1^{-1},X_2^{-1},X_3^{-1},\ldots$, respectively. 
\item Above the line 
$x+y=0$, the step set is $\{(1,0),(0,1)\}$. 
 Above the line $x+y=j-1$, horizontal steps of the path that ends in $E_j$ are 
equipped with the weight $w$.
\item The paths can be assumed to be non-intersecting below the line $x+y=0$. In case $w=1$, we can also assume them 
to be non-intersecting above the line $x+y=0$. In case $w=0$, $E_j$ can be replaced by $E'_j=(n-j+l+1,2j-n-l-2)$, $j=1,2,\ldots,n$, and then we can also assume the paths to be non-intersecting above the line $x+y=0$. 
\item The sign of family of paths is the sign of the permutation $\sigma$ with the property that the $i$-th path connects $A_i$ to $E_{\sigma(i)}$ with an extra contribution of $-1$ if we choose $(-i+1,-3i+1)$ from $A_i$. Moreover, we have an overall factor of 
$$
(-1)^{\binom{n+1}{2}} \prod_{i=1}^{n} X_i^{l} (X_i^{-1}+1+w + X_i)(1+X_i).
$$
\item In case $w=0,1$, when restricting to non-intersecting paths, let $1 \le i_1 < i_2,\ldots < i_m < n$ be the indices for which we chose 
$(-3i+1,-i+1)$ from $A_i$. Then the sign can assumed to be $(-1)^{i_1+\ldots+i_m}$ and the overall factor is
$$
\prod_{i=1}^{n} X_i^{l} (X_i^{-1}+1+w + X_i)(1+X_i).
$$
\end{itemize} 
(2) Assume that $m=2l$. Then, to obtain an interpretation for the right-hand side of \eqref{rincluded1}, we only need to replace $E_j$ by a set of two possible endpoints $E_j=\{(n-j+l+1,j-l-2),(n-j+l,j-l-1)\}$. The overall factor is 
$$
(-1)^{\binom{n+1}{2}} \prod_{i=1}^{n} X_i^{l} (X_i^{-1}+1+w + X_i)
$$
in the case when we do not specialize $w$. 
The endpoints are replaced by $E'_j=\{(n-j+l+1,2j-n-l-2),(n-j+l,2j-n-l-1)\}$ if $w=0$. 
In case $w=0,1$ if we restrict to non-intersecting paths and the sign is taken care of as above, then the overall factor is 
$$
\prod_{i=1}^{n} X_i^{l} (X_i^{-1}+1+w + X_i).
$$
\end{theorem} 

We discuss the weight and the sign on the example in Figure~\ref{examplepaths}. The weights that come from the 
individual paths are 
$$
X_1^{-1} \cdot X_1^{-1} \cdot X_2 \cdot X_1 X_3 \cdot X_2 X_3^{-1} \cdot X_5^{-2},  
$$
where the factors are arranged in a manner that the $i$-th factor is the weight of the path that starts in the set 
$A_i$. To compute the sign, observe that $\sigma = (6 \, 5 \, 4 \, 3 \, 2 \, 1)$ in one-line notation so that 
$\sgn \sigma = -1$ and that we choose the second starting point in $A_i$ except for $i=1$, so that the total sign 
is $(-1)\cdot(-1)^5=1$.


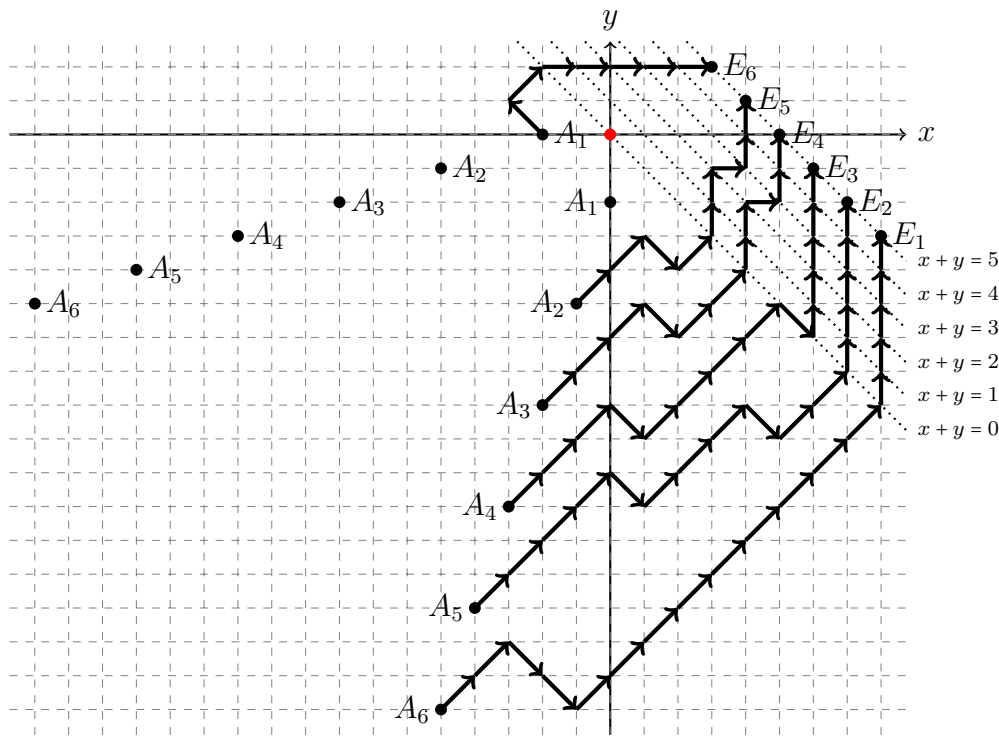
\begin{figure}[htb]
	\centering
	\begin{tikzpicture}[scale=.45,baseline=(current bounding box.center)]	
	        \draw[->,thick] (-17.75,0)--(8.75,0) node[right]{$x$};
		\draw[->,thick] (0,-17.75)--(0,2.75) node[above]{$y$};
		\draw[dotted,thick] (-2.75,2.75)--(8.75,-8.75) node[right]{\tiny $x+y=0$};
		\draw[dotted,thick] (-1.75,2.75)--(8.75,-7.75) node[right]{\tiny $x+y=1$};
		\draw[dotted,thick] (-0.75,2.75)--(8.75,-6.75) node[right]{\tiny $x+y=2$};
		\draw[dotted,thick] (0.25,2.75)--(8.75,-5.75) node[right]{\tiny $x+y=3$};
		\draw[dotted,thick] (1.25,2.75)--(8.75,-4.75) node[right]{\tiny $x+y=4$};
		\draw[dotted,thick] (2.25,2.75)--(8.75,-3.75) node[right]{\tiny $x+y=5$};
	         \draw [help lines,step=1cm,dashed] (-17.75,-17.75) grid (8.75,2.75);
		\fill[red] (0,0) circle (5pt);
		\fill (-2,0) circle (5pt) node[right]{$A_1$};
		\fill (-5,-1) circle (5pt) node[right]{$A_2$};
		\fill (-8,-2) circle (5pt) node[right]{$A_3$};
		\fill (-11,-3) circle (5pt) node[right]{$A_4$};
		\fill (-14,-4) circle (5pt) node[right]{$A_5$};
		\fill (-17,-5) circle (5pt) node[right]{$A_6$};
		
		\fill (0,-2) circle (5pt) node[left]{$A_1$};
		\fill (-1,-5) circle (5pt) node[left]{$A_2$};
		\fill (-2,-8) circle (5pt) node[left]{$A_3$};
		\fill (-3,-11) circle (5pt) node[left]{$A_4$};
		\fill (-4,-14) circle (5pt) node[left]{$A_5$};
		\fill (-5,-17) circle (5pt) node[left]{$A_6$};
		
		\fill (8,-3) circle (5pt) node[right]{$E_1$};
		\fill (7,-2) circle (5pt) node[right]{$E_2$};
		\fill (6,-1) circle (5pt) node[right]{$E_3$};
		\fill (5,0) circle (5pt) node[right]{$E_4$};
		\fill (4,1) circle (5pt) node[right]{$E_5$};
		\fill (3,2) circle (5pt) node[right]{$E_6$};
		
		\path[decoration=arrows, decorate] (-2,0) --++ (-1,1) --++ (1,1) 
		--++ (1,0) --++ (1,0) --++ (1,0) --++ (1,0) --++ (1,0);
		
		\path[decoration=arrows, decorate] (-1,-5) --++ (1,1) --++ (1,1) --++ (1,-1) --++ (1,1)
		--++ (0,1) --++ (0,1) --++ (1,0) --++ (0,1) --++ (0,1);
		
		\path[decoration=arrows, decorate] (-2,-8) --++ (1,1) --++ (1,1) --++ (1,1) --++ (1,-1)
		--++ (1,1) --++ (1,1) --++ (0,1) --++ (0,1) --++ (1,0)--++(0,1)--++(0,1);
		
		\path[decoration=arrows, decorate] (-3,-11) --++ (1,1) --++ (1,1) --++ (1,1) --++ (1,-1)
		--++ (1,1) --++ (1,1) --++ (1,1) --++ (1,1)--++ (1,-1) --++ (0,1) --++ (0,1)--++(0,1)--++(0,1)
		--++(0,1);

		\path[decoration=arrows, decorate] (-4,-14) --++ (1,1) --++ (1,1) --++ (1,1) --++ (1,1)
		--++ (1,-1) --++ (1,1) --++ (1,1) --++ (1,1)--++ (1,-1)  --++ (1,1) --++ (1,1)--++ (0,1) --++ (0,1)--++(0,1)
		--++(0,1)--++(0,1);
		
		\path[decoration=arrows, decorate] (-5,-17) --++ (1,1) --++ (1,1) --++ (1,-1) --++ (1,-1)
		--++ (1,1) --++ (1,1) --++ (1,1) --++ (1,1)--++ (1,1)  --++ (1,1) --++ (1,1)--++ (1,1) --++ (1,1)
		--++ (0,1) --++ (0,1)--++(0,1)
		--++(0,1)--++(0,1);
				
	\end{tikzpicture}
	\caption{\label{examplepaths} An example of families of lattice paths in Theorem~\ref{RHScomplicated}.}
\end{figure}

In the case that $m$ is odd, we always need to choose the second lattice point in $A_i$ if $l \ge n-2$ because then all $E_i$ have a non-positive $y$-coordinate and this implies that they cannot be reached by any of the first 
lattice points in $A_i$ since any lattice path starting from the first lattice point in $A_i$ intersects the line $x+y=0$ in a lattice point with positive $y$-coordinate. This implies that, in the non-intersecting case, the sign is always $1$. In the case that $m$ is even, the condition is $l \ge n-1$. 

In theses cases and when we have in addition $w=0$, we can translate the lattice paths easily into pairs of plane partitions. 
The case $m=2l+1$ is illustrated in Figure~\ref{odd1}, while the case $m=2l$ is illustrated in Figure~\ref{even}. A similar 
result can in principal be derived for the case $w=1$, but we omit this here.

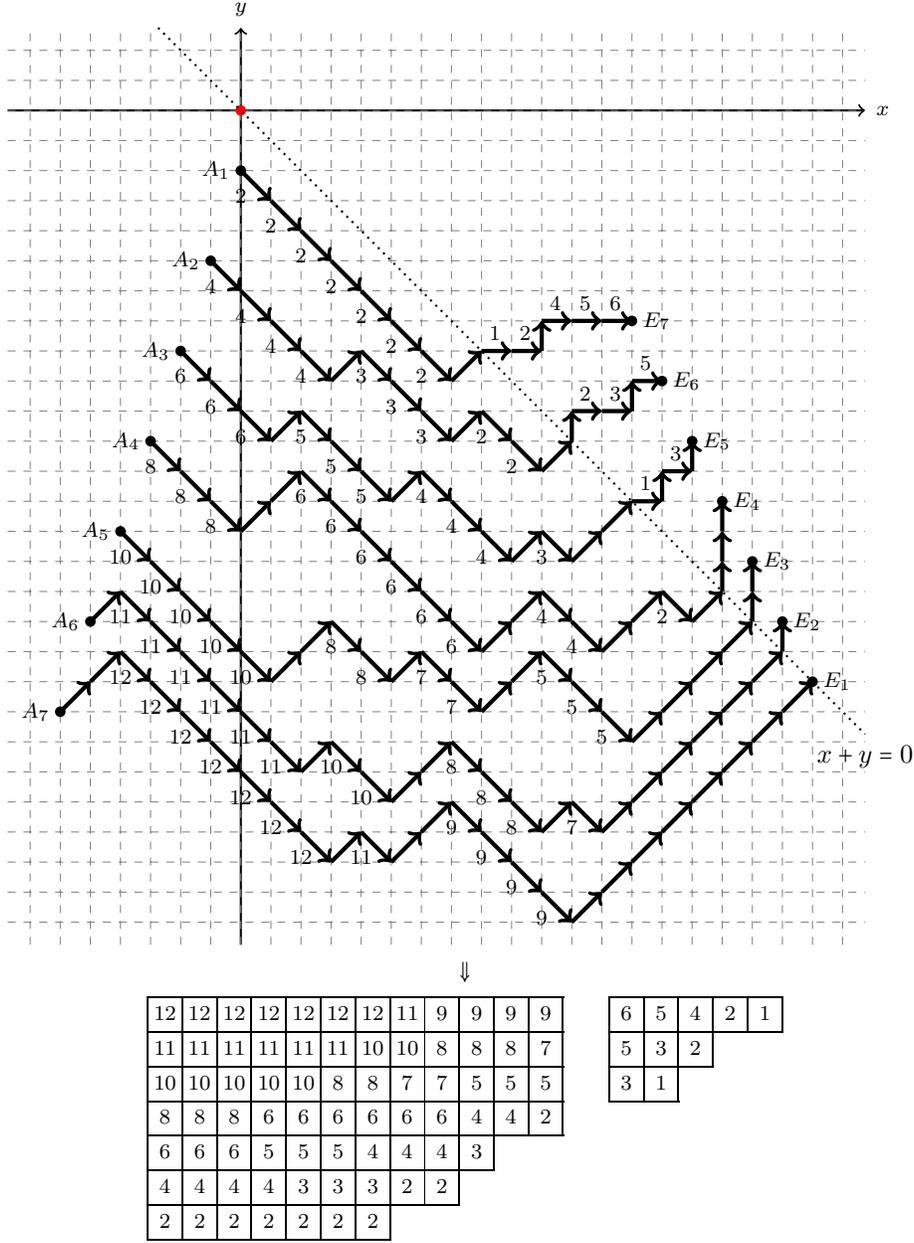
\begin{figure}
\begin{tikzpicture}[scale=.40,baseline=(current bounding box.center)]	
	        \draw[->,thick] (-7.75,0)--(20.75,0) node[right]{\tiny$x$};
		\draw[->,thick] (0,-27.75)--(0,2.75) node[above]{\tiny$y$};
	         \draw [help lines,step=1cm,dashed] (-7.75,-27.75) grid (20.75,2.75);
		\fill[red] (0,0) circle (5pt);
		\draw[dotted,thick] (-2.75,2.75)--(20.75,-20.75) node[below]{\footnotesize $x+y=0$};
		\fill (0,-2) circle (5pt) node[left]{\tiny$A_1$};
		\fill (-1,-5) circle (5pt) node[left]{\tiny$A_2$};
		\fill (-2,-8) circle (5pt) node[left]{\tiny$A_3$};
		\fill (-3,-11) circle (5pt) node[left]{\tiny$A_4$};
		\fill (-4,-14) circle (5pt) node[left]{\tiny$A_5$};
		\fill (-5,-17) circle (5pt) node[left]{\tiny$A_6$};
		\fill (-6,-20) circle (5pt) node[left]{\tiny$A_7$};
		
		\fill (19,-19) circle (5pt) node[right]{\tiny$E_1$};
		\fill (18,-17) circle (5pt) node[right]{\tiny$E_2$};
		\fill (17,-15) circle (5pt) node[right]{\tiny$E_3$};
		\fill (16,-13) circle (5pt) node[right]{\tiny$E_4$};
		\fill (15,-11) circle (5pt) node[right]{\tiny$E_5$};
		\fill (14,-9) circle (5pt) node[right]{\tiny$E_6$};
		\fill (13,-7) circle (5pt) node[right]{\tiny$E_7$};
		
		\path[decoration=arrows, decorate] (0,-2)node[below=0.1cm]{\tiny$2$} --++ (1,-1)node[below=0.1cm]{\tiny$2$} --++ 
		(1,-1)node[below=0.1cm]{\tiny$2$} --++ (1,-1)node[below=0.1cm]{\tiny$2$} --++ (1,-1)node[below=0.1cm]{\tiny$2$}
		--++ (1,-1)node[below=0.1cm]{\tiny$2$} --++ (1,-1)node[below=0.1cm]{\tiny$2$} --++ (1,-1)--++(1,1)
		--++ (1,0)node[above left]{\tiny$1$}--++
		(1,0)node[above left]{\tiny$2$}
		--++(0,1)--++
		(1,0) node[above left]{\tiny$4$}--++(1,0)
		node[above left]{\tiny$5$}
		--++(1,0) node[above left]{\tiny$6$};
		
		\path[decoration=arrows, decorate] (-1,-5)node[below=0.1cm]{\tiny$4$} --++ (1,-1)node[below=0.1cm]{\tiny$4$}
		--++ (1,-1)node[below=0.1cm]{\tiny$4$}--++ (1,-1)node[below=0.1cm]{\tiny$4$} --++ (1,-1)
		--++ (1,1)node[below=0.1cm]{\tiny$3$} --++ (1,-1)node[below=0.1cm]{\tiny$3$} --++ (1,-1)node[below=0.1cm]{\tiny$3$} 
		--++(1,-1) --++ (1,1)node[below=0.1cm]{\tiny$2$}--++(1,-1)node[below=0.1cm]{\tiny$2$}--++(1,-1)--++(1,1)--++(0,1)
		--++(1,0)node[above left]{\tiny$2$}--++(1,0)node[above left]{\tiny$3$}--++(0,1)--++(1,0)node[above left]{\tiny$5$};
		
		\path[decoration=arrows, decorate] (-2,-8)node[below=0.1cm]{\tiny$6$} --++ (1,-1)node[below=0.1cm]{\tiny$6$} --++ (1,-1)node[below=0.1cm]  {\tiny$6$}--++ (1,-1)--++ (1,1)node[below=0.1cm]{\tiny$5$}
		--++ (1,-1)node[below=0.1cm]{\tiny$5$} --++ (1,-1)node[below=0.1cm]{\tiny$5$} --++ (1,-1)--++(1,1)node[below=0.1cm]{\tiny$4$} --++ (1,-1)node[below=0.1cm]{\tiny$4$}--++(1,-1)node[below=0.1cm]{\tiny$4$}--++(1,-1)--++(1,1)node[below=0.1cm]{\tiny$3$}--++
		(1,-1)--++(1,1)--++(1,1)--++(1,0)node[above left]{\tiny$1$}
		--++(0,1)--++(1,0)node[above left]{\tiny$3$}
		--++(0,1);
		
		\path[decoration=arrows, decorate] (-3,-11)node[below=0.1cm]{\tiny$8$}  --++ (1,-1)node[below=0.1cm]{\tiny$8$} --++ (1,-1)node[below=0.1cm]{\tiny$8$} --++ (1,-1)--++ (1,1)
		--++ (1,1)node[below=0.1cm]{\tiny$6$} --++ (1,-1)node[below=0.1cm]{\tiny$6$} --++ (1,-1)node[below=0.1cm]{\tiny$6$}--++(1,-1)node[below=0.1cm]{\tiny$6$} --++ (1,-1)node[below=0.1cm]{\tiny$6$}--++(1,-1)node[below=0.1cm]{\tiny$6$}--++(1,-1)--++(1,1)--++
		(1,1)node[below=0.1cm]{\tiny$4$}--++(1,-1)node[below=0.1cm]{\tiny$4$}--++(1,-1)--++(1,1)--++(1,1)node[below=0.1cm]{\tiny$2$}
		--++(1,-1)--++(1,1)--++(0,1)--++(0,1)--++(0,1);
				
		\path[decoration=arrows, decorate] (-4,-14)node[below=0.1cm]{\tiny$10$} --++(1,-1)node[below=0.1cm]{\tiny$10$}--++(1,-1)node[below=0.1cm]  {\tiny$10$}--++(1,-1)node[below=0.1cm]{\tiny$10$}--++(1,-1)node[below=0.1cm]{\tiny$10$}--++
		(1,-1)--++(1,1)--++(1,1)node[below=0.1cm]{\tiny$8$}--++(1,-1)node[below=0.1cm]{\tiny$8$}--++(1,-1)--++(1,1)node[below=0.1cm]{\tiny$7$}--++(1,-1)node[below=0.1cm]{\tiny$7$}--++(1,-1)--++(1,1)--++(1,1)node[below=0.1cm]{\tiny$5$}
		--++(1,-1)node[below=0.1cm]{\tiny$5$}--++(1,-1)node[below=0.1cm]{\tiny$5$}--++(1,-1)--++(1,1)--++(1,1)--++(1,1)--++(1,1)--++(0,1)--++(0,1);
		
		\path[decoration=arrows, decorate] (-5,-17)--++(1,1)node[below=0.1cm]{\tiny$11$}--++(1,-1)node[below=0.1cm]{\tiny$11$}--++
		(1,-1)node[below=0.1cm]{\tiny$11$}--++(1,-1)node[below=0.1cm]{\tiny$11$}--++
		(1,-1)node[below=0.1cm]{\tiny$11$}--++(1,-1)node[below=0.1cm]{\tiny$11$}--++(1,-1)--++(1,1)node[below=0.1cm]{\tiny$10$}--++(1,-1)node[below=0.1cm]{\tiny$10$}--++(1,-1)--++(1,1)--++(1,1)node[below=0.1cm]{\tiny$8$}--++(1,-1)node[below=0.1cm]{\tiny$8$}--++
		(1,-1)node[below=0.1cm]{\tiny$8$}--++(1,-1)--++(1,1)node[below=0.1cm]{\tiny$7$}--++(1,-1)--++(1,1)--++(1,1)--++(1,1)--++(1,1)--++(1,1)--++(1,1)
		--++(0,1);
          
                \path[decoration=arrows, decorate] (-6,-20)--++(1,1)--++(1,1)node[below=0.1cm]{\tiny$12$}--++(1,-1)node[below=0.1cm]{\tiny$12$}--++(1,-1)node[below=0.1cm]{\tiny$12$}--++(1,-1)node[below=0.1cm]{\tiny$12$}
                --++(1,-1)node[below=0.1cm]{\tiny$12$}--++(1,-1)node[below=0.1cm]{\tiny$12$}--++(1,-1)node[below=0.1cm]{\tiny$12$}--++(1,-1)--++(1,1)node[below=0.1cm]{\tiny$11$}--++(1,-1)--++(1,1)--++(1,1)node[below=0.1cm]{\tiny$9$}--++(1,-1)node[below=0.1cm]{\tiny$9$}--++
                (1,-1)node[below=0.1cm]{\tiny$9$}--++(1,-1)node[below=0.1cm]{\tiny$9$}--++(1,-1)--++(1,1)--++(1,1)--++(1,1)--++(1,1)--++(1,1)--++(1,1)--++(1,1)--++(1,1);
\end{tikzpicture}
\tiny
$$
\Downarrow
$$
$$
\begin{ytableau}
12 & 12 & 12 & 12 & 12 & 12 & 12 & 11 & 9 & 9 & 9 & 9 \\ 
11 & 11 & 11 & 11 & 11 & 11 & 10 & 10 & 8 & 8 & 8 & 7 \\
10 & 10 & 10 & 10 & 10 & 8 & 8 & 7 & 7 & 5 & 5 & 5 \\
8 & 8 & 8 & 6 & 6 & 6 & 6 & 6 & 6 & 4 & 4 & 2 \\
6 & 6 & 6 & 5 & 5 & 5 & 4 & 4 & 4 & 3 \\
4 & 4 & 4 & 4 & 3 & 3 & 3 & 2 & 2 \\
2 & 2 & 2 & 2 & 2 & 2 & 2 
\end{ytableau} \qquad 
\begin{ytableau}
6 & 5 & 4 & 2 & 1  \\
5 & 3 & 2  \\
3 & 1 
\end{ytableau}	
$$
\caption{\label{odd1} Illustration of Corollary~\ref{signless} (1) for $n=7$ and $l=12$.}   
\end{figure} 

\begin{figure}
\begin{tikzpicture}[scale=.40,baseline=(current bounding box.center)]	
	        \draw[->,thick] (-7.75,0)--(20.75,0) node[right]{\tiny$x$};
		\draw[->,thick] (0,-27.75)--(0,2.75) node[above]{\tiny$y$};
	         \draw [help lines,step=1cm,dashed] (-7.75,-27.75) grid (20.75,2.75);
		\fill[red] (0,0) circle (5pt);
		\draw[dotted,thick] (-2.75,2.75)--(20.75,-20.75) node[below]{\footnotesize $x+y=0$};
		\fill (0,-2) circle (5pt) node[left]{\tiny$A_1$};
		\fill (-1,-5) circle (5pt) node[left]{\tiny$A_2$};
		\fill (-2,-8) circle (5pt) node[left]{\tiny$A_3$};
		\fill (-3,-11) circle (5pt) node[left]{\tiny$A_4$};
		\fill (-4,-14) circle (5pt) node[left]{\tiny$A_5$};
		\fill (-5,-17) circle (5pt) node[left]{\tiny$A_6$};
		\fill (-6,-20) circle (5pt) node[left]{\tiny$A_7$};
		
		\fill (20,-20) circle (5pt) node[right]{\tiny$E_1$};
		\fill (19,-18) circle (5pt) node[right]{\tiny$E_2$};
		\fill (18,-16) circle (5pt) node[right]{\tiny$E_3$};
		\fill (16,-13) circle (5pt) node[right]{\tiny$E_4$};
		\fill (15,-11) circle (5pt) node[right]{\tiny$E_5$};
		\fill (14,-9) circle (5pt) node[right]{\tiny$E_6$};
		\fill (13,-7) circle (5pt) node[right]{\tiny$E_7$};
		
		\path[decoration=arrows, decorate] (0,-2)node[below=0.1cm]{\tiny$2$} --++ (1,-1)node[below=0.1cm]{\tiny$2$} --++ 
		(1,-1)node[below=0.1cm]{\tiny$2$} --++ (1,-1)node[below=0.1cm]{\tiny$2$} --++ (1,-1)node[below=0.1cm]{\tiny$2$}
		--++ (1,-1)node[below=0.1cm]{\tiny$2$} --++ (1,-1)node[below=0.1cm]{\tiny$2$} --++ (1,-1)--++(1,1)
		--++ (1,0)node[above left]{\tiny$1$}--++
		(1,0)node[above left]{\tiny$2$}
		--++(0,1)--++
		(1,0) node[above left]{\tiny$4$}--++(1,0)
		node[above left]{\tiny$5$}
		--++(1,0) node[above left]{\tiny$6$};
		
		\path[decoration=arrows, decorate] (-1,-5)node[below=0.1cm]{\tiny$4$} --++ (1,-1)node[below=0.1cm]{\tiny$4$}
		--++ (1,-1)node[below=0.1cm]{\tiny$4$}--++ (1,-1)node[below=0.1cm]{\tiny$4$} --++ (1,-1)
		--++ (1,1)node[below=0.1cm]{\tiny$3$} --++ (1,-1)node[below=0.1cm]{\tiny$3$} --++ (1,-1)node[below=0.1cm]{\tiny$3$} 
		--++(1,-1) --++ (1,1)node[below=0.1cm]{\tiny$2$}--++(1,-1)node[below=0.1cm]{\tiny$2$}--++(1,-1)--++(1,1)--++(0,1)
		--++(1,0)node[above left]{\tiny$2$}--++(1,0)node[above left]{\tiny$3$}--++(0,1)--++(1,0)node[above left]{\tiny$5$};
		
		\path[decoration=arrows, decorate] (-2,-8)node[below=0.1cm]{\tiny$6$} --++ (1,-1)node[below=0.1cm]{\tiny$6$} --++ (1,-1)node[below=0.1cm]  {\tiny$6$}--++ (1,-1)--++ (1,1)node[below=0.1cm]{\tiny$5$}
		--++ (1,-1)node[below=0.1cm]{\tiny$5$} --++ (1,-1)node[below=0.1cm]{\tiny$5$} --++ (1,-1)--++(1,1)node[below=0.1cm]{\tiny$4$} --++ (1,-1)node[below=0.1cm]{\tiny$4$}--++(1,-1)node[below=0.1cm]{\tiny$4$}--++(1,-1)--++(1,1)node[below=0.1cm]{\tiny$3$}--++
		(1,-1)--++(1,1)--++(1,1)--++(1,0)node[above left]{\tiny$1$}
		--++(0,1)--++(1,0)node[above left]{\tiny$3$}
		--++(0,1);
		
		\path[decoration=arrows, decorate] (-3,-11)node[below=0.1cm]{\tiny$8$}  --++ (1,-1)node[below=0.1cm]{\tiny$8$} --++ (1,-1)node[below=0.1cm]{\tiny$8$} --++ (1,-1)--++ (1,1)
		--++ (1,1)node[below=0.1cm]{\tiny$6$} --++ (1,-1)node[below=0.1cm]{\tiny$6$} --++ (1,-1)node[below=0.1cm]{\tiny$6$}--++(1,-1)node[below=0.1cm]{\tiny$6$} --++ (1,-1)node[below=0.1cm]{\tiny$6$}--++(1,-1)node[below=0.1cm]{\tiny$6$}--++(1,-1)--++(1,1)--++
		(1,1)node[below=0.1cm]{\tiny$4$}--++(1,-1)node[below=0.1cm]{\tiny$4$}--++(1,-1)--++(1,1)--++(1,1)node[below=0.1cm]{\tiny$2$}
		--++(1,-1)--++(1,1)--++(0,1)--++(0,1)--++(0,1);
				
		\path[decoration=arrows, decorate] (-4,-14)node[below=0.1cm]{\tiny$10$} --++(1,-1)node[below=0.1cm]{\tiny$10$}--++(1,-1)node[below=0.1cm]  {\tiny$10$}--++(1,-1)node[below=0.1cm]{\tiny$10$}--++(1,-1)node[below=0.1cm]{\tiny$10$}--++
		(1,-1)--++(1,1)--++(1,1)node[below=0.1cm]{\tiny$8$}--++(1,-1)node[below=0.1cm]{\tiny$8$}--++(1,-1)--++(1,1)node[below=0.1cm]{\tiny$7$}--++(1,-1)node[below=0.1cm]{\tiny$7$}--++(1,-1)--++(1,1)--++(1,1)node[below=0.1cm]{\tiny$5$}
		--++(1,-1)node[below=0.1cm]{\tiny$5$}--++(1,-1)node[below=0.1cm]{\tiny$5$}--++(1,-1)--++(1,1)--++(1,1)--++(1,1)--++(1,1)--++(0,1)--++(1,0)node[above left]{\tiny$2$};
		
		\path[decoration=arrows, decorate] (-5,-17)--++(1,1)node[below=0.1cm]{\tiny$11$}--++(1,-1)node[below=0.1cm]{\tiny$11$}--++
		(1,-1)node[below=0.1cm]{\tiny$11$}--++(1,-1)node[below=0.1cm]{\tiny$11$}--++
		(1,-1)node[below=0.1cm]{\tiny$11$}--++(1,-1)node[below=0.1cm]{\tiny$11$}--++(1,-1)--++(1,1)node[below=0.1cm]{\tiny$10$}--++(1,-1)node[below=0.1cm]{\tiny$10$}--++(1,-1)--++(1,1)--++(1,1)node[below=0.1cm]{\tiny$8$}--++(1,-1)node[below=0.1cm]{\tiny$8$}--++
		(1,-1)node[below=0.1cm]{\tiny$8$}--++(1,-1)--++(1,1)node[below=0.1cm]{\tiny$7$}--++(1,-1)--++(1,1)--++(1,1)--++(1,1)--++(1,1)--++(1,1)--++(1,1)
		--++(1,0)node[above left]{\tiny$1$};
          
                \path[decoration=arrows, decorate] (-6,-20)--++(1,1)--++(1,1)node[below=0.1cm]{\tiny$12$}--++(1,-1)node[below=0.1cm]{\tiny$12$}--++(1,-1)node[below=0.1cm]{\tiny$12$}--++(1,-1)node[below=0.1cm]{\tiny$12$}
                --++(1,-1)node[below=0.1cm]{\tiny$12$}--++(1,-1)node[below=0.1cm]{\tiny$12$}--++(1,-1)node[below=0.1cm]{\tiny$12$}--++(1,-1)--++(1,1)node[below=0.1cm]{\tiny$11$}--++(1,-1)--++(1,1)--++(1,1)node[below=0.1cm]{\tiny$9$}--++(1,-1)node[below=0.1cm]{\tiny$9$}--++
                (1,-1)node[below=0.1cm]{\tiny$9$}--++(1,-1)node[below=0.1cm]{\tiny$9$}--++(1,-1)--++(1,1)--++(1,1)--++(1,1)--++(1,1)--++(1,1)--++(1,1)--++(1,1)--++(1,1)--++(1,-1);
\end{tikzpicture}
\tiny
$$
\Downarrow
$$
$$
\begin{ytableau}
12 & 12 & 12 & 12 & 12 & 12 & 12 & 11 & 9 & 9 & 9 & 9 \\ 
11 & 11 & 11 & 11 & 11 & 11 & 10 & 10 & 8 & 8 & 8 & 7 \\
10 & 10 & 10 & 10 & 10 & 8 & 8 & 7 & 7 & 5 & 5 & 5 \\
8 & 8 & 8 & 6 & 6 & 6 & 6 & 6 & 6 & 4 & 4 & 2 \\
6 & 6 & 6 & 5 & 5 & 5 & 4 & 4 & 4 & 3 \\
4 & 4 & 4 & 4 & 3 & 3 & 3 & 2 & 2 \\
2 & 2 & 2 & 2 & 2 & 2 & 2 
\end{ytableau} \qquad 
\begin{ytableau}
\none & 6 & 5 & 4 & 2 & 1  \\
\none & 5 & 3 & 2  \\
\none &3 & 1 \\
2 \\
1
\end{ytableau}	
$$
\caption{\label{even} Illustration of Corollary~\ref{signless} (1) for $n=7$ and $l=13$.}   
\end{figure}

\begin{cor} 
\label{signless} 
Let $w=0$. 

(1) Assume that $m=2l+1$. In case $l \ge n-2$, the right hand side of \eqref{rincluded1} is the generating function of plane partitions $(P,Q)$ of shapes $\lambda, \mu$, respectively, where   
$\mu$ is the complement of $\lambda$ in the $n \times l$-rectangle, 
$P$ is a column-strict plane partition such that the entries in the $i$-th row are bounded by 
$2n+2-2i$, and $Q$ is a row-strict plane partition of positive integers such that the entries in the $i$-th row are bounded by $n-i$. The weight is 
$$
\prod_{i=1}^n X_i^l (X_i^{-1} +1 + X_i)(1+X_i) X_i^{\# \text{ of } 2i-1 \text{ in } P} X_i^{- \, \# \text{ of } 2i \text{ in } P}. 
$$ 
(2) Assume that $m=2l$. In case $l \ge n-1$, the right-hand side of \eqref{rincluded1} is the generating function of plane partitions $(P,Q)$ of (straight) shape $\lambda$ and skew shape $\mu$, respectively, such that $\mu$ is 
the complement of $\lambda$ in the $n \times (l-1)$-rectangle after possibly deleting the first column of $\mu$, $P$ is a column strict plane partition such that the entries in the $i$-th row are bounded by $2n+2-2i$ and $Q$ is a row-strict plane partition such that the entries in the $i$-th row are bounded by $n-i$. The weight is 
$$
\prod_{i=1}^n X_i^l (X_i^{-1} +1 + X_i) X_i^{\# \text{ of } 2i-1 \text{ in } P} X_i^{- \, \# \text{ of } 2i \text{ in } P}. 
$$
\end{cor}

\begin{proof}
We consider the case $m$ is odd. 
 Assume that $1 \le k_1 < k_2 < \ldots < k_n$ are chosen such that 
$(k_i,-k_i)$ is the last point in the intersection of the line $x+y=0$ with the path that connects $A_i$ to $E'_{n+1-i}$ when traversing the path from $A_i$ to $E'_{n+1-i}$. Note that the portion of the path from 
$A_i$ to $(k_i,-k_i)$ has $k_i-i$ steps of type $(1,-1)$ and $2i-1$ steps of type $(1,1)$. These portions correspond to the plane partition 
$P$ is follows:
The $i$-th path corresponds to the $(n+1-i)$-th row where the $(1,-1)$-steps correspond to the parts, where we fill the cells in the Ferrers diagram from left to right
when traversing the path from $A_i$ to $(k_i,-k_i)$, and a $(1,-1)$-step at distance $d$ from $x+y=0$ gives the entry $d+1$. It follows that the length of row $i$ is $k_{n+1-i}-n-1+i$ and that the entries in row $i$ are bounded by $2n+2-2i$. 

Now the portion of the path from $(k_i,-k_i)$ to $E'_{n+1-i}$ corresponds to the $i$-th row of the plane partition 
$Q$. More precisely, the horizontal steps correspond to the parts, where we fill the cells in the Ferrers diagram from right to left when 
traversing the path from $(k_i,-k_i)$ to $E'_{n+1-i}$, where the $j$-th step gives the entry  $j$. Note that there are $i-k_i+l$ steps of type $(1,0)$ in this portion, while there are $n-i$ steps in total, so that the length of the $i$-th row is $i-k_i+l$ and the entries in row $i$ are bounded by $n-i$.

The case $m$ is even is very similar and it is therefore omitted here.
\end{proof} 

\begin{remark}
(1) The plane partitions $P$ in the corollary are in easy bijection with symplectic tableaux as defined in \cite[Section~4]{KoiTer90}. Also the weight is up to an overall multiplicative factor essentially just the weight that is used for symplectic tableaux. As a consequence, the corollary can be interpreted as to provide the expansion of the generating function of arrowed Gelfand-Tsetlin into symplectic characters. This is in the vein of main results in \cite{TSPP} and in \cite[Remark 2.6]{nVASMDPP}.

(2) In the case $m$ is odd, the plane partitions $Q$ are in easy bijective correspondence with $2n \times 2n \times 2n$ totally symmetric self-complementary plane partitions. The bijection is provided in \cite[Remark 2.6]{nVASMDPP}. 
In the case $m$ is even, we place the part $n+1-i$ into the cell in the $i$-th row of the inner shape and that way we obtain plane partitions that are in easy bijective correspondence with $(2n+2) \times (2n+2) \times (2n+2)$ totally symmetric self-complementary plane partitions.
\end{remark}

\subsection{The cases $n=2$ and $m=2,3$} 
In this section, we give a list of all objects for the left-hand side and right-hand side of  \eqref{rincluded1} in the case $n=2$ and $m=2,3$. We start with the case that $m=3$, since this is easier on the right-hand side.

Note that $m=3$ implies $l=1$. The arrowed monotone triangles are as follows, using the notation from Section~\ref{LHS}. 

\begin{multline*}
\begin{array}{ccc}
& {LR \atop 0} & \\ 
{L \atop 0} && {LR \atop 1}   
\end{array}, 
\begin{array}{ccc}
& {LR \atop 1} & \\ 
{LR \atop 0} && {R \atop 1}   
\end{array},
\begin{array}{ccc}
& {LR \atop 0} & \\ 
{L \atop 0} && {LR \atop 2}   
\end{array},
\begin{array}{ccc}
& {LR \atop 1} & \\ 
{LR \atop 0} && {LR \atop 2}   
\end{array},
\begin{array}{ccc}
& {LR \atop 2} & \\ 
{LR \atop 0} && {R \atop 2}   
\end{array},
\begin{array}{ccc}
& {LR \atop 0} & \\ 
{L \atop 0} && {LR \atop 3}   
\end{array},\\
\begin{array}{ccc}
& {LR \atop 1} & \\ 
{LR \atop 0} && {LR \atop 3}   
\end{array},
\begin{array}{ccc}
& {LR \atop 2} & \\ 
{LR \atop 0} && {LR \atop 3}   
\end{array},
\begin{array}{ccc}
& {LR \atop 3} & \\ 
{LR \atop 0} && {R \atop 3}   
\end{array},
\begin{array}{ccc}
& {LR \atop 1} & \\ 
{L \atop 1} && {LR \atop 2}   
\end{array},
\begin{array}{ccc}
& {LR \atop 2} & \\ 
{LR \atop 1} && {R \atop 2}   
\end{array},
\begin{array}{ccc}
& {LR \atop 1} & \\ 
{L \atop 1} && {LR \atop 3}   
\end{array}, \\
\begin{array}{ccc}
& {LR \atop 2} & \\ 
{LR \atop 1} && {LR \atop 3}   
\end{array}, 
\begin{array}{ccc}
& {LR \atop 3} & \\ 
{L \atop 1} && {LR \atop 3}   
\end{array}, 
\begin{array}{ccc}
& {LR \atop 2} & \\ 
{L \atop 2} && {LR \atop 3}   
\end{array}, 
\begin{array}{ccc}
& {LR \atop 3} & \\
{LR \atop 2} && {R \atop 3}   
\end{array}
\end{multline*} 
The weights are 
\begin{multline} 
\label{first}
X_2(1+X_2^{-1}), 
X_1(1+X_2),
X_2^2(1+X_2^{-1}),
X_1 X_2(X_2^{-1}+1+w+X_2),
X_1^2(1+X_2),
X_2^3(1+X_2^{-1}),\\
X_1 X_2^2(X_2^{-1}+1+w+X_2),
X_1^2 X_2(X_2^{-1}+1+w+X_2), 
X_1^3(1+X_2), 
X_1 X_2^2(1+X_2^{-1}), 
X_1^2 X_2(1+X_2),\\
X_1 X_2^3 (1+X_2^{-1}),
X_1^2 X_2^2 (X_2^{-1} + 1 + w + X_2),
X_1^3 X_2 (1+X_2),
X_1^2 X_2^3(1+X_2^{-1}),
X_1^3 X_2^2(1+X_2),
\end{multline} 
up to the overall factor $LR(X_1) LR(X_2)$, setting $t=u=v=1$.

The corresponding paths from Theorem~\ref{RHScomplicated} are as follows.

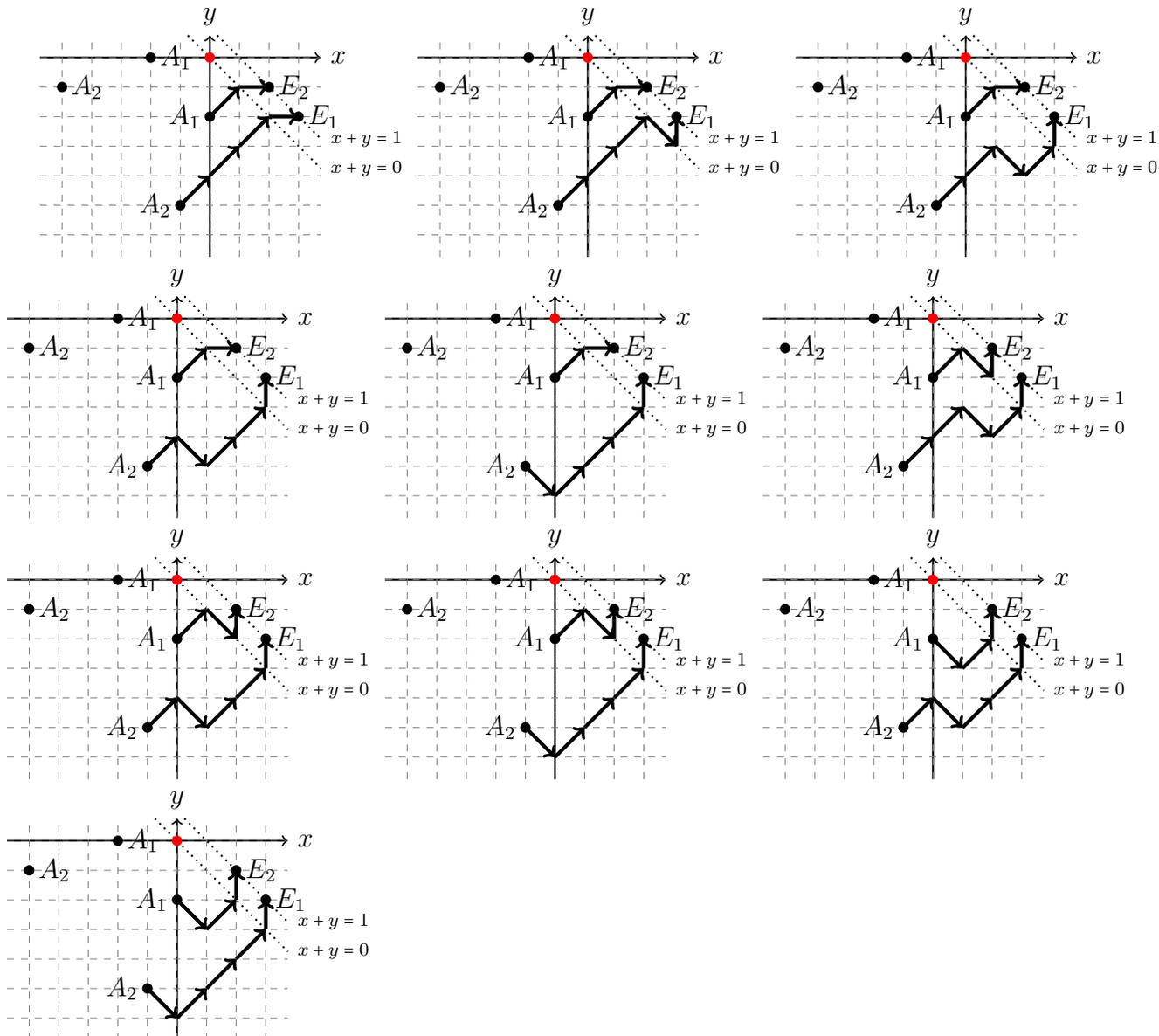
\begin{tikzpicture}[scale=.45,baseline=(current bounding box.center)]	
	        \draw[->,thick] (-5.75,0)--(3.75,0) node[right]{$x$};
		\draw[->,thick] (0,-6.75)--(0,0.75) node[above]{$y$};
		\draw[dotted,thick] (-0.75,0.75)--(3.75,-3.75) node[right]{\tiny $x+y=0$};
		\draw[dotted,thick] (0.25,0.75)--(3.75,-2.75) node[right]{\tiny $x+y=1$};
	         \draw [help lines,step=1cm,dashed] (-5.75,-6.75) grid (3.75,0.75);
		\fill[red] (0,0) circle (5pt);
		\fill (-2,0) circle (5pt) node[right]{$A_1$};
		\fill (-5,-1) circle (5pt) node[right]{$A_2$};

		\fill (0,-2) circle (5pt) node[left]{$A_1$};
		\fill (-1,-5) circle (5pt) node[left]{$A_2$};
		
		\fill (3,-2) circle (5pt) node[right]{$E_1$};
		\fill (2,-1) circle (5pt) node[right]{$E_2$};
		
		\path[decoration=arrows, decorate] (0,-2) --++ (1,1) --++ (1,0);
		\path[decoration=arrows, decorate] (-1,-5) --++ (1,1) --++ (1,1) --++ (1,1) --++(1,0);
\end{tikzpicture}
\begin{tikzpicture}[scale=.45,baseline=(current bounding box.center)]	
	        \draw[->,thick] (-5.75,0)--(3.75,0) node[right]{$x$};
		\draw[->,thick] (0,-6.75)--(0,0.75) node[above]{$y$};
		\draw[dotted,thick] (-0.75,0.75)--(3.75,-3.75) node[right]{\tiny $x+y=0$};
		\draw[dotted,thick] (0.25,0.75)--(3.75,-2.75) node[right]{\tiny $x+y=1$};
	         \draw [help lines,step=1cm,dashed] (-5.75,-6.75) grid (3.75,0.75);
		\fill[red] (0,0) circle (5pt);
		\fill (-2,0) circle (5pt) node[right]{$A_1$};
		\fill (-5,-1) circle (5pt) node[right]{$A_2$};

		\fill (0,-2) circle (5pt) node[left]{$A_1$};
		\fill (-1,-5) circle (5pt) node[left]{$A_2$};
		
		\fill (3,-2) circle (5pt) node[right]{$E_1$};
		\fill (2,-1) circle (5pt) node[right]{$E_2$};
		
		\path[decoration=arrows, decorate] (0,-2) --++ (1,1) --++ (1,0);
		\path[decoration=arrows, decorate] (-1,-5) --++ (1,1) --++ (1,1) --++ (1,1) --++(1,-1)--++(0,1);
\end{tikzpicture}
\begin{tikzpicture}[scale=.45,baseline=(current bounding box.center)]	
	        \draw[->,thick] (-5.75,0)--(3.75,0) node[right]{$x$};
		\draw[->,thick] (0,-6.75)--(0,0.75) node[above]{$y$};
		\draw[dotted,thick] (-0.75,0.75)--(3.75,-3.75) node[right]{\tiny $x+y=0$};
		\draw[dotted,thick] (0.25,0.75)--(3.75,-2.75) node[right]{\tiny $x+y=1$};
	         \draw [help lines,step=1cm,dashed] (-5.75,-6.75) grid (3.75,0.75);
		\fill[red] (0,0) circle (5pt);
		\fill (-2,0) circle (5pt) node[right]{$A_1$};
		\fill (-5,-1) circle (5pt) node[right]{$A_2$};

		\fill (0,-2) circle (5pt) node[left]{$A_1$};
		\fill (-1,-5) circle (5pt) node[left]{$A_2$};
		
		\fill (3,-2) circle (5pt) node[right]{$E_1$};
		\fill (2,-1) circle (5pt) node[right]{$E_2$};
		
		\path[decoration=arrows, decorate] (0,-2) --++ (1,1) --++ (1,0);
		\path[decoration=arrows, decorate] (-1,-5) --++ (1,1) --++ (1,1) --++(1,-1)--++(1,1)--++(0,1);
\end{tikzpicture}
\begin{tikzpicture}[scale=.45,baseline=(current bounding box.center)]	
	        \draw[->,thick] (-5.75,0)--(3.75,0) node[right]{$x$};
		\draw[->,thick] (0,-6.75)--(0,0.75) node[above]{$y$};
		\draw[dotted,thick] (-0.75,0.75)--(3.75,-3.75) node[right]{\tiny $x+y=0$};
		\draw[dotted,thick] (0.25,0.75)--(3.75,-2.75) node[right]{\tiny $x+y=1$};
	         \draw [help lines,step=1cm,dashed] (-5.75,-6.75) grid (3.75,0.75);
		\fill[red] (0,0) circle (5pt);
		\fill (-2,0) circle (5pt) node[right]{$A_1$};
		\fill (-5,-1) circle (5pt) node[right]{$A_2$};

		\fill (0,-2) circle (5pt) node[left]{$A_1$};
		\fill (-1,-5) circle (5pt) node[left]{$A_2$};
		
		\fill (3,-2) circle (5pt) node[right]{$E_1$};
		\fill (2,-1) circle (5pt) node[right]{$E_2$};
		
		\path[decoration=arrows, decorate] (0,-2) --++ (1,1) --++ (1,0);
		\path[decoration=arrows, decorate] (-1,-5) --++ (1,1)--++(1,-1)--++(1,1)--++(1,1)--++(0,1);
\end{tikzpicture}
\begin{tikzpicture}[scale=.45,baseline=(current bounding box.center)]	
	        \draw[->,thick] (-5.75,0)--(3.75,0) node[right]{$x$};
		\draw[->,thick] (0,-6.75)--(0,0.75) node[above]{$y$};
		\draw[dotted,thick] (-0.75,0.75)--(3.75,-3.75) node[right]{\tiny $x+y=0$};
		\draw[dotted,thick] (0.25,0.75)--(3.75,-2.75) node[right]{\tiny $x+y=1$};
	         \draw [help lines,step=1cm,dashed] (-5.75,-6.75) grid (3.75,0.75);
		\fill[red] (0,0) circle (5pt);
		\fill (-2,0) circle (5pt) node[right]{$A_1$};
		\fill (-5,-1) circle (5pt) node[right]{$A_2$};

		\fill (0,-2) circle (5pt) node[left]{$A_1$};
		\fill (-1,-5) circle (5pt) node[left]{$A_2$};
		
		\fill (3,-2) circle (5pt) node[right]{$E_1$};
		\fill (2,-1) circle (5pt) node[right]{$E_2$};
		
		\path[decoration=arrows, decorate] (0,-2) --++ (1,1) --++ (1,0);
		\path[decoration=arrows, decorate] (-1,-5) --++ (1,-1)--++(1,1)--++(1,1)--++(1,1)--++(0,1);
		\end{tikzpicture}
\begin{tikzpicture}[scale=.45,baseline=(current bounding box.center)]	
	        \draw[->,thick] (-5.75,0)--(3.75,0) node[right]{$x$};
		\draw[->,thick] (0,-6.75)--(0,0.75) node[above]{$y$};
		\draw[dotted,thick] (-0.75,0.75)--(3.75,-3.75) node[right]{\tiny $x+y=0$};
		\draw[dotted,thick] (0.25,0.75)--(3.75,-2.75) node[right]{\tiny $x+y=1$};
	         \draw [help lines,step=1cm,dashed] (-5.75,-6.75) grid (3.75,0.75);
		\fill[red] (0,0) circle (5pt);
		\fill (-2,0) circle (5pt) node[right]{$A_1$};
		\fill (-5,-1) circle (5pt) node[right]{$A_2$};

		\fill (0,-2) circle (5pt) node[left]{$A_1$};
		\fill (-1,-5) circle (5pt) node[left]{$A_2$};
		
		\fill (3,-2) circle (5pt) node[right]{$E_1$};
		\fill (2,-1) circle (5pt) node[right]{$E_2$};
		
		\path[decoration=arrows, decorate] (0,-2) --++ (1,1)--++(1,-1) --++ (0,1);
		\path[decoration=arrows, decorate] (-1,-5) --++ (1,1) --++ (1,1) --++(1,-1)--++(1,1)--++(0,1);
\end{tikzpicture}
\begin{tikzpicture}[scale=.45,baseline=(current bounding box.center)]	
	        \draw[->,thick] (-5.75,0)--(3.75,0) node[right]{$x$};
		\draw[->,thick] (0,-6.75)--(0,0.75) node[above]{$y$};
		\draw[dotted,thick] (-0.75,0.75)--(3.75,-3.75) node[right]{\tiny $x+y=0$};
		\draw[dotted,thick] (0.25,0.75)--(3.75,-2.75) node[right]{\tiny $x+y=1$};
	         \draw [help lines,step=1cm,dashed] (-5.75,-6.75) grid (3.75,0.75);
		\fill[red] (0,0) circle (5pt);
		\fill (-2,0) circle (5pt) node[right]{$A_1$};
		\fill (-5,-1) circle (5pt) node[right]{$A_2$};

		\fill (0,-2) circle (5pt) node[left]{$A_1$};
		\fill (-1,-5) circle (5pt) node[left]{$A_2$};
		
		\fill (3,-2) circle (5pt) node[right]{$E_1$};
		\fill (2,-1) circle (5pt) node[right]{$E_2$};
		
		\path[decoration=arrows, decorate] (0,-2) --++ (1,1)--++(1,-1) --++ (0,1);
		\path[decoration=arrows, decorate] (-1,-5) --++ (1,1)--++(1,-1)--++(1,1)--++(1,1)--++(0,1);
\end{tikzpicture}
\begin{tikzpicture}[scale=.45,baseline=(current bounding box.center)]	
	        \draw[->,thick] (-5.75,0)--(3.75,0) node[right]{$x$};
		\draw[->,thick] (0,-6.75)--(0,0.75) node[above]{$y$};
		\draw[dotted,thick] (-0.75,0.75)--(3.75,-3.75) node[right]{\tiny $x+y=0$};
		\draw[dotted,thick] (0.25,0.75)--(3.75,-2.75) node[right]{\tiny $x+y=1$};
	         \draw [help lines,step=1cm,dashed] (-5.75,-6.75) grid (3.75,0.75);
		\fill[red] (0,0) circle (5pt);
		\fill (-2,0) circle (5pt) node[right]{$A_1$};
		\fill (-5,-1) circle (5pt) node[right]{$A_2$};

		\fill (0,-2) circle (5pt) node[left]{$A_1$};
		\fill (-1,-5) circle (5pt) node[left]{$A_2$};
		
		\fill (3,-2) circle (5pt) node[right]{$E_1$};
		\fill (2,-1) circle (5pt) node[right]{$E_2$};
		
		\path[decoration=arrows, decorate] (0,-2) --++ (1,1)--++(1,-1) --++ (0,1);
		\path[decoration=arrows, decorate] (-1,-5) --++ (1,-1)--++(1,1)--++(1,1)--++(1,1)--++(0,1);
		\end{tikzpicture}
\begin{tikzpicture}[scale=.45,baseline=(current bounding box.center)]	
	        \draw[->,thick] (-5.75,0)--(3.75,0) node[right]{$x$};
		\draw[->,thick] (0,-6.75)--(0,0.75) node[above]{$y$};
		\draw[dotted,thick] (-0.75,0.75)--(3.75,-3.75) node[right]{\tiny $x+y=0$};
		\draw[dotted,thick] (0.25,0.75)--(3.75,-2.75) node[right]{\tiny $x+y=1$};
	         \draw [help lines,step=1cm,dashed] (-5.75,-6.75) grid (3.75,0.75);
		\fill[red] (0,0) circle (5pt);
		\fill (-2,0) circle (5pt) node[right]{$A_1$};
		\fill (-5,-1) circle (5pt) node[right]{$A_2$};

		\fill (0,-2) circle (5pt) node[left]{$A_1$};
		\fill (-1,-5) circle (5pt) node[left]{$A_2$};
		
		\fill (3,-2) circle (5pt) node[right]{$E_1$};
		\fill (2,-1) circle (5pt) node[right]{$E_2$};
		
		\path[decoration=arrows, decorate] (0,-2) --++ (1,-1)--++(1,1) --++ (0,1);
		\path[decoration=arrows, decorate] (-1,-5) --++ (1,1)--++(1,-1)--++(1,1)--++(1,1)--++(0,1);
\end{tikzpicture}
\begin{tikzpicture}[scale=.45,baseline=(current bounding box.center)]	
	        \draw[->,thick] (-5.75,0)--(3.75,0) node[right]{$x$};
		\draw[->,thick] (0,-6.75)--(0,0.75) node[above]{$y$};
		\draw[dotted,thick] (-0.75,0.75)--(3.75,-3.75) node[right]{\tiny $x+y=0$};
		\draw[dotted,thick] (0.25,0.75)--(3.75,-2.75) node[right]{\tiny $x+y=1$};
	         \draw [help lines,step=1cm,dashed] (-5.75,-6.75) grid (3.75,0.75);
		\fill[red] (0,0) circle (5pt);
		\fill (-2,0) circle (5pt) node[right]{$A_1$};
		\fill (-5,-1) circle (5pt) node[right]{$A_2$};

		\fill (0,-2) circle (5pt) node[left]{$A_1$};
		\fill (-1,-5) circle (5pt) node[left]{$A_2$};
		
		\fill (3,-2) circle (5pt) node[right]{$E_1$};
		\fill (2,-1) circle (5pt) node[right]{$E_2$};
		
		\path[decoration=arrows, decorate] (0,-2) --++ (1,-1)--++(1,1) --++ (0,1);
		\path[decoration=arrows, decorate] (-1,-5) --++ (1,-1)--++(1,1)--++(1,1)--++(1,1)--++(0,1);
		\end{tikzpicture}
		
The weights are   
\begin{equation}
\label{first} 
-w,-X_1,-X_1^{-1},-X_2,-X_2^{-1},-X_1^{-1} -X_2, -X_1^{-1} X_2^{-1},-1,-X_1 X_2, -X_1 X_2^{-1}, 
\end{equation} 	
up to the overall factor 
$$- X_1 X_2 (1+X_1)(1+X_2)(X_1^{-1} + 1 + w + X_1)(X_2^{-1} +1 +w + X_2)=
-X_1 X_2 (1+X_1)(1+X_2)LR(X_1) LR(X_2),
$$
and, as can easily be seen, the sum of weights agrees with those for the arrowed 
Gelfand-Tsetlin patterns.

Now we consider the case $m=2$. We have $l=1$. The arrowed monotone triangles are as follows. 

\begin{multline*}
\begin{array}{ccc}
& {LR \atop 0} & \\ 
{L \atop 0} && {LR \atop 1}   
\end{array}, 
\begin{array}{ccc}
& {LR \atop 1} & \\ 
{LR \atop 0} && {R \atop 1}   
\end{array},
\begin{array}{ccc}
& {LR \atop 0} & \\ 
{L \atop 0} && {LR \atop 2}   
\end{array},
\begin{array}{ccc}
& {LR \atop 1} & \\ 
{LR \atop 0} && {LR \atop 2}   
\end{array},
\begin{array}{ccc}
& {LR \atop 2} & \\ 
{LR \atop 0} && {R \atop 2}   
\end{array}, 
\begin{array}{ccc}
& {LR \atop 1} & \\ 
{L \atop 1} && {LR \atop 2}   
\end{array}, \\
\begin{array}{ccc}
& {LR \atop 2} & \\ 
{LR \atop 1} && {R \atop 2}   
\end{array}\end{multline*} 
The weights are 
\begin{multline*} 
X_2(1+X_2^{-1}), 
X_1(1+X_2),
X_2^2(1+X_2^{-1}),
X_1 X_2(X_2^{-1}+1+w+X_2),
X_1^2(1+X_2),
X_1 X_2^2(1+X_2^{-1}), \\
X_1^2 X_2(1+X_2),
\end{multline*}
up to the overall factor $LR(X_1) LR(X_2)$, setting $t=u=v=1$.

As for the lattice paths, the situation is very similar to the case $m=3,l=1$, only $E_1=(3,-2)$ is replaced by the set $E_1=\{(2,-1),(3,-2)\}$ and $E_2=(2,-1)$ is replaced by the set $E_2=\{(1,0),(2,-1)\}$. It follows that all the families of paths from the case $m=3$ and $l=1$ also appear here. In addition, we have the following families of lattice paths.

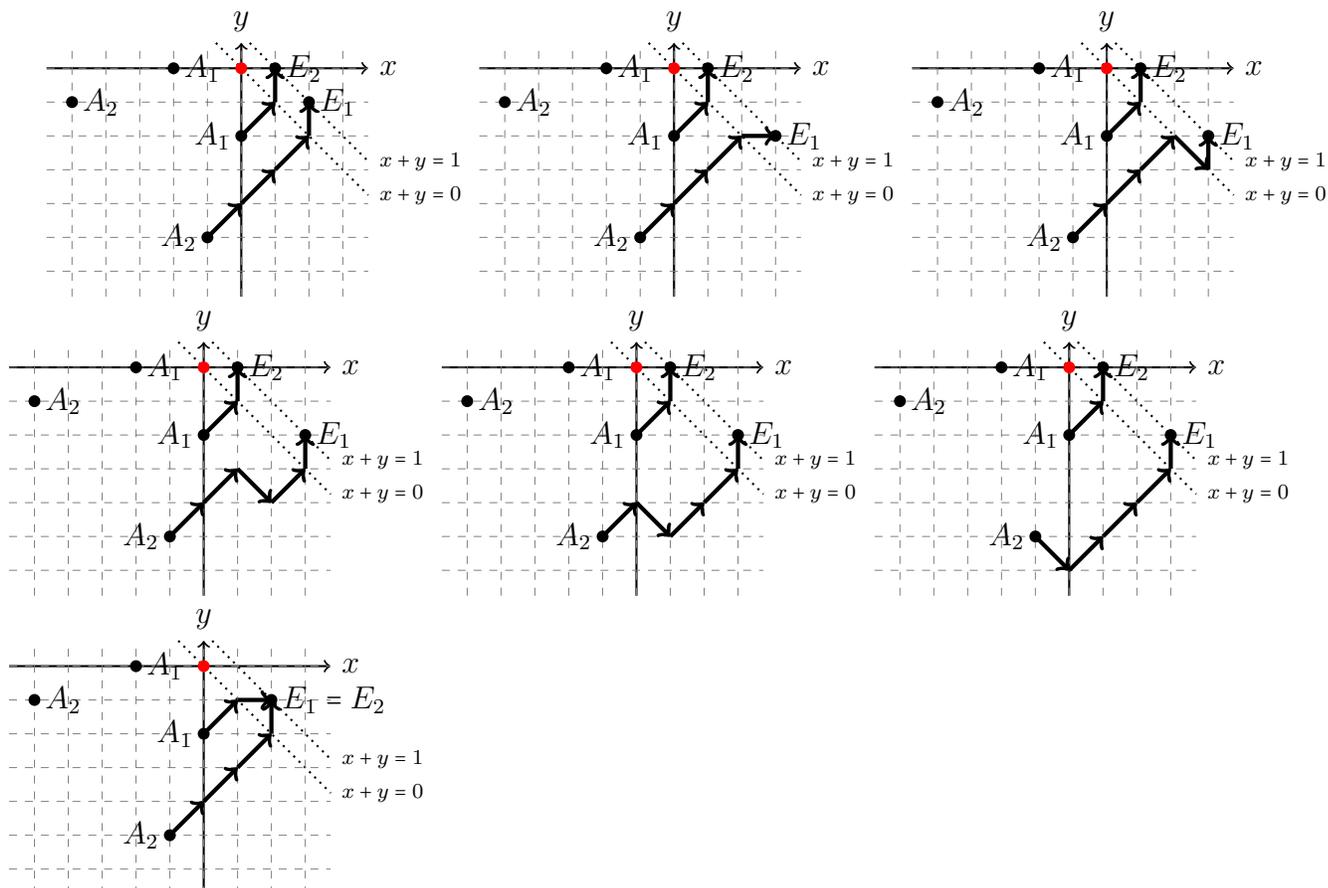
\begin{tikzpicture}[scale=.45,baseline=(current bounding box.center)]	
	        \draw[->,thick] (-5.75,0)--(3.75,0) node[right]{$x$};
		\draw[->,thick] (0,-6.75)--(0,0.75) node[above]{$y$};
		\draw[dotted,thick] (-0.75,0.75)--(3.75,-3.75) node[right]{\tiny $x+y=0$};
		\draw[dotted,thick] (0.25,0.75)--(3.75,-2.75) node[right]{\tiny $x+y=1$};
	         \draw [help lines,step=1cm,dashed] (-5.75,-6.75) grid (3.75,0.75);
		\fill[red] (0,0) circle (5pt);
		\fill (-2,0) circle (5pt) node[right]{$A_1$};
		\fill (-5,-1) circle (5pt) node[right]{$A_2$};

		\fill (0,-2) circle (5pt) node[left]{$A_1$};
		\fill (-1,-5) circle (5pt) node[left]{$A_2$};
		
		\fill (2,-1) circle (5pt) node[right]{$E_1$};
		\fill (1,0) circle (5pt) node[right]{$E_2$};
		
		\path[decoration=arrows, decorate] (0,-2) --++ (1,1)--++ (0,1);
		\path[decoration=arrows, decorate] (-1,-5) --++ (1,1)--++(1,1)--++(1,1)--++(0,1);
		\end{tikzpicture}
\begin{tikzpicture}[scale=.45,baseline=(current bounding box.center)]	
	        \draw[->,thick] (-5.75,0)--(3.75,0) node[right]{$x$};
		\draw[->,thick] (0,-6.75)--(0,0.75) node[above]{$y$};
		\draw[dotted,thick] (-0.75,0.75)--(3.75,-3.75) node[right]{\tiny $x+y=0$};
		\draw[dotted,thick] (0.25,0.75)--(3.75,-2.75) node[right]{\tiny $x+y=1$};
	         \draw [help lines,step=1cm,dashed] (-5.75,-6.75) grid (3.75,0.75);
		\fill[red] (0,0) circle (5pt);
		\fill (-2,0) circle (5pt) node[right]{$A_1$};
		\fill (-5,-1) circle (5pt) node[right]{$A_2$};

		\fill (0,-2) circle (5pt) node[left]{$A_1$};
		\fill (-1,-5) circle (5pt) node[left]{$A_2$};
		
		\fill (3,-2) circle (5pt) node[right]{$E_1$};
		\fill (1,0) circle (5pt) node[right]{$E_2$};
		
		\path[decoration=arrows, decorate] (0,-2) --++ (1,1)--++ (0,1);
		\path[decoration=arrows, decorate] (-1,-5) --++ (1,1)--++(1,1)--++(1,1)--++(1,0);
		\end{tikzpicture}
\begin{tikzpicture}[scale=.45,baseline=(current bounding box.center)]	
	        \draw[->,thick] (-5.75,0)--(3.75,0) node[right]{$x$};
		\draw[->,thick] (0,-6.75)--(0,0.75) node[above]{$y$};
		\draw[dotted,thick] (-0.75,0.75)--(3.75,-3.75) node[right]{\tiny $x+y=0$};
		\draw[dotted,thick] (0.25,0.75)--(3.75,-2.75) node[right]{\tiny $x+y=1$};
	         \draw [help lines,step=1cm,dashed] (-5.75,-6.75) grid (3.75,0.75);
		\fill[red] (0,0) circle (5pt);
		\fill (-2,0) circle (5pt) node[right]{$A_1$};
		\fill (-5,-1) circle (5pt) node[right]{$A_2$};

		\fill (0,-2) circle (5pt) node[left]{$A_1$};
		\fill (-1,-5) circle (5pt) node[left]{$A_2$};
		
		\fill (3,-2) circle (5pt) node[right]{$E_1$};
		\fill (1,0) circle (5pt) node[right]{$E_2$};
		
		\path[decoration=arrows, decorate] (0,-2) --++ (1,1)--++ (0,1);
		\path[decoration=arrows, decorate] (-1,-5) --++ (1,1)--++(1,1)--++(1,1)--++(1,-1)--++(0,1);
		\end{tikzpicture}
\begin{tikzpicture}[scale=.45,baseline=(current bounding box.center)]	
	        \draw[->,thick] (-5.75,0)--(3.75,0) node[right]{$x$};
		\draw[->,thick] (0,-6.75)--(0,0.75) node[above]{$y$};
		\draw[dotted,thick] (-0.75,0.75)--(3.75,-3.75) node[right]{\tiny $x+y=0$};
		\draw[dotted,thick] (0.25,0.75)--(3.75,-2.75) node[right]{\tiny $x+y=1$};
	         \draw [help lines,step=1cm,dashed] (-5.75,-6.75) grid (3.75,0.75);
		\fill[red] (0,0) circle (5pt);
		\fill (-2,0) circle (5pt) node[right]{$A_1$};
		\fill (-5,-1) circle (5pt) node[right]{$A_2$};

		\fill (0,-2) circle (5pt) node[left]{$A_1$};
		\fill (-1,-5) circle (5pt) node[left]{$A_2$};
		
		\fill (3,-2) circle (5pt) node[right]{$E_1$};
		\fill (1,0) circle (5pt) node[right]{$E_2$};
		
		\path[decoration=arrows, decorate] (0,-2) --++ (1,1)--++ (0,1);
		\path[decoration=arrows, decorate] (-1,-5) --++ (1,1)--++(1,1)--++(1,-1)--++(1,1)--++(0,1);
		\end{tikzpicture}						
\begin{tikzpicture}[scale=.45,baseline=(current bounding box.center)]	
	        \draw[->,thick] (-5.75,0)--(3.75,0) node[right]{$x$};
		\draw[->,thick] (0,-6.75)--(0,0.75) node[above]{$y$};
		\draw[dotted,thick] (-0.75,0.75)--(3.75,-3.75) node[right]{\tiny $x+y=0$};
		\draw[dotted,thick] (0.25,0.75)--(3.75,-2.75) node[right]{\tiny $x+y=1$};
	         \draw [help lines,step=1cm,dashed] (-5.75,-6.75) grid (3.75,0.75);
		\fill[red] (0,0) circle (5pt);
		\fill (-2,0) circle (5pt) node[right]{$A_1$};
		\fill (-5,-1) circle (5pt) node[right]{$A_2$};

		\fill (0,-2) circle (5pt) node[left]{$A_1$};
		\fill (-1,-5) circle (5pt) node[left]{$A_2$};
		
		\fill (3,-2) circle (5pt) node[right]{$E_1$};
		\fill (1,0) circle (5pt) node[right]{$E_2$};
		
		\path[decoration=arrows, decorate] (0,-2) --++ (1,1)--++ (0,1);
		\path[decoration=arrows, decorate] (-1,-5) --++ (1,1)--++(1,-1)--++(1,1)--++(1,1)--++(0,1);
		\end{tikzpicture}
\begin{tikzpicture}[scale=.45,baseline=(current bounding box.center)]	
	        \draw[->,thick] (-5.75,0)--(3.75,0) node[right]{$x$};
		\draw[->,thick] (0,-6.75)--(0,0.75) node[above]{$y$};
		\draw[dotted,thick] (-0.75,0.75)--(3.75,-3.75) node[right]{\tiny $x+y=0$};
		\draw[dotted,thick] (0.25,0.75)--(3.75,-2.75) node[right]{\tiny $x+y=1$};
	         \draw [help lines,step=1cm,dashed] (-5.75,-6.75) grid (3.75,0.75);
		\fill[red] (0,0) circle (5pt);
		\fill (-2,0) circle (5pt) node[right]{$A_1$};
		\fill (-5,-1) circle (5pt) node[right]{$A_2$};

		\fill (0,-2) circle (5pt) node[left]{$A_1$};
		\fill (-1,-5) circle (5pt) node[left]{$A_2$};
		
		\fill (3,-2) circle (5pt) node[right]{$E_1$};
		\fill (1,0) circle (5pt) node[right]{$E_2$};
		
		\path[decoration=arrows, decorate] (0,-2) --++ (1,1)--++ (0,1);
		\path[decoration=arrows, decorate] (-1,-5) --++ (1,-1)--++(1,1)--++(1,1)--++(1,1)--++(0,1);
		\end{tikzpicture}
\begin{tikzpicture}[scale=.45,baseline=(current bounding box.center)]	
	        \draw[->,thick] (-5.75,0)--(3.75,0) node[right]{$x$};
		\draw[->,thick] (0,-6.75)--(0,0.75) node[above]{$y$};
		\draw[dotted,thick] (-0.75,0.75)--(3.75,-3.75) node[right]{\tiny $x+y=0$};
		\draw[dotted,thick] (0.25,0.75)--(3.75,-2.75) node[right]{\tiny $x+y=1$};
	         \draw [help lines,step=1cm,dashed] (-5.75,-6.75) grid (3.75,0.75);
		\fill[red] (0,0) circle (5pt);
		\fill (-2,0) circle (5pt) node[right]{$A_1$};
		\fill (-5,-1) circle (5pt) node[right]{$A_2$};

		\fill (0,-2) circle (5pt) node[left]{$A_1$};
		\fill (-1,-5) circle (5pt) node[left]{$A_2$};
		
		\fill (2,-1) circle (5pt) node[right]{$E_1=E_2$};
		
		\path[decoration=arrows, decorate] (0,-2) --++ (1,1)--++ (1,0);
		\path[decoration=arrows, decorate] (-1,-5) --++ (1,1)--++(1,1)--++(1,1)--++(0,1);
		\end{tikzpicture}

Thus, in addition to the weights in \eqref{first}, these families of lattice paths give 
$$
-1,-w,-X_1,-X_1^{-1},-X_2,-X_2^{-1},-1,w,
$$
up to the overall factor 
$$- X_1 X_2 (X_1^{-1} + 1 + w + X_1)(X_2^{-1} +1 +w + X_2)=
-X_1 X_2 LR(X_1) LR(X_2),
$$
where the last two weights come from the last picture, first by interpreting the endpoint of the path that starts in $A_1$ as element of $E_2$ and second as element of $E_1$. 										

\subsection{First proof of Theorem~\ref{RHScomplicated}}
The approach of the first proof of Theorem~\ref{RHScomplicated} is closely related to the approach we used in the proof of Theorem~2.2 in \cite{nVASMDPP}.

We consider the following bases for Laurent polynomials in $X$ that are invariant under the transformation $X \to X^{-1}$: let 
$$
q_i(X)=\frac{X^i-X^{-i}}{X-X^{-1}} \qquad \text{and} \qquad b_i(X)=(X+X^{-1})^i,
$$
then $(q_i(X))_{i \ge 0}$ and $(b_i(X))_{i \ge 0}$ are two such bases. It is not hard to verify that 
\begin{equation}
\label{transform} 
q_m(X)= \sum_{r=0}^{(m-1)/2} (-1)^r \binom{m-r-1}{r} b_{m-1-2r}(X).
\end{equation} 
In order to derive a combinatorial interpretation of the right-hand side of \eqref{rincluded1}, consider 
\begin{equation}
\label{det}  
\det_{1 \le i, j \le n} \left( X_i^{j-1}  (1+X_i)^{j-1} (1+ w X_i)^{n-j}  - X_i^{m+2n-j} (1+X_i^{-1})^{j-1}  (1+w X_i^{-1})^{n-j} \right).
\end{equation} 
\emph{We start by considering the case that $m$ is odd:} We set $m=2l+1$, and pull out $\prod_{i=1}^n X_i^{l+n}$.
\begin{equation}
\label{odd}
\prod_{i=1}^n X_i^{l+n} \det_{1 \le i, j \le n} \left( X_i^{j-l-n-1}  (1+X_i)^{j-1} (1+ w X_i)^{n-j}  - X_i^{-j+l+n+1} (1+X_i^{-1})^{j-1}  (1+w X_i^{-1})^{n-j} \right)
\end{equation} 
The entry in the $i$-th row and $j$-th column of the matrix underlying the determinant is obtained from 
\begin{multline*} 
\frac{X^{j-l-n-1} (1+X)^{j-1} (1+w X)^{n-j} - X^{-j+l+n+1} (1+X^{-1})^{j-1} (1+w X^{-1})^{n-j}}{X-X^{-1}} 
\\
= \sum_{p,q \ge 0} \binom{j-1}{p} \binom{n-j}{q} w^q \frac{X^{j-l-n+p+q-1}-X^{-j+l+n-p-q+1}}{X-X^{-1}}
\end{multline*} 
by multiplying with $X-X^{-1}$ and then setting $X=X_i$. Note that this expression is invariant under replacing $X$ by $X^{-1}$. 
From 
\eqref{transform}, it follows that this is further equal to 
\begin{multline*}
\sum_{p,q,r \ge 0 \atop |j-l-n+p+q-1|-1-2r \ge 0} \sgn(j-l-n+p+q-1) (-1)^r w^q \binom{j-1}{p} \binom{n-j}{q} \\ \times 
\binom{|j-l-n+p+q-1|-r-1}{r}   b_{|j-l-n+p+q-1|-1-2r}(X).
\end{multline*} 

We apply the following lemma. A proof can be found in \cite[Lemma 7.2]{nASMDPP}. Note that the lemma also involves 
complete homogeneous symmetric polynomials $h_k$ with negative $k$ as defined in \cite[Section~5]{nASMDPP}. Concretely, we 
define $h_k(X_1,\ldots,X_n)=0$ for $k=-1,-2,\ldots,-n+1$ and 
\begin{equation}
\label{reci} 
h_k(X_1,\ldots,X_n) = (-1)^{n+1} X_1^{-1} \ldots X_n^{-1} h_{-k-n}(X_1^{-1},\ldots,X_n^{-1})
\end{equation}
for $k \le -n$. Note that a consequence of this definition is that the latter relation is true for any $k$. 

\begin{lemma}
\label{limit}
Let $f_j(Y)$ be formal Laurent series for $1 \le j \le n$, and define 
$$
f_j[Y_1,\ldots,Y_i]=\sum_{k \in \mathbb{Z}} \langle Y^{k} \rangle f_j(Y) \cdot h_{k-i+1}(Y_1,\ldots,Y_i),
$$
where $\langle Y^{k} \rangle f_j(Y)$ denotes the coefficient of $Y^{k}$ in $f_j(Y)$ and $h_{k-i+1}$ denotes the complete homogeneous symmetric polynomial of degree~$k-i+1$.
Then 
$$
\frac{\det_{1 \le i, j \le n} \left( f_j(Y_i) \right) }{\prod_{1 \le i < j \le n} (Y_j - Y_i)} = \det_{1 \le i, j \le n} \left( f_j[Y_1,\ldots,Y_i] \right).
$$
\end{lemma}

Noting that a Laurent polynomial in $X$ that is invariant under the replacement $X \to X^{-1}$ can be written as a polynomial in $X+X^{-1}$, we use the lemma to basically rewrite \eqref{odd} as follows.
\begin{multline*} 
\frac{
\det_{1 \le i, j \le n} \left( X_i^{j-l-n-1}  (1+X_i)^{j-1} (1+ w X_i)^{n-j}  - X_i^{-j+l+n+1} (1+X_i^{-1})^{j-1}  (1+w X_i^{-1})^{n-j} \right)}
{\prod_{1 \le i < j \le n} (X_j + X_j^{-1} - X_i - X_i^{-1})} \\
= \prod_{i=1}^n (X_i - X_i^{-1})  \det_{1 \le i, j \le n} \left( \sum_{p,q,r \ge 0 \atop |j-l-n+p+q-1|-i-2r \ge 0} 
\sgn(j-l-n+p+q-1) (-1)^r w^q \binom{j-1}{p} \binom{n-j}{q}  \right. \\ \left. \phantom{\sum_{p,q,r \ge 0}}  \times 
\binom{|j-l-n+p+q-1|-r-1}{r} h_{|j-l-n+p+q-1|-i-2r}(X_1+X_1^{-1},\ldots,X_i+X_i^{-1}) \right) 
\end{multline*} 
Now, as $X_j + X_j^{-1} - X_i - X_i^{-1}=(X_i-X_j)(1-X_i X_j)X_i^{-1} X_j^{-1}$, in order to find a combinatorial interpretation for 
the right-hand side of \eqref{rincluded1}, we need to find a combinatorial interpretation of 
\begin{multline*} 
(-1)^{\binom{n}{2}} \prod_{i=1}^{n} X_i^{l+1} (X_i^{-1}+1+w + X_i) (X_i - X_i^{-1})(1-X_i)^{-1} \\ \times 
\det_{1 \le i, j \le n} \left( \sum_{p,q,r \ge 0 \atop |j-l-n+p+q-1|-i-2r \ge 0} \sgn(j-l-n+p+q-1) (-1)^r w^q \binom{j-1}{p} \binom{n-j}{q}   \right. \\ \left. \phantom{\sum_{p,q,r \ge 0}}  \times \binom{|j-l-n+p+q-1|-r-1}{r}
h_{|j-l-n+p+q-1|-i-2r}(X_1+X_1^{-1},\ldots,X_i+X_i^{-1}) \right)\\
=(-1)^{\binom{n+1}{2}} \prod_{i=1}^{n} X_i^{l} (X_i^{-1}+1+w + X_i)(1+X_i) \\ \times 
\det_{1 \le i, j \le n} \left( \sum_{p,q,r \ge 0 \atop |j-l-n+p+q-1|-i-2r \ge 0} \sgn(j-l-n+p+q-1) (-1)^r w^q \binom{j-1}{p} \binom{n-j}{q}   \right. \\ \left. \phantom{\sum_{p,q,r \ge 0}}  \times \binom{|j-l-n+p+q-1|-r-1}{r}
h_{|j-l-n+p+q-1|-i-2r}(X_1+X_1^{-1},\ldots,X_i+X_i^{-1}) \right).
\end{multline*}  

For this purpose, we find a combinatorial interpretation of the entry of the underlying matrix, i.e.,  
\begin{multline*}
\sum_{p,q,r \ge 0 \atop |j-l-n+p+q-1|-i-2r \ge 0} \sgn(j-l-n+p+q-1) (-1)^r w^q \binom{j-1}{p} \binom{n-j}{q} \binom{|j-l-n+p+q-1|-r-1}{r}  \\  \times 
h_{|j-l-n+p+q-1|-i-2r}(X_1+X_1^{-1},\ldots,X_i+X_i^{-1})
\end{multline*} 
in terms of a lattice paths generating function. We simplify the expression using the transformation $q \to n-j-q$. 
\begin{multline}
\label{triple} 
\sum_{p,q,r \ge 0 \atop |p-q-l-1|-i-2r \ge 0} \sgn(p-q-l-1) (-1)^r w^{n-j-q} \binom{j-1}{p} \binom{n-j}{q} \binom{|p-q-l-1|-r-1}{r}  \\
\times 
h_{|p-q-l-1|-i-2r}(X_1+X_1^{-1},\ldots,X_i+X_i^{-1})
\end{multline} 
We simplify the expression further using the following lemma. A combinatorial proof of it using a sign-reversing involution is provided in 
\cite[Lemma 7.7]{nVASMDPP}.
\begin{lemma} 
\label{h}
Let $a,i$ be positive integers with $i \le a$. Then 
$$
\sum_{r=0}^{(a-i)/2} (-1)^r \binom{a-r-1}{r} h_{a-i-2r}(X_1+X_1^{-1},\ldots,X_i+X_i^{-1}) = h_{a-i}(X_1,X_1^{-1},\ldots,X_i,X_i^{-1}).
$$
\end{lemma} 
Therefore, the sum in \eqref{triple} is equal to 
\begin{equation}
\label{path}  
\sum_{p,q} \sgn(p-q-l-1)  w^{n-j-q} \binom{j-1}{p} \binom{n-j}{q}   \\
h_{|p-q-l-1|-i}(X_1,X_1^{-1},\ldots,X_i,X_i^{-1}).
\end{equation} 

We claim the following: If $p-q-l-1 \ge 0$, then \eqref{path} is the generating function of lattice paths from 
$(-3i+1,-i+1)$ to $(n-j+l+1,j-l-2)$ such that the following is satisfied.
\begin{itemize}
\item Below and on the line $x+y=0$, the step set is $\{(1,1),(-1,1)\}$. Steps of type $(-1,1)$ with distances $0,2,4,\ldots$ from 
$x+y=n$ are equipped with the weights $X_1,X_2,X_3,\ldots$, respectively, while steps of type $(-1,1)$ with distances $1,3,5,\ldots$ 
are equipped with the weights $X_1^{-1},X_2^{-1},X_3^{-1},\ldots$, respectively.
\item Above the line $x+y=0$, the step set is $\{(1,0),(0,1)\}$. Above the line $x+y=j-1$, horizontal steps are equipped with the weight
$w$.
\end{itemize} 
Namely, if we assume that there are $q$ steps of type $(0,1)$ above the line $x+y=j-1$, and, therefore, $n-j-q$ steps of type
$(1,0)$, then the path intersects the line $x+y=j-1$ in the lattice point $(l+1+q,j-l-2-q)$, assuming that the endpoint of the path is 
$(n-j+l+1,j-l-2)$, and there are $\binom{n-j}{q}$ of such paths 
each of them contributing $w^{n-j-q}$ to the weight. Note that this weight depends on $j$ if $w \not= 0,1$, and this causes complications when applying the Lindstr\"om-Gessel-Viennot lemma. 

If we further assume that there are $p$ steps of type $(1,0)$ below the 
line $x+y=j-1$, and, therefore, $j-p$ steps of type $(0,1)$,  then the last lattice point of such a path on the line $x+y=0$ when traversing 
the path from $(-3i+1,-i+1)$ to $(n-j+l+1,j-l-2)$ is $(-p+q+l+1,p-q-l-1)$. Note that by the assumption $p-q-l-1 \ge 0$, the lattice point 
$(-p+q+l+1,p-q-l-1)$ is in the second quadrant, i.e., $\{(x,y)|x \le 0, y \ge 0 \}$.

Finally, lattice paths from $(-3i+1,-i+1)$ to $(-p+q+l+1,p-q-l-1)$ with step set $\{(1,1),(-1,1)\}$ have $p-q-l-1-i$ steps of type 
$(-1,1)$ and $2i-1$ steps of type $(1,1)$. The generating function of such paths is clearly 
$h_{p-q-l-1-i}(X_1,X_1^{-1},\ldots,X_i,X_i^{-1})=h_{|p-q-l-1|-i}(X_1,X_1^{-1},\ldots,X_i,X_i^{-1})$.

The situation is very similar if $p-q-l-1 \le 0$, except that we need to replace the starting point $(-3i+1,-i+1)$ by $(-i+1,-3i+1)$ and 
the step set is $\{(1,1),(1,-1)\}$ below the line $x+y=0$. Again we can assume that $(-p+q+l+1,p-q-l-1)$ is the last lattice point on the line $x+y=0$ when traversing the path from $(-i+1,-3i+1)$ to $(-p+q+l+1,p-q-l-1)$. In this case, $(-p+q+l+1,p-q-l-1)$ lies in the fourth quadrant $\{(x,y)|x \le 0, y \le 0 \}$. We have $-p+q+l+1-i=|p-q-l-1|-i$ steps of type $(1,-1)$ and $2i-1$ steps of type $(1,1)$, thus the generating function in this segment is also $h_{|p-q-l-1|-i}(X_1,X_1^{-1},\ldots,X_i,X_i^{-1})$. 

Consequently, we can conclude that the right-hand side of \eqref{rincluded1} has the following combinatorial interpretation: We consider families of $n$ lattice paths from $A_i=\{(-3i+1,-i+1),(-i+1,-3i+1)\}$, $i=1,2,\ldots,n$, to $E_j=(n-j+l+1,j-l-2)$, $j=1,2,\ldots,n$, with steps sets and 
weights as described above. By the Lindstr\"om-Gessel-Viennot lemma \cite{Lin73,GesVie85,GesVie89}, the paths can be assumed to be 
non-intersecting on and below the line $x+y=0$. 

In case $w=0,1$, we can also assume them to be non-intersecting. This is clear for $w=1$.
In case $w=0$, we can assume that there are no steps of type $(1,0)$ above the line $x+y=j-1$, and, therefore, we can also have $(n-j+l+1,2j-n-l-2)$ on the line $x+y=j-1$ as endpoint since above the line all the $n-j$ steps have to be of type $(0,1)$. Whenever we choose $(-i+1,-3i+1)$, this contributes $-1$ to the weight. 

In the non-intersecting setting, suppose we choose 
$(-i+1,-3i+1)$ from $A_i$ for $1 \le i_1 < \ldots < i_m \le n$, then the sign of the permutation $\sigma$ such that $A_i$ is connected to 
$E_{\sigma(i)}$ via the paths is $(-1)^{i_1+i_2+\ldots+i_m-m}$. This gives a total sign of  
$(-1)^{i_1+i_2+\ldots+i_m}$.
Recall also that we have an additional overall weight of 
$$
(-1)^{\binom{n+1}{2}} \prod_{i=1}^{n} X_i^{l} (X_i^{-1}+1+w + X_i)(1+X_i).
$$
Combining the sign from above with $(-1)^{\binom{n+1}{2}}=(-1)^{1+2+\ldots+n}$, the sign can also be computed as follows: suppose we choose $(-3i+1,-i+1)$ from $A_i$ precisely for $i_1,\ldots,i_m$, then the sign is $(-1)^{i_1+i_2+\ldots+i_m}$ and in this setting the overall weight is 
$$
\prod_{i=1}^{n} X_i^{l} (X_i^{-1}+1+w + X_i)(1+X_i).
$$
This concludes the proof of the first part of Theorem~\ref{RHScomplicated}.

\bigskip

\emph{Now we consider the case that $m$ is even:} We set $m=2l$ in \eqref{det}, and pull out $\prod_{i=1}^n X_i^{l+n}$.
$$
\prod_{i=1}^n X_i^{l+n} \det_{1 \le i, j \le n} \left( X_i^{j-l-n-1}  (1+X_i)^{j-1} (1+ w X_i)^{n-j}  - X_i^{-j+l+n} (1+X_i^{-1})^{j-1}  (1+w X_i^{-1})^{n-j} \right).
$$
The entry in the $i$-th row of the $j$-th column of the matrix underlying the determinant is obtained from
\begin{multline*} 
\frac{X^{j-l-n-1} (1+X)^{j-1} (1+w X)^{n-j} - X^{-j+l+n} (1+X^{-1})^{j-1} (1+w X^{-1})^{n-j}}{1-X^{-1}} 
\\
= \sum_{p,q \ge 0} \binom{j-1}{p} \binom{n-j}{q} w^q \frac{X^{j-l-n+p+q-1}-X^{-j+l+n-p-q}}{1-X^{-1}}
\end{multline*} 
when multiplying with $1-X^{-1}$ and then setting $X=X_i$. Now note that 
$$
\frac{X^m-X^{-m-1}}{1-X^{-1}} = q_{m+1}(X) + q_{m}(X) 
$$
for any integer $m$, so that we obtain 
$$
\sum_{p,q \ge 0} \binom{j-1}{p} \binom{n-j}{q} w^q \left(q_{j-l-n+p+q}(X) + q_{j-l-n+p+q-1}(X) \right).
$$
It follows from \eqref{transform} that this is 
\begin{multline*}
\sum_{p,q,r \ge 0 \atop |j-l-n+p+q|-1-2r \ge 0} \sgn(j-l-n+p+q) (-1)^r w^q \binom{j-1}{p} \\ 
\times \binom{n-j}{q} \binom{|j-l-n+p+q|-r-1}{r}   
b_{|j-l-n+p+q|-1-2r}(X) \\
+ \sum_{p,q,r \ge 0 \atop |j-l-n+p+q-1|-1-2r \ge 0} \sgn(j-l-n+p+q-1) (-1)^r w^q \binom{j-1}{p} \binom{n-j}{q} 
\\ \times \binom{|j-l-n+p+q-1|-r-1}{r}   
b_{|j-l-n+p+q-1|-1-2r}(X).
\end{multline*} 
Also here we simplify the expression using the replacement $q \to n-j-q$.
\begin{multline*}
\sum_{p,q,r \ge 0 \atop |-l+p-q|-1-2r \ge 0} \sgn(-l+p-q) (-1)^r w^{n-j-q} \binom{j-1}{p} \binom{n-j}{q} \binom{|-l+p-q|-r-1}{r}   
b_{|-l+p-q|-1-2r}(X) \\
+ \sum_{p,q,r \ge 0 \atop |-l+p-q-1|-1-2r \ge 0} \sgn(-l+p-q-1) (-1)^r w^{n-j-q} \binom{j-1}{p} \binom{n-j}{q} \binom{|-l+p-q-1|-r-1}{r}   
b_{|-l+p-q-1|-1-2r}(X)
\end{multline*} 
This implies the following.
\begin{multline*} 
\frac{
\det_{1 \le i, j \le n} \left( X_i^{j-l-n-1}  (1+X_i)^{j-1} (1+ w X_i)^{n-j}  - X_i^{-j+l+n} (1+X_i^{-1})^{j-1}  (1+w X_i^{-1})^{n-j} \right)}
{\prod_{1 \le i < j \le n} (X_j + X_j^{-1} - X_i - X_i^{-1})} \\
= \prod_{i=1}^n (1-X_i^{-1})  \det_{1 \le i, j \le n} \left( a_{i,j} 
 \right),  
\end{multline*}
with 
\begin{multline*} 
a_{i,j} = \sum_{p,q,r \ge 0 \atop |p-q-l|-1-2r \ge 0} \sgn(p-q-l) (-1)^r w^{n-j-q} \binom{j-1}{p} \binom{n-j}{q} \\ \times \binom{|p-q-l|-r-1}{r}   
h_{|p-q-l|-i-2r}(X_1+X_1^{-1},\ldots,X_i+X_i^{-1}) \ \\
 + \sum_{p,q,r \ge 0 \atop |p-q-l-1|-1-2r \ge 0} \sgn(p-q-l-1) (-1)^r w^{n-j-q} \binom{j-1}{p} \binom{n-j}{q} \\ \times \binom{|p-q-l-1|-r-1}{r}   
h_{|p-q-l-1|-i-2r}(X_1+X_1^{-1},\ldots,X_i+X_i^{-1}).
\end{multline*} 
Using Lemma~\ref{h}, we see that this is equal to 
\begin{multline} 
\label{path2}
b_{i,j} = \sum_{p,q \ge 0} \sgn(p-q-l) w^{n-j-q} \binom{j-1}{p} \binom{n-j}{q}   
h_{|p-q-l|-i}(X_1,X_1^{-1},\ldots,X_i,X_i^{-1}) \ \\
 + \sum_{p,q \ge 0} \sgn(p-q-l-1) w^{n-j-q} \binom{j-1}{p} \binom{n-j}{q} 
h_{|p-q-l-1|-i}(X_1,X_1^{-1},\ldots,X_i,X_i^{-1}).
\end{multline} 

Here we need to find a combinatorial interpretation of 
$$
(-1)^{\binom{n+1}{2}} \prod_{i=1}^{n} X_i^{l} (X_i^{-1}+1+w + X_i)  
\det_{1 \le i, j \le n} \left( b_{i,j} \right).
$$
The only modification compared to the odd case is that the endpoints have to be replaced by the following set of two endpoints 
$E_j=\{(n-j+l+1,j-l-2),(n-j+l,j-l-1)\}$ and that the overall factor is 
$$
\prod_{i=1}^{n} X_i^{l} (X_i^{-1}+1+w+ X_i) , 
$$
given that the sign is taken care of as above. 

This concludes the proof of Theorem~\ref{RHScomplicated}.




\subsection{Right-hand side of \eqref{rincluded1}, second proof}

In this section, we sketch a second proof of Theorem~\ref{RHScomplicated}. 
It is closely related to the proof of Theorem~2.4 in \cite{nVASMDPP}. We only study the case 
$m=2l+1$. Again, we need to consider 
\begin{equation}
\label{odd2nd}  
\det_{1 \le i, j \le n} \left( X_i^{j-l-n-1}  (1+X_i)^{j-1} (1+ w X_i)^{n-j}  - X_i^{-j+l+n+1} (1+X_i^{-1})^{j-1}  (1+w X_i^{-1})^{n-j} \right).
\end{equation}

We have the following lemma.

\begin{lemma} 
\label{22}
For $n \ge 1$ and $l \in \mathbb{Z}$, the following identity holds.
\begin{multline*} 
\frac{1}{\prod_{1 \le i < j \le n} (X_j-X_i)(X_j^{-1} - X_i^{-1}) \prod_{i,j=1}^n (X_j^{-1} - X_i)} \\
\times
\det_{1 \le i, j \le n} \left( X_i^{j-l-n-1}  (1+X_i)^{j-1} (1+ w X_i)^{n-j}  - X_i^{-j+l+n+1} (1+X_i^{-1})^{j-1}  (1+w X_i^{-1})^{n-j} \right) \\
\times 
\det_{1 \le i, j \le n} \left( X_i^{j-l-n-1}  (1+X_i)^{j-1} (1+ w X_i)^{n-j}  + X_i^{-j+l+n+1} (1+X_i^{-1})^{j-1}  (1+w X_i^{-1})^{n-j} \right) \\
= \frac{(-1)^n}{2} \det_{1 \le i, j \le n} \left( \sum_{k,q} \binom{j-1}{-j+k+l+n-q+1} \binom{n-j}{q} w^q (h_{k-i+1}-h_{k+i-1-2n})  \right) \\
\times 
\det_{1 \le i,j \le n} 
\left( \sum_{k,q} \binom{j-1}{-j+k+l+n-q+1} \binom{n-j}{q} w^q 
(h_{k+i-1-n}+ h_{k-i+1-n}) \right)
\end{multline*} 
\end{lemma} 


\begin{proof} 
We use 
$$
\det(A-B) \det(A+B) = 
\det \left( \begin{array}{c|c} A-B & B \\ \hline  0 & A+B \end{array} \right) =
\det \left( \begin{array}{c|c} A-B & B \\ \hline  B-A & A \end{array} \right) =
 \det \left( \begin{array}{c|c} A & B \\ \hline  B & A \end{array} \right)
$$
to see that the product of determinants on the left-hand side in the assertion of the lemma 
is equal to 
$$
\det \left( \begin{array}{c|c} 
\left( X_i^{j-l-n-1}  (1+X_i)^{j-1} (1+ w X_i)^{n-j}  \right)_{1 \le i, j \le n} & 
\left( X_i^{-j+l+n+1} (1+X_i^{-1})^{j-1}  (1+w X_i^{-1})^{n-j}  \right)_{1 \le i,j \le n} 
\\ \hline  
\left( X_i^{-j+l+n+1} (1+X_i^{-1})^{j-1}  (1+w X_i^{-1})^{n-j} \right)_{1 \le i,j \le n} & 
\left( X_i^{j-l-n-1}  (1+X_i)^{j-1} (1+ w X_i)^{n-j} \right)_{1 \le i,j \le n}
 \end{array} \right). 
$$
Setting $X_{n+i} =  X_i^{-1}$ for $i=1,2,\ldots,n$, we can also write this is as
$$
\det \left( \begin{array}{c|c} 
\left( X_i^{j-l-n-1}  (1+X_i)^{j-1} (1+ w X_i)^{n-j} \right)_{1 \le i \le 2n \atop 1 \le j \le n} & 
\left( X_i^{-j+l+n+1} (1+X_i^{-1})^{j-1}  (1+w X_i^{-1})^{n-j} \right)_{1 \le i \le 2n \atop 1 \le j \le n} 
\end{array} \right). 
$$
We apply Lemma~\ref{limit} to
$$
\frac{\det \left( \begin{array}{c|c} 
\left( X_i^{j-l-n-1}  (1+X_i)^{j-1} (1+ w X_i)^{n-j} \right)_{1 \le i \le 2n \atop 1 \le j \le n} & 
\left( X_i^{-j+l+n+1} (1+X_i^{-1})^{j-1}  (1+w X_i^{-1})^{n-j} \right)_{1 \le i \le 2n \atop 1 \le j \le n} 
\end{array} \right)}{\prod_{1 \le i < j \le 2n} (X_j-X_i)}
$$
and obtain 
\begin{multline*}
\det \left(  \left. 
\left( \sum_{k,q} \binom{j-1}{-j+k+l+n-q+1} \binom{n-j}{q} w^q h_{k-i+1}(X_1,\ldots,X_i)  \right)_{1 \le i \le 2n \atop 1 \le j \le n} \right|  \right. 
\\
\left.  \left( \sum_{k,q} \binom{j-1}{-j-k+l+n-q+1} \binom{n-j}{q} w^q h_{k-i+1}(X_1,\ldots,X_i)  \right)_{1 \le i \le 2n \atop 1 \le j \le n} 
  \right). 
\end{multline*} 
We multiply from the left with the following matrix 
$$
(h_{j-i}(X_j,X_{j+1},\ldots,X_{2n}))_{1 \le i,j \le 2n}
$$
with determinant $1$. For this purpose, note that  
$$
\sum_{l=1}^{2n} h_{l-i}(X_l,X_{l+1},\ldots,X_{2n}) h_{k-l+1}(X_1,\ldots,X_l)=h_{k-i+1}(X_1,\ldots,X_{2n}), 
$$
and, therefore, the multiplication results in 
\begin{multline*} 
 \det \left( \left.
\left( \sum_{k,q} \binom{j-1}{-j+k+l+n-q+1} \binom{n-j}{q} w^q h_{k-i+1}(X_1,\ldots,X_{2n})  \right)_{1 \le i \le 2n \atop 1 \le j \le n} \right| \right. \\ 
\left. \left( \sum_{k,q} \binom{j-1}{-j-k+l+n-q+1} \binom{n-j}{q} w^q h_{k-i+1}(X_1,\ldots,X_{2n})  \right)_{1 \le i \le 2n \atop 1 \le j \le n} 
\right). 
\end{multline*} 
We set $X_{i+n}=X_i^{-1}$ for $i=1,2,\ldots,n$ (so that the arguments of all complete symmetric functions are 
$(X_1,\ldots,X_n,X_1^{-1},\ldots,X_n^{-1})$) and omit the $X_i$'s now. Also note that, under this specialization, the 
denominator $\prod_{1 \le i < j \le 2n} (X_j-X_i)$ specializes to the denominator on the left-hand side in the assertion of 
the lemma.
$$
\det \left( \begin{array}{c|c} 
\left( \sum_{k,q} \binom{j-1}{-j+k+l+n-q+1} \binom{n-j}{q} w^q h_{k-i+1}  \right)_{1 \le i \le 2n \atop 1 \le j \le n} & 
\left( \sum_{k,q} \binom{j-1}{-j-k+l+n-q+1} \binom{n-j}{q} w^q h_{k-i+1}  \right)_{1 \le i \le 2n \atop 1 \le j \le n} 
\end{array} \right)
$$
With this specialization, we have $h_k=-h_{-k-2n}$ using \eqref{reci}. Therefore, the above is 
\begin{multline*}
(-1)^n
\det \left( \begin{array}{c|c} 
\left( \sum_{k,q} \binom{j-1}{-j+k+l+n-q+1} \binom{n-j}{q} w^q h_{k-i+1}  \right)_{1 \le i \le 2n \atop 1 \le j \le n} & 
\left( \sum_{k,q} \binom{j-1}{-j-k+l+n-q+1} \binom{n-j}{q} w^q h_{-k+i-1-2n}  \right)_{1 \le i \le 2n \atop 1 \le j \le n} 
\end{array} \right) \\
= (-1)^n
\det \left( \begin{array}{c|c} 
\left( \sum_{k,q} \binom{j-1}{-j+k+l+n-q+1} \binom{n-j}{q} w^q h_{k-i+1}  \right)_{1 \le i \le 2n \atop 1 \le j \le n} & 
\left( \sum_{k,q} \binom{j-1}{-j+k+l+n-q+1} \binom{n-j}{q} w^q h_{k+i-1-2n}  \right)_{1 \le i \le 2n \atop 1 \le j \le n} 
\end{array} \right).
\end{multline*} 
Now, for $j=1,2,\ldots,n$, we subtract the $(j+n)$-th column from the $j$-th column.
\begin{multline*} 
(-1)^n
\det \left(  \left.
\left( \sum_{k,q} \binom{j-1}{-j+k+l+n-q+1} \binom{n-j}{q} w^q (h_{k-i+1}-h_{k+i-1-2n})  \right)_{1 \le i \le 2n \atop 1 \le j \le n} \right| \right. \\
\left. \left( \sum_{k,q} \binom{j-1}{-j+k+l+n-q+1} \binom{n-j}{q} w^q h_{k+i-1-2n}  \right)_{1 \le i \le 2n \atop 1 \le j \le n} 
 \right)
\end{multline*} 
For $i=n+2,n+3,\ldots,2n$, we add the $(2n+2-i)$-th row to the $i$-th row. This gives a zero block for 
$\{(i,j) | n+1 \le i \le 2n, 1 \le j \le n\}$, since 
$$
 \sum_{k,q} \binom{j-1}{-j+k+l+n-q+1} \binom{n-j}{q} w^q (h_{k-i+1}-h_{k+i-1-2n}+h_{k-(2n+2-i)+1}-h_{k+(2n+2-i)-1-2n})=0.
$$
The lower right block is 
\begin{multline*} 
\det_{n+1 \le i \le 2n \atop 1 \le j \le n} 
\left( \sum_{k,q} \binom{j-1}{-j+k+l+n-q+1} \binom{n-j}{q} w^q 
(h_{k+i-1-2n}+ [i\not=n+1]h_{k+2n+2-i-1-2n}) \right) \\
= \det_{n+1 \le i \le 2n \atop 1 \le j \le n} 
\left( \sum_{k,q} \binom{j-1}{-j+k+l+n-q+1} \binom{n-j}{q} w^q 
(h_{k+i-1-2n}+ [i\not=n+1]h_{k-i+1}) \right) \\
= \frac{1}{2} \det_{1 \le i,j \le n} 
\left( \sum_{k,q} \binom{j-1}{-j+k+l+n-q+1} \binom{n-j}{q} w^q 
(h_{k+i-1-n}+ h_{k-i+1-n}) \right).
\end{multline*} 
This concludes the proof of the lemma.
\end{proof}

The identity in the lemma involves, up to factors, a product of two determinants on the left-hand side and also a product of two determinants on the right-hand side. This suggests that each of the determinants on the left-hand side equals up to factors a determinant on the right-hand side. This is indeed the case. 

More specifically, one can show that \eqref{odd2nd} is up to factors equal to 
$$
(-1)^n \prod_{i=1}^n (1+X_i) X_i^l  \det_{1 \le i,j \le n} 
\left( \sum_{k,q} \binom{j-1}{-j+k+l+n-q+1} \binom{n-j}{q} w^q 
(h_{k+i-1-n}+ h_{k-i+1-n}) \right).
$$
This can be shown by induction with respect to $n$ as suggested in \cite[Remark~7.4]{nVASMDPP}. Thus Lemma~\ref{22} would actually not have been necessary, however, it explains how the expression was obtained much better than the proof by induction. 

Using Lemma~7.5 from \cite{nVASMDPP}, we can conclude further that the expression is equal to 
\begin{multline*} 
(-1)^n \prod_{i=1}^n (1+X_i) X_i^l  \\ \times 
\det_{1 \le i,j \le n} 
\left( \sum_{k,q} \binom{j-1}{-j+k+l+n-q+1} \binom{n-j}{q} w^q 
h_{k+i-1-n}(X_1,X_1^{-1},\ldots,X_{n-i+1},X_{n-i+1}^{-1}) \right).
\end{multline*} 
Also this formula can be proven directly by induction with respect to $n$.

Our next goal would be to give a combinatorial interpretation of
$$
\sum_{k,q} \binom{j-1}{-j+k+l+n-q+1} \binom{n-j}{q} w^q 
(h_{k+i-1-n}(X_1,X_1^{-1},\ldots,X_{n-i+1},X_{n-i+1}^{-1}).
$$
We replace $i$ by $n-i+1$ and $q$ by $n-j-q$, and then we get rid of $k$ by setting $p=k+l+q+1$.
$$
\sum_{p,q} \binom{j-1}{p} \binom{n-j}{q} w^{n-j-q}
h_{p-q-l-1-i}(X_1,X_1^{-1},\ldots,X_{i},X_{i}^{-1}).
$$
This is equal to \eqref{path} when taking into account the definition of complete symmetric functions $h_k$ for negative $k$'s as given in 
\eqref{reci}.






\section{Explicit product formulas in case $(X_1,\ldots,X_n)=(1,\dots,1)$ and $w=0,-1$}
\label{outlook}

When evaluating the specializations of the LHS or RHS of \eqref{rincluded} at $(X_1,\ldots,X_n)=(1,\ldots,1)$ in the cases $w=0,-1$ for small values of $n$, one observes that the numbers involve only small prime factors, and, therefore, it is likely that they are expressible by product formulas. (A similar observation is true for the case $w=1$, but there the explanation is simple, since $\prod_{1 \le i < j \le n} (1+ w X_i+X_j + X_i X_j)$ on the left-hand side of \eqref{rincluded} is symmetric then.) For the LHS and the case $m=n-1$, these are unpublished conjectures of Florian Schreier-Aigner from 2018. For instance, in the case $w=-1$ and $m=n-1$, we obtain the following numbers
$$
1,4,60,3328,678912,\ldots = 2^{n(n-1)/2} \prod_{j=0}^{n-1} \frac{(4j+2)!}{(n+2j+1)!}
$$
that have also appeared in recent work of Di Francesco \cite{twenty}.

Related conjectures for arbitrary $m$ have now been proven and will appear in forthcoming work with Florian Schreier-Aigner. The approach we have been successful with involves the transformation of the bialternant formula on the RHS of \eqref{rincluded} into a Jacobi-Trudi type determinant. Then we can 
easily set $X_i=1$ and we were able to guess the LU-decompositions of the relevant matrices. Proving these guesses involves the evaluation of certain triple sums, which is possible using Sister Celine's algorithm \cite{celine} and the fabulous Mathematica packages provided by RISC \cite{risc}. We state the results next. First we deal with the case $w=0$ .

\begin{theorem} 
The specialization of the generating function of arrowed Gelfand-Tsetlin patterns with $n$ rows and strictly increasing non-negative bottom row where the entries are bounded by $m$ at 
$(X_1,\ldots,X_n)=(1,\ldots,1)$, $u=v=t=1$ and $w=0$ is
equal to 
$$
3^{\binom{n+1}{2}} \prod_{i=1}^n \frac{(2n+m+2-3i)_i}{(i)_i}.
$$ 
\end{theorem}

Now we turn to the case $w=-1$.

\begin{theorem} 
The specialization of the generating function of arrowed Gelfand-Tsetlin patterns with $n$ rows and strictly increasing non-negative bottom row where the entries are bounded by $m$ at 
$(X_1,\ldots,X_n)=(1,\ldots,1)$, $u=v=t=1$ and $w=-1$ is
equal to 
$$
2^n \prod_{i=1}^n \frac{(m-n+3i+1)_{i-1} (m-n+i+1)_i}{\left(\frac{m-n+i+2}{2} \right)_{i-1} (i)_i}.
$$
\end{theorem}

Other future work concerns extending results that have been obtained using \eqref{littlewoodASM1} by replacing this identity by Theorem~\ref{boundedrQ} or specializations thereof.

\appendix

\section{Aspects of the combinatorics of the classical Littlewood identity and its bounded version}
\label{combclassical}

\subsection{The combinatorics of the classical Littlewood identity \eqref{littlewood}} 
We start by reviewing the classical combinatorial proof of  \eqref{littlewood}: one can interpret the Schur polynomial 
$s_{\lambda}(X_1,\ldots,X_n)$ as the (multivariate) generating function of semistandard Young tableaux of shape $\lambda$ with entries in $\{1,2,\ldots,n\}$, where the exponent of $X_i$ is just the number of occurrences of $i$ in a given semistandard Young tableau. Then the left-hand side of \eqref{littlewood} is simply the generating function of all such semistandard Young tableaux of any shape $\lambda$. The right-hand side can be interpreted as the generating function of symmetric $n \times n$ matrices with non-negative integer entries: expanding $\frac{1}{1-X_i X_j} = \sum_{a_{i,j} \ge 0} (X_i X_j)^{a_{i,j}}$ corresponds to the entries $a_{i,j}=a_{j,i}$ for $i<j$, while expanding $\frac{1}{1-X_i} = \sum_{a_{i,i} \ge 0} X_i^{a_{i,i}}$ corresponds to the diagonal entries $a_{i,i}$. Then such a matrix determines a two-line array with $a_{i,j}$ occurrences of the pair $\left( i \atop j \right)$ such that the pairs are ordered lexicographically. 
The semistandard Young tableau $P$ is simply obtained by applying the Robinson-Schensted-Knuth (RSK) algorithm to the bottom row of the two-line array. It suffices to construct the so-called insertion tableau because by the symmetry of the RSK algorithm, it is equal to the recording tableau. Thus, to reconstruct the two-line array, we apply the inverse Robinson-Schensted-Knuth algorithm to $(P,P)$. 

\subsection{Simpler decription of the classical bijection} 
Now we discuss a related, but simpler bijective proof of \eqref{littlewood} that does not invoke the symmetry of
the RSK algorithm. After its description, we will actually discover that ``only'' the description of the 
algorithm is simpler as we will show that the bijection agrees with the classical one. However, this second version could be of interest for developing the combinatorics of  \eqref{littlewoodASM2} and \eqref{rincluded}.

As discussed above, the right-hand side of \eqref{littlewood} can be interpreted as the generating function of symmetric $n \times n$ matrices $A=(a_{i,j})_{1 \le i,j \le n}$ with non-negative integer entries. They are also equivalent to lexicographically ordered two-line arrays with the property that the upper entry in each column is no smaller than the lower entry: For $i \le j$, let $a_{i,j}=a_{j,i}$ be the number of columns of type $\left( j \atop i \right)$. Comparing to the two-line array from the classical proof, we just have to delete all columns $\left( j \atop i \right)$ with $i > j$.

\bigskip

Now we apply the following variant of RSK, which transforms a lexicographically ordered two-line array such that no upper element is smaller than the corresponding lower element into a semistandard Young tableau.
\begin{itemize} 
\item As usual, we work through the columns of the two-line array from left to right. 
\item Suppose $\left( j \atop i \right)$, $i \le j$, is our current column. We use the usual RSK algorithm to insert $i$ in to the current tableau.
\item If $i < j$, we additionally place $j$ into the tableau as follows: Suppose that the insertion of $i$ ends with adding an entry to row $r$, then we add $j$ to row $r+1$ in the leftmost column where there is no entry so far. 
\end{itemize} 

\begin{example}
To give an example, observe that the symmetric matrix 
$$
A = \begin{pmatrix} 1 & 0 & 2 & 1 \\ 0 & 0 & 1 & 4 \\ 2 & 1 & 2 & 0 \\ 1 & 4 & 0 & 1 \end{pmatrix} 
$$
is equivalent to the following two-line array 
$$
\left( 
\begin{array}{cccccccccccc}
1 & 3 & 3 & 3 & 3 & 3 & 4 & 4 & 4 & 4 & 4 & 4 \\
1 & 1 & 1 & 2 & 3 & 3 & 1 & 2 & 2 & 2 & 2 & 4
\end{array} 
\right)
$$
and that the algorithm results in the following semistandard Young tableau.
$$
\begin{ytableau}
1 & 1  & 1 & 1 & 2 & 2 & 2 & 2 & 4 \\
2 & 3  & 3 & 3 & 3 & 4 & 4 \\
3 & 4 & 4 \\
4 
\end{ytableau}
$$
\end{example} 

\emph{Well-definedness of the algorithm.} 
We argue that the resulting tableau is always a semistandard Young tableau. For this, we need an observation that  
can be deduced from 
\cite[Lemma 7.11.2 (b)]{Sta99}, which says that if we insert a weakly increasing sequence of positive integers
$i_1 \le i_2 \le \ldots \le i_r$ from left to right into a semistandard Young tableau, then the ``insertion path'' 
of an earlier element lies strictly to the left of a later element. Moreover, for $p < q$, the insertion path of 
$i_p$ ends in a row below and to the left of the end of the insertion path of $i_q$, or in the same row to the left of the end of the insertion 
path of $i_q$. This implies that if the $i_k$'s are the bottom elements of the columns with top element $j$ in the two line array, then, if the insertion path of an $i_k$ with $i_k < j$ ends in row $r$, 
the elements in row $1,2,\ldots,r$ are in $\{1,2,\ldots,j-1\}$. 

We show by induction on the number of elements in the tableau that our algorithm always leads to a 
semistandard Young tableau.
Now, if we insert the element $i$ of the column $\left( j \atop i \right)$ using the classical RSK algorithm into the current semistandard Young tableau, then we obtain another semistandard Young tableau, see \cite[Lemma 7.11.3]{Sta99}.  Placing the top element $j$ in case  $j>i$ into the next row will also not destroy the columnstrictness as the elements above the row of $j$ are in $\{1,2,\ldots,j-1\}$, as discussed in the previous paragraph.

\begin{remark}
Note that from the proof of well-definedness it follows that we may also add all top $j$'s at once after we have inserted the bottom entries of columns that have $j$'s as top entries in our algorithm: Consider the skew shape $\lambda / \mu$ where $\mu$ is the shape of the tableau that we had before the insertion of all these bottom entries and $\lambda$  is the shape of the tableau we obtain after the insertion (but not yet adding the $j$'s from the top row of the two-line array) except that we exclude in the latter tableau all $j$'s that come from the bottom of the two-line array. Now, if there are $c$ cells in row $r$ of the skew shape then we add $c$ $j$'s in row $r+1$ to the semistandard Young tableaux with the bottom entries inserted, now including also those that come from columns $\left( j \atop j \right)$. This is because the cells of the skew shape are added to the tableau in the course of insertion from bottom to top and within a row from left to right.
\end{remark} 

\emph{Reverse algorithm.} 
We construct the inverse algorithm inductively, where the induction is with respect to the largest element in the tableau.
Suppose $n$ is the largest element in the semistandard Young tableau, then we want to recover the part of the two-line array that has $n$ in the top row (which is an ending section of the array). Suppose 
$$
\left( 
\begin{array}{cccc}
n & n & \ldots & n \\
i_1 & i_2 & \ldots & i_s
\end{array} 
\right)
$$
is this section, which implies $i_1 \le i_2 \le \ldots \le i_s$, and let $r$ be maximal with $i_r < n$ so that 
$i_{r+1} = i_{r+2}= \ldots = i_s = n$. Now, from the algorithm it follows that $s-r$ is just the number $n$'s in the top row of the tableau and we can delete these elements. Again it follows from \cite[Lemma 7.11.2 (b)]{Sta99} that we need to determine the number $u$ of $n$'s in the second row, remove them, and the apply the inverse bumping algorithm to the last $u$ element in the first row, from right to left (which means that we just remove them and put them in the bottom row of the two-line array). We continue by counting (and removing) the $n$'s in the third row, and, if $v$ is this number, apply the inverse bumping to the last $v$ elements in the second row, from right to left. We work through the rows from top to bottom in this way.

\medskip

Finally, we discover that this algorithm is just another description of the classical bijection.

\begin{prop}
The algorithm just described establishes the same bijection between symmetric $n \times n$ matrices $A$ with non-negative integer entries and semistandard Young tableaux with entries in $\{1,2,\ldots,n\}$ as the classical one.
\end{prop} 

\begin{proof}[Sketch of proof]
The proof is by induction with respect to $n$. For $n=1$, there is nothing to prove since the two algorithms coincide in this case.

We perform the step from $n-1$ to $n$. We can assume $a_{n,n}=0$ since increasing $a_{n,n}$ has the same effect in both algorithms as in both cases we just add $a_{n,n}$ columns $\left( n \atop n \right)$ at the end of the 
two-line arrays and apply the same procedure to these columns, in both cases at the end of the algorithm.

Suppose $B$ is the restriction of $A$ to the first $n-1$ rows and the first $n-1$ columns. By the induction hypothesis, we know that $B$ is transformed into the same semistandard Young tableau $P$ under both algorithms.
Moreover, let $a$ be the two-line array that corresponds to $A$ in the classical algorithm and $a'$ be the initial section that disregards all columns with an $n$ in the top row. Clearly, we can obtain $P$ also by applying RSK to the bottom row of $a'$ and then deleting all $n$'s because the two-line array $b$ that corresponds to $B$ under the classical algorithm is obtained from $a'$ by deleting all columns that have an $n$ in the bottom row and the $n$'s will never bump an element, but at most be bumped in final steps of insertions. Let $Q$ denote the semistandard Young tableau where the $n$'s are kept (i.e., what we obtain after applying RSK to the bottom row of $a'$). 

Now note that the final sections of the two-line array with $n$ in the top row agree for both two-line arrays, and denote it by $s$. 
Since we assume $a_{n,n}=0$, the bottom row of $s$ does not contain any $n$.
It is also clear that we will obtain the same tableau if we apply the following two different procedures: Insert the bottom row of $s$ to $P$, or, insert the bottom row of $s$ to $Q$ and then delete the $n$'s. This is because $P$ and $Q$ agree on all entries different from $n$ and  $n$'s are at most bumped in final steps in the second case.

This implies that the two procedures (namely, the ``classical'' one and the one that is the subject of this section) result in the same two tableaux when disregarding the $n$'s. Therefore, it remains to show that they also agree on the $n$'s. Now we use the fact that the positions of the $n$'s (as for any other entry) can also be determined by considering the recording tableau (which is due to the symmetry of the classical RSK algorithm), in particular we need to study how the recording tableau is built up 
when adding $s$ since this is the only time when $n$'s are added to the recording tableau. These $n$'s are added in the final cells of the insertion paths when inserting the bottom row of $s$ into $Q$. Such an insertion path can either agree with the corresponding insertion path in $P$ or it has one additional step where an $n$ gets bumped. As we already know that up to the $n$'s we obtain the same tableaux in both cases, we are always in the case that $n$'s are bumped and this proves the assertion.
\end{proof}

\subsection{RSK in terms of Gelfand-Tsetlin patterns} 
\label{gtpattern}

It is well-known that semistandard Young tableaux can be replaced by Gelfand-Tsetlin patterns 
in the definition of Schur polynomials (and thus in the combinatorial interpretation of the left-hand sides of \eqref{littlewood} and \eqref{littlewoodbounded}) as there is an easy bijective correspondence, which will be described next. This point of view is valuable for us because the left-hand sides of our Littlewood-type identities can also be interpreted combinatorially as generating functions of Gelfand-Tsetlin-pattern-type objects (see Section~\ref{LHS}). The purpose of the current section is to indicate how the classical RSK algorithm works on (classical) Gelfand-Tsetlin patterns, with the hope that something similar can be established for our variant (i.e., arrowed Gelfand-Tsetlin patterns, see Section~\ref{agtp}).

A Gelfand-Tsetlin pattern is a finite triangular array of integers with centered rows as follows
$$
\begin{array}{ccccccc}
&&& a_{1,1} &&& \\
&&a_{2,1} && a_{2,2} \\
& \iddots && \ldots && \ddots & \\
a_{n,1} && a_{n,2} && \ldots && a_{n,n} 
\end{array} 
$$
such that we have a weak increase in $\nearrow$-direction as well as in $\searrow$-direction, i.e., $a_{i+1,j} \le a_{i,j} \le a_{i+1,j+1}$, for all $1 \le j \le i \le n-1$. The bijection between semistandard Young tableaux of shape 
$(\lambda_1,\lambda_2,\ldots,\lambda_n)$ (we allow zero entries here) and parts in $\{1,2,\ldots,n\}$, and Gelfand-Tsetlin patterns with bottom row $(\lambda_n,\lambda_{n-1},\ldots,\lambda_1)$ is as follows: reading the $i$-th row of a Gelfand-Tsetlin pattern 
in reverse order gives a partition, and this is precisely the shape constituted by the entries less than or equal to $i$ in the corresponding semistandard Young tableau. Under this bijection, the number of entries equal to $i$ in the semistandard Young tableau is equal to the difference of the $i$-th row sum and the $(i-1)$-st row sum in the Gelfand-Tsetlin pattern. Therefore, 
$$
s_{(\lambda_1,\ldots,\lambda_n)}(X_1,\ldots,X_n) = \sum \prod_{i=1}^n X_i^{\sum_{j=1}^i a_{i,j} - \sum_{j=1}^{i-1} a_{i-1,j}}, 
$$
where the sum is over all Gelfand-Tsetlin patterns $(a_{i,j})_{1 \le j \le i \le n}$ with bottom row 
$(\lambda_n,\lambda_{n-1},\ldots,\lambda_1)$. 

To give an example, observe that the Gelfand-Tsetlin pattern corresponding to the following 
semistandard Young tableaux 
\begin{equation} 
\label{SYTex}
\begin{ytableau} 
1 & 1 & 1 & 2 & 2 & 3 & 5 \\
2 & 2 & 4 & 5 & 7 & 8 \\
4 & 5 & 5 & 7 & 8 \\
5 & 6 & 6 & 8 \\
7 & 8
\end{ytableau} 
\end{equation} 
is 
$$
\begin{array}{ccccccccccccccc}
&&&&&&& 3 &&&&&&& \\
&&&&&& 2 && 5 &&&&&& \\ 
&&&&& 0 && 2 && 6 &&&&& \\
&&&& 0 && 1 && 3 && 6 &&&& \\
&&& 0 && 1 && 3 && 4 && 7 &&& \\
&& 0 && 0 && 3 && 3 && 4 && 7 && \\
& 0 && 0 && 1 && 3 && 4 && 5 && 7 & \\
0 && 0 && 0 && 2 && 4 && 5 && 6 && 7 
\end{array}.
$$

Now suppose we use the RSK algorithm to insert the integer $m$ into a semistandard Young tableau. On the corresponding Gelfand-Tsetlin pattern, we have to do the following.
\begin{itemize} 
\item If the number $n$ of rows of the pattern is less than $m$ and the bottom row of the pattern is 
$k_1,\ldots,k_n$, then we add rows of the form $0,\ldots,0,k_1,\ldots,k_n$ with the appropriate number of $0$'s until we have $m$ rows.
\item Now we start a path in the pattern that starts at the last entry in row $m$ with (unit) steps in 
$\searrow$-direction or $\swarrow$-direction progressing from one entry to a neighboring entry in this direction. The rule is as follows: Whenever the $\searrow$-neighbor of the current entry is equal to the current entry we extend our path to the next entry 
in $\searrow$-direction, otherwise we go to the next entry in $\swarrow$-direction. We continue with this path until we reach the bottom row.
\item Finally, we add $1$ to all entries in the path.
\end{itemize}

To give an example, if we use RSK to insert $3$ into the semistandard Young tableau from \eqref{SYTex}, we obtain the following tableau, where the insertion path is indicated in red.
$$
\begin{ytableau} 
1 & 1 & 1 & 2 & 2 & 3 & \color{red} 3 \\
2 & 2 & 4 & 5 & \color{red} 5 & 8 \\
4 & 5 & 5 & 7 & \color{red} 7 \\
5 & 6 & 6 & 8 & \color{red} 8 \\
7 & 8
\end{ytableau} 
$$
On the corresponding Gelfand-Tsetlin pattern, we obtain the following.
$$
\begin{array}{ccccccccccccccc}
&&&&&&& 3 &&&&&&& \\
&&&&&& 2 && 5 &&&&&& \\ 
&&&&& 0 && 2 && \color{red} 7  &&&&& \\
&&&& 0 && 1 && 3 && \color{red} 7 &&&& \\
&&& 0 && 1 && 3 && \color{red} 5 && 7 &&& \\
&& 0 && 0 && 3 && 3 && \color{red} 5 && 7 && \\
& 0 && 0 && 1 && 3 && \color{red} 5  && 5 && 7 & \\
0 && 0 && 0 && 2 && \color{red} 5 && 5 && 6 && 7 
\end{array}
$$
It corresponds to the tableau with the $3$ inserted.

Now suppose in our simplified algorithm to prove \eqref{littlewood}, we ``insert'' the column  $\left( j \atop i \right)$ into the Gelfand-Tsetlin pattern. At this point, the Gelfand-Tsetlin pattern should have $j$ rows. Then we apply the algorithm just described to insert $i$ into the pattern. To insert also $j$ (in case $j \not=i$), add $1$ to the entry immediately left of the entry that is the end of the path that is induced by the insertion of $i$. Whenever we progress to the first column 
with $j$ as top element in the two-line array, we add one row to the Gelfand-Tsetlin by copying the current bottom row and adding one $0$ at the beginning.

\subsection{The right-hand side of the bounded Littlewood identity \eqref{littlewoodbounded}}

The irreducible characters of the special orthogonal group $SO_{2n+1}(\mathbb{C})$ associated with the 
partition $\lambda=(\lambda_1,\ldots,\lambda_n)$ are 
$$
so^{\text{odd}}_{\lambda}(X_1,\ldots,X_n)=\prod_{i=1}^{n} X_i^{n-1/2}
\frac{\det_{1 \le i, j \le n} \left( X_i^{-\lambda_j-n+j-1/2} - X_i^{\lambda_j+n-j+1/2} \right)}
{(1 + [\lambda_n=0])\prod_{i=1}^n (1-X_i) \prod_{1 \le i<j \le n} (X_j-X_i)(1-X_i X_j)}, 
$$
see \cite[Eq. (24.28)]{FulHar91}.
These characters can be seen as generating functions of certain halved Gelfand-Tsetlin patterns that are defined next. This can even be extended to so-called half-integer partitions as will be explained also. A half-integer partition is a finite, weakly decreasing sequence of positive half-integers. 

\begin{definition} 
For a positive integer $n$, a $2n$-split orthogonal (Gelfand-Tsetlin) pattern is an array of non-negative integers or non-negative half-integers with $2n$ rows of lengths $1,1,2,2,\ldots,n,n$, which are aligned as follows for $n=3$
$$
\begin{array}{cccccc}
a_{1,1}  &              &              &            &             & \\
             & a_{2,1}  &              &            &             & \\
a_{3,1} &               & a_{3,2} &             &             & \\
             & a_{4,1} &              & a_{4,2} &             & \\
a_{5,1} &              & a_{5,2}  &             & a_{5,3} & \\
             & a_{6,1} &              & a_{6,2} &              & a_{6,3} 
\end{array} ,
$$
such that the entries are weakly increasing along $\nearrow$-diagonals and $\searrow$-diagonals, and in which the entries, except for the first entries
in the odd rows (called odd starters), are either all non-negative integers or all non-negative half-integers. Each starter is independently either a 
non-negative integer or a non-negative half-integer. The weight of a $2n$-split orthogonal pattern is 
$$
\prod_{i=1}^n X_i^{r_{2i}-2 r_{2i-1}+r_{2i-2}}
$$
where $r_i$ is the sum of entries in row $i$ and $r_0=0$. 
\end{definition} 

The following theorem is the first part of Theorem~7.1 in 
\cite{Pro94}.

\begin{theorem} 
\label{split}
Let $\lambda=(\lambda_1,\ldots,\lambda_n)$ be a partition (allowing zero entries) or a half-integer partition. Then the generating function of $2n$-split orthogonal patterns with respect to the above weight that have $\lambda$ as bottom row, written in increasing order,  
is 
$$
\prod_{i=1}^{n} X_i^{n-1/2}
\frac{\det_{1 \le i, j \le n} \left( X_i^{-\lambda_j-n+j-1/2} - X_i^{\lambda_j+n-j+1/2} \right)}
{(1 + [\lambda_n=0])\prod_{i=1}^n (1-X_i) \prod_{1 \le i<j \le n} (X_j-X_i)(1-X_i X_j)}.
$$
\end{theorem}

Now the right-hand side of \eqref{littlewoodbounded} can be written as
\begin{multline*} 
\frac{ \det_{1 \le i, j \le n} \left( X_i^{j-1} - X_i^{m+2n-j}  \right)
}{\prod_{i=1}^n (1-X_i) \prod_{1 \le i < j \le n} (X_j-X_i)(1-X_i X_j)} \\ 
=
\prod_{i=1}^n X_i^{(m-1)/2+n} \frac{ \det_{1 \le i, j \le n} \left( X_i^{j-n-(m+1)/2} - X_i^{-j+n+(m+1)/2}  \right)
}{\prod_{i=1}^n (1-X_i) \prod_{1 \le i < j \le n} (X_j-X_i)(1-X_i X_j)}, 
\end{multline*} 
so that we can deduce from Theorem~\ref{split} that it is equal to 
$$\prod_{i=1}^n X_i^{m/2} so^{\text{odd}}_{(m/2,m/2,\ldots,m/2)}(X_1,\ldots,X_n).$$
From \eqref{littlewoodbounded}, it now follows that 
\begin{equation} 
\label{orthogonal}
\sum_{\lambda \subseteq (m^n)} s_{\lambda}(X_1,\ldots,X_n) = 
\prod_{i=1}^n X_i^{m/2} so^{\text{odd}}_{(m/2,m/2,\ldots,m/2)}(X_1,\ldots,X_n).
\end{equation}
A combinatorial proof of this fact can be found in \cite[Corollary 7.4]{Ste90}.

It would be interesting to see whether there is a bijective proof of \eqref{orthogonal} that uses RSK. More concretely, under the bijection that is used in the classical bijective proof of \eqref{littlewood} semistandard Young tableaux whose shape is in $(m^n)$ correspond to two-line arrays such that the longest increasing subsequence of the bottom row has at most $m$ elements, see \cite[Proposition 7.23.10]{Sta99}.

Next we argue that we can also read off the $m$ from the two-line array we use for our simplified proof of \eqref{littlewood}, in the following sense. The longest increasing subsequence of the bottom row of the ``classical'' two-line array can be read off the corresponding matrix $A$ with non-negative integers 
as follows: we consider walks through the matrix with unit $\rightarrow$-steps and unit $\downarrow$-steps and add up the entries we traverse. The maximal sum we can achieve with such a path is the length of the longest increasing subsequence of the bottom row of the classical two-line array. Now, if the matrix $A$ is symmetric, we can confine such walks to be weakly above the main diagonal and the two-line array of the simplified algorithm is constituted by this part of the matrix.

Finally, we give a bijective proof of \eqref{orthogonal} in the case $n=2$. The left hand side can be seen as the generating function of semistandard Young tableaux with entries in $\{1,2\}$, with the weight 
$$
X_1^{\# \text{ of $1$'s}} X_2^{\# \text{ of $2$'s}}.
$$
Such tableaux have at most $2$ rows and can be encoded by three non-negative integers $x,y,z$: let $y$ be the number of $2$'s in the 
second row, $z$ be the number of $2$'s in the first row and $x+y$ be the number of $1$'s, which are necessarily in the first row. The two-line array that corresponds to such a tableau under our simplified algorithm is constituted by $x$ columns $\left( 1 \atop 1 \right)$, 
$y$ columns $\left( 2 \atop 1 \right)$ and $z$ columns $\left( 2 \atop 2 \right)$, ordered lexicograhpically. The corresponding $4$-split pattern can be obtained as follows: Add $\frac{x+y+z}{2}$ to all entries of the following $4$-split pattern.
$$
\begin{array}{cccc}
\frac{-x-y-\min(x,z)}{2}  &              &              &             \\
             & - \min(x,z)  &              &            \\
\frac{-y-z-\min(x,z)}{2} &               & 0 &              \\
             & 0 &              & \phantom{1234} 0 
\end{array}
$$

\section{Further combinatorial interpretations of the left-hand sides} 
\label{furtherLHS} 

\subsection{Generating function of AGTPs with respect to the bottom row} 

Setting $X_1=X_2=\ldots=X_n=1$ in Theorem~\ref{robbins}, we see that the generating function of AGTPs with bottom row $k_1,\ldots,k_n$ and with respect to the weight 
\begin{equation}
\label{weight2} 
\sgn(A) t^{\emptyset} u^{\nearrow} v^{\nwarrow} w^{\nenwarrow} 
\end{equation} 
is 
\begin{multline} 
\label{gfun}
(t+u+v+w)^n  
\prod_{1 \le i < j \le n} \left( t  + u  \e_{k_i} + v  \e_{k_j}^{-1} + w \e_{k_i} \e_{k_j}^{-1} \right) 
\prod_{1 \le i < j \le n} \frac{k_j-k_i+j-i}{j-i} \\ 
=(t+u+v+w)^n  
\prod_{1 \le i < j \le n} \left( t \e_{k_j}  + u  \e_{k_i} \e_{k_j} + v  + w \e_{k_i}  \right) 
\prod_{j=1}^n \e_{k_j}^{-j+1} \prod_{1 \le i < j \le n} \frac{k_j-k_i+j-i}{j-i} \\
=(t+u+v+w)^n  
\prod_{1 \le i < j \le n} \left( t \e_{k_j}  + u  \e_{k_i} \e_{k_j} + v  + w \e_{k_i}  \right) 
 \prod_{1 \le i < j \le n} \frac{k_j-k_i}{j-i}, 
\end{multline}
using the fact $s_{(k_n,k_{n-1},\ldots,k_1)}(1,\ldots,1) = \prod_{1 \le i < j \le n} \frac{k_j-k_i+j-i}{j-i}$, 
which follows from \cite[(7.105)]{Sta99} when taking the limit $q \to 1$.  

Generalizing a computation in Section~6 of \cite{halved2009} slightly, it can be seen that the coefficient of
$X_1^{k_1} X_2^{k_2} \cdots X_n^{k_n}$ in 
$$
(t+u+v+w)^n \prod_{i=1}^n X_i^{-n+1} (1-X_i)^{-n} \prod_{1 \le i < j \le n} (X_j-X_i)(u+ t X_i + w X_j + v X_i X_j) 
$$
is the generating function of AGTPs with bottom row $k_1,k_2,\ldots,k_n$ as given in \eqref{gfun}, when interpreting the rational function 
as a formal Laurent series in $X_1,X_2,\ldots,X_n$ with $(1-X_i)^{-1}= \sum_{k \ge 0} X_i^k$ and assuming 
$(k_1,k_2,\ldots,k_n) \ge 0$. Phrased differently, for any $(k_1,\ldots,k_n),(m_1,\ldots,m_n) \in \mathbb{Z}^n$ 
with $(k_1+m_1,\ldots,k_n+m_n) \ge 0$, the coefficient of $X_1^{m_1} \cdots X_n^{m_n}$ in 
$$
(t+u+v+w)^n \prod_{i=1}^n X_i^{-n+1-k_i} (1-X_i)^{-n} \prod_{1 \le i < j \le n} (X_j-X_i)(u+ t X_i + w X_j + v X_i X_j) 
$$
is the generating function of AGTPs with bottom $(k_1+m_1,\ldots,k_n+m_n)$. 
Therefore, the coefficient of $X_1^{m_1} \cdots X_n^{m_n}$ in
\begin{multline*} 
(t+u+v+w)^n \\ 
\times \sym_{X_1,\ldots,X_n} \left[ \prod_{i=1}^n X_i^{-n+1-k_i} (1-X_i)^{-n} \prod_{1 \le i < j \le n} (X_j-X_i)(u+ t X_i + w X_j + v X_i X_j) \right] 
\end{multline*} 
is the generating function of pairs of AGTPs and permutations $\sigma$, where the difference of the bottom row and $(k_1,\ldots,k_n)$ 
is the permutation of $\{m_{1},\ldots,m_{n}\}$ given by $\sigma$, assuming $(k_1+m_{\sigma(1)},\ldots,k_n+m_{\sigma(n)}) \ge 0$ for every permutation $\sigma$. The latter is always satisfied if $(k_1,\ldots,k_n),(m_1,\ldots,m_n) \ge 0$. The above expression is equal to 
\begin{multline*}
(t+u+v+w)^n \prod_{i=1}^n (1-X_i)^{-n} \prod_{1 \le i < j \le n} (X_j-X_i) \\
\times \asym_{X_1,\ldots,X_n} \left[ \prod_{i=1}^n X_i^{-k_i}  \prod_{1 \le i < j \le n} (v+ w X_i^{-1} + t X_j^{-1} + u X_i^{-1} X_j^{-1}) \right]. 
\end{multline*}
We sum over all $0 \le k_1 < k_2 < \ldots < k_n \le m$.
\begin{multline*} 
(t+u+v+w)^n \prod_{i=1}^n (1-X_i)^{-n} \prod_{1 \le i < j \le n} (X_j-X_i) \\
\times  \asym_{X_1,\ldots,X_n} \left[ \prod_{1 \le i < j \le n} (v+ w X_i^{-1} + t X_j^{-1} + u X_i^{-1} X_j^{-1})
\sum_{0 \le k_1 < k_2 < \ldots < k_n \le m} X_1^{-k_1} X_2^{-k_2} \cdots X_n^{-k_n}  \right] 
\end{multline*} 
For $(m_1,\ldots,m_n) \ge 0$, the coefficient of $X_1^{m_1} \cdots X_n^{m_n}$ is the generating function of pairs of AGTPs $A$ and permutations $\sigma$ of $\{1,2,\ldots,n\}$ such that, if $(m_{\sigma(1)},\ldots,m_{\sigma(n)})$ is added to the bottom row of $A$, we obtain a strictly increasing sequence of non-negative integers. In particular, the constant term is the generating function of AGTPs (with respect to the weight \eqref{weight2}), whose bottom row is a strictly increasing sequence of non-negative integers, multiplied by $n!$. 
Setting $t=u=v=1$, this is by \eqref{rincluded} equal to 
\begin{multline*}
(3+w)^n \prod_{i=1}^n (1-X_i)^{-n} \prod_{1 \le i < j \le n} (X_j-X_i) \\ \times 
\frac{\det_{1 \le i, j \le n} \left( X_i^{-j+1}  (1+X_i^{-1})^{j-1} (1+ w X_i^{-1})^{n-j}  - X_i^{-m-2n+j} (1+X_i)^{j-1}  (1+w X_i)^{n-j} \right)}{\prod\limits_{i=1}^n (1-X_i^{-1}) \prod\limits_{1 \le i < j \le n} (1-X_i^{-1} X_j^{-1})} \\
= (3+w)^n \prod_{i=1}^n (1-X_i)^{-n} \prod_{1 \le i < j \le n} (X_j-X_i) \\ \times
\frac{\det_{1 \le i, j \le n} \left( X_i^{-j+2}  (1+X_i)^{j-1} (w+ X_i)^{n-j}  - X_i^{-m-n+j} (1+X_i)^{j-1}  (1+w X_i)^{n-j} \right)}{\prod\limits_{i=1}^n (X_i-1) \prod\limits_{1 \le i < j \le n} (X_i X_j-1)}.
\end{multline*} 

\subsection{Generating function of alternating sign triangles with respect to the positions of the $1$-columns} 

Alternating sign triangles have been introduced recently in \cite{extreme}. 

\begin{definition} An alternating sign triangle (AST) with $n \ge 1$ rows is a triangular array with $n$ centered rows of the following shape 
$$
\begin{array}{ccccccc} 
a_{1,1} & a_{1,2} & \ldots & \ldots & \ldots & \ldots & a_{1,2n-1} \\
             & a_{2,2} & \ldots & \ldots & \ldots & a_{2,2n-2} &   \\
             &              & \ldots & \ldots & \ldots &                &   \\
             &              &           & a_{n,n}           &                & 
\end{array}
$$
such that $a_{i,j} \in \{0,1,-1\}$, non-zero entries alternate in each row and column, all rows sum to $1$ and the topmost non-zero entry (if any) in each column is $1$.
\end{definition} 

Next we give an example of an AST with $5$ rows.
$$
\begin{array}{ccccccc} 
0 & 0 & 0 & 1 & 0 & 0 & 0 \\
   & 0 & 1 & -1 &0 & 1 &   \\
   &    & 0 & 0 & 1 &    &   \\
   &    &    & 1 &    &    &
\end{array}
$$

It is known that there is the same number of $n \times n$ ASMs as there is of ASTs with $n$ rows, but no bijection is known so far. It has even been possible to identify certain equidistributed statistics, see \cite{Fis19a,Fis19b,extreme}. 

The columns of an AST sum to $0$ or $1$. A column that sums to $1$ is said to be a $1$-column.
The central column is always a $1$-column. Since the sum of all entries in an AST with $n$ rows is $n$, there are precisely $n-1$ other $1$-columns. A certain type of generating function with respect to the $1$-columns has been derived in 
\cite[Theorem~7]{Fis19a}. It involves one other statistic, which we introduce next: A $11$-column is a $1$-column with $1$ as bottom element, while a $10$-column is a $1$-column with $0$ as bottom element. For an AST, $T$ we define 
\begin{multline*} 
\rho(T)= \# \text{$11$-columns left of the central column} \\
+ \# \text{$10$-columns right of the central column} + 1 
\end{multline*} 

\begin{theorem}
\label{ASTtheorem}
Let $n$ be a positive integer, $0 \le r \le n-1$ and $0 \le j_1 < j_2 < \ldots < j_{n-1} \le 2n-3$. The coefficient of 
$t^{r-1} X_1^{j_1} X_2^{j_2} \cdots X_{n-1}^{j_{n-1}}$ in 
\begin{equation}
\label{ASTgfun} 
\prod_{i=1}^{n-1} (t+X_i) \prod_{1 \le i < j \le n-1} (1+X_i + X_i X_j)(X_j-X_i)
\end{equation} 
is the number of ASTs $T$ with $n$ rows, $\rho(T)=r$  and $1$-columns in positions $j_1,j_2,\ldots,j_{n-1}$, where we exclude the central column and count from the left starting with $0$.  
\end{theorem}
For what follows, the crucial question is whether we can give the coefficient of $t^{r-1} X_1^{j_1} X_2^{j_2} \cdots X_{n-1}^{j_{n-1}}$ of 
\eqref{ASTgfun} also a meaning if $(j_1,\ldots,j_{n-1})$ is not strictly increasing. Such an interpretation does not exist so far.

Phrased differently, the theorem states that the coefficient of $X_1^{m_1} X_2^{m_2} \cdots X_{n-1}^{m_{n-1}}$ in 
\begin{multline*} 
\prod_{i=1}^{n-1} (t+X_i^{-1}) X_i^{j_i} \prod_{1 \le i < j \le n-1} (1+X_i^{-1} + X_i^{-1} X_j^{-1})(X_j^{-1}-X_i^{-1}) \\
= \prod_{i=1}^{n-1} (1 + t X_i) X_i^{j_i-2n+1} \prod_{1 \le i < j \le n-1} (1+X_j + X_i X_j)(X_i-X_j)
\end{multline*} 
is the generating function of ASTs with $1$-columns in positions $j_1-m_1,j_2-m_2,\ldots,j_{n-1}-m_{n-1}$ with respect to $\rho(T)-1$, provided that 
$j_1-m_1 < j_2 - m_2 < \cdots < j_{n-1}-m_{n-1}$. Therefore, the coefficient of $X_1^{m_1} X_2^{m_2} \cdots X_{n-1}^{m_{n-1}}$ in 
$$
\sym_{X_1,\ldots,X_{n-1}} \left[ \prod_{i=1}^{n-1} (1 + t X_i) X_i^{j_i-2n+1} \prod_{1 \le i < j \le n-1} (1+X_j + X_i X_j)(X_i-X_j) \right] 
$$
is the generating function of pairs of ASTs and permutations of $\{1,2,\ldots,n-1\}$, such that 
$j_1-m_{\sigma(1)},j_2-m_{\sigma(2)},\ldots,j_{n-1}-m_{\sigma(n-1)}$ are the positions of $1$-columns, provided that 
$(j_1-m_{\sigma(1)},j_2-m_{\sigma(2)},\ldots,j_{n-1}-m_{\sigma(n-1)})$ is strictly increasing for all $\sigma$. Note that it is possible to 
satisfy the strictly increasing condition, for instance if $(j_1,\ldots,j_{n-1})$ is strictly increasing and the differences between consecutive $j_l$ are large while the $m_l$ are small.

The expression is equal to 
$$
\prod_{i=1}^{n-1} (1+t X_i) X_i^{-2n+1} \prod_{1 \le i < j \le n-1} (X_i-X_j) \, \asym_{X_1,\ldots,X_{n-1}} \left[ \prod_{i=1}^{n-1} X_i^{j_i} \prod_{1 \le i < j \le n-1} (1+X_j + X_i X_j) \right].
$$ 
We sum over all $p \le j_1 < j_2 < \ldots < j_{n-1} \le q$. 
\begin{multline}
\label{ASTlast}  
\prod_{i=1}^{n-1} (1+t X_i) X_i^{-2n+1+p} \prod_{1 \le i < j \le n-1} (X_i-X_j) \\ 
\times \asym_{X_1,\ldots,X_{n-1}} \left[  \prod_{1 \le i < j \le n-1} (1+X_j + X_i X_j) \sum_{0 \le j_1 < j_2 < \cdots < j_{n-1} \le q-p} X_1^{j_1} \cdots X_{n-1}^{j_{n-1}} \right]
\end{multline}
Now, the coefficient of $X_1^{m_1} X_2^{m_2} \cdots X_{n-1}^{m_{n-1}}$ in this expression is the generating function of pairs of, let us say, \emph{extended} ASTs and permutations of $\{1,2,\ldots,n-1\}$ such that if $m_{\sigma(1)},\ldots,m_{\sigma(n-1)}$ is added to the positions of the $1$-columns, we obtain a strictly increasing sequence of integers between $p$ and $q$. Extended refers to the fact that we now would need an extended version of Theorem~\ref{ASTtheorem} as indicated above, as we cannot guarantee that 
$(j_1-m_{\sigma(1)},j_2-m_{\sigma(2)},\ldots,j_{n-1}-m_{\sigma(n-1)})$ are strictly increasing when we sum over all $p \le j_1 < j_2 < \ldots < j_{n-1} \le q$. 

An exception in this respect is the case when all $m_l=0$. It follows that the constant term of \eqref{ASTlast} is the generating function of ASTs with $n$ rows whose $1$-columns are between $p$ and $q$.  
Using \eqref{rincluded}, this is equal to 
\begin{multline*} 
\prod_{i=1}^{n-1} \frac{(1+t X_i) X_i^{-2n+1+p}}{1-X_i} \prod_{1 \le i < j \le n-1} \frac{X_i-X_j}{1-X_i X_j} 
\det_{1 \le i, j \le n-1} \left( X_i^{j-1}  (1+X_i)^{j-1}  - X_i^{q-p+2n-2j+1} (1+X_i)^{j-1}  \right).
\end{multline*}


\end{document}